\documentclass[12pt]{amsart}
\usepackage{amssymb}
\usepackage{verbatim}
\usepackage{hyperref}
\usepackage{color}

\begin{document}

% define theorem environments
\newtheorem{theorem}{Theorem}    %[section]
\newtheorem{proposition}[theorem]{Proposition}
\newtheorem{conjecture}[theorem]{Conjecture}
\def\theconjecture{\unskip}
\newtheorem{corollary}[theorem]{Corollary}
\newtheorem{lemma}[theorem]{Lemma}
\newtheorem{sublemma}[theorem]{Sublemma}
\newtheorem{observation}[theorem]{Observation}
\newtheorem{remark}[theorem]{Remark}
\newtheorem{definition}[theorem]{Definition}
\theoremstyle{definition}
\newtheorem{notation}[theorem]{Notation}
\newtheorem{question}[theorem]{Question}
\newtheorem{questions}[theorem]{Questions}
\newtheorem{example}[theorem]{Example}
\newtheorem{problem}[theorem]{Problem}
\newtheorem{exercise}[theorem]{Exercise}

\numberwithin{theorem}{section} \numberwithin{theorem}{section}
\numberwithin{equation}{section}

\def\earrow{{\mathbf e}}
\def\rarrow{{\mathbf r}}
\def\uarrow{{\mathbf u}}
\def\varrow{{\mathbf V}}
\def\tpar{T_{\rm par}}
\def\apar{A_{\rm par}}

\def\reals{{\mathbb R}}
\def\torus{{\mathbb T}}
\def\heis{{\mathbb H}}
\def\integers{{\mathbb Z}}
\def\naturals{{\mathbb N}}
\def\complex{{\mathbb C}\/}
\def\distance{\operatorname{distance}\,}
\def\support{\operatorname{support}\,}
\def\dist{\operatorname{dist}\,}
\def\Span{\operatorname{span}\,}
\def\degree{\operatorname{degree}\,}
\def\kernel{\operatorname{kernel}\,}
\def\dim{\operatorname{dim}\,}
\def\codim{\operatorname{codim}}
\def\trace{\operatorname{trace\,}}
\def\Span{\operatorname{span}\,}
\def\dimension{\operatorname{dimension}\,}
\def\codimension{\operatorname{codimension}\,}
\def\nullspace{\scriptk}
\def\kernel{\operatorname{Ker}}
\def\ZZ{ {\mathbb Z} }
\def\p{\partial}
\def\rp{{ ^{-1} }}
\def\Re{\operatorname{Re\,} }
\def\Im{\operatorname{Im\,} }
\def\ov{\overline}
\def\eps{\varepsilon}
\def\lt{L^2}
\def\diver{\operatorname{div}}
\def\curl{\operatorname{curl}}
\def\etta{\eta}
\newcommand{\norm}[1]{ \|  #1 \|}
\def\expect{\mathbb E}
\def\bull{$\bullet$\ }
\def\C{\mathbb{C}}
\def\R{\mathbb{R}}
\def\Rn{{\mathbb{R}^n}}
\def\Sn{{{S}^{n-1}}}
\def\M{\mathbb{M}}
\def\N{\mathbb{N}}
\def\Q{{\mathbb{Q}}}
\def\Z{\mathbb{Z}}
\def\F{\mathcal{F}}
\def\L{\mathcal{L}}
\def\S{\mathcal{S}}
\def\supp{\operatorname{supp}}
\def\dist{\operatorname{dist}}
\def\essi{\operatornamewithlimits{ess\,inf}}
\def\esss{\operatornamewithlimits{ess\,sup}}
\def\xone{x_1}
\def\xtwo{x_2}
\def\xq{x_2+x_1^2}
\newcommand{\abr}[1]{ \langle  #1 \rangle}

\newcommand{\Norm}[1]{ \left\|  #1 \right\| }
\newcommand{\set}[1]{ \left\{ #1 \right\} }
\def\one{\mathbf 1}
\def\whole{\mathbf V}
\newcommand{\modulo}[2]{[#1]_{#2}}
\renewcommand{\thefootnote}{\fnsymbol{footnote}}
\def\scriptf{{\mathcal F}}
\def\scriptg{{\mathcal G}}
\def\scriptm{{\mathcal M}}
\def\scriptb{{\mathcal B}}
\def\scriptc{{\mathcal C}}
\def\scriptt{{\mathcal T}}
\def\scripti{{\mathcal I}}
\def\scripte{{\mathcal E}}
\def\scriptv{{\mathcal V}}
\def\scriptw{{\mathcal W}}
\def\scriptu{{\mathcal U}}
\def\scriptS{{\mathcal S}}
\def\scripta{{\mathcal A}}
\def\scriptr{{\mathcal R}}
\def\scripto{{\mathcal O}}
\def\scripth{{\mathcal H}}
\def\scriptd{{\mathcal D}}
\def\scriptl{{\mathcal L}}
\def\scriptn{{\mathcal N}}
\def\scriptp{{\mathcal P}}
\def\scriptk{{\mathcal K}}
\def\frakv{{\mathfrak V}}
\allowdisplaybreaks
\begin{comment}
\def\scriptx{{\mathcal X}}
\def\scriptj{{\mathcal J}}
\def\scriptr{{\mathcal R}}
\def\scriptS{{\mathcal S}}
\def\scripta{{\mathcal A}}
\def\scriptk{{\mathcal K}}
\def\scriptp{{\mathcal P}}
\def\frakg{{\mathfrak g}}
\def\frakG{{\mathfrak G}}
\def\boldn{\mathbf N}
\end{comment}
%
%If a theorem-like environment should not be numbered,
%add * after \newtheorem, and delete the counter option such as [theorem].
\newtheorem*{remark0}{\indent\sc Remark}
%
%%%%% Proof %%%%%
\renewcommand{\proofname}{\indent\sc Proof.} % Name is small caps
%The following commands are available in the proof environment:
%\begin{proof}
%\end{proof}
%The end of a proof is marked with a square.
%%%%%%%%%%%%%%%%%%%%%%%%%%%%%%%%%%%%%%%%%

\title[New weighted inequalities in the Bessel setting and applications]
{Weighted norm inequalities for operators with new Muckenhoupt type weights in the Bessel setting and applications}
%{Weighted inequalities for variation of heat semigroups associated to Schr\"{o}dinger operators with weights related to critical radius functions}
%title of paper and the running head option
\author{Yongming Wen$^{\ast}$ and Huoxiong Wu}
%%%%%%%%%%%%%%% footnote %%%%%%%%%%%%%%%%
\footnote[0]{{2020} \textit{Mathematics Subject Classification}.
42B20; 42B25; 42B35.
}
%In case \subjclass[2000] command is not effective
%(or the version of amsart.cls is old), write as follows instead:
%\renewcommand{\thefootnote}{\fnsymbol{footnote}}
%\footnote[0]{2000\textit{ Mathematics Subject Classification}.

%Primary 00; Secondary 00.}
%
\keywords{$\widetilde{A}_{p,\lambda}$ weights, Bessel operators, sparse operators, quantitative weighted bounds, restricted weak type inequalities, extrapolation theorems}
\thanks{$^*$Corresponding author.}
\thanks{Yongming Wen is supported by National Natural Science Foundation of China (Grant no. 12301119), Fujian Provincial Natural Science Foundation of China (Nos. 2025J01361), President's fund of Minnan Normal University (No. KJ2020020), Institute of Meteorological Big Data-Digital Fujian, Fujian Key Laboratory of Data Science and Statistics and Fujian Key Laboratory of Granular Computing and Applications (Minnan Normal University), China. Huoxiong Wu is supported by National Natural Science Foundation of China (Grant nos. 12271041, 12671114).}
%%%%%%%%%%%% Authors addresses %%%%%%%%%%%%%
\address{School of Mathematics and Statistics, Minnan Normal University, Zhangzhou 363000, China} \email{wenyongmingxmu@163.com}
\address{School of Mathematical Sciences, Xiamen University, Xiamen 361005, China} \email{huoxwu@xmu.edu.cn}

%\address{}
%\email{}

%%%%%%%%%%%%%%%%%%%%%%%%%%%%%%%%%%%%%%%%%

\begin{abstract}
Let $\lambda>-1/2,\neq0$. Consider the Bessel operators
\begin{align*}
\Delta_\lambda:=-\frac{d^2}{dx^2}-\frac{2\lambda}{x}\frac{d}{dx}
\end{align*}
on $\mathbb{R}_+$. In this paper, we first establish several characterizations for $\widetilde{A}_{\infty,\lambda}$ weights introduced by Li et al. (Math. Z. 309(1) (2025), Paper No. 13, 29pp), which are different from the previous $A_p$ class with respect to the measure $x^{2\lambda}dx$. Among these characterizations, we obtain the sharp reverse H\"{o}lder inequality (RHI) for $\widetilde{A}_{\infty,\lambda}$ weights. As applications of RHI, we derive a precise open property of $\widetilde{A}_{p,\lambda}$ weights, and obtain a sharp weighted bounds for the Hardy-Littlewood maximal operator $\mathcal{M}_\lambda$ associated to $\Delta_\lambda$, which was missing in the previous papers. Next, we also establish several extrapolation theorems related to $\widetilde{A}_{p,\lambda}$ weights. Furthermore, we characterize the restricted weighted weak type estimates for $\mathcal{M}_\lambda$ via a new class of weights $\widetilde{A}_{p,\lambda}^{\mathcal{R}}$, which is strictly larger than $\widetilde{A}_{p,\lambda}$. Meanwhile, we obtain quantitative restricted weighted weak type $(p,p)$ bounds for sparse operators. As applications, we derive various kinds of weighted inequalities for Riesz transforms, maximal operators, square functions, variation operators and their commutators associated to $\Delta_\lambda$.
\end{abstract}

\maketitle

\section{Introduction and main results}
The Laplacian operator $\Delta$ is intrinsically tied to the classical theory of weights. It is known that the Muckenhoupt $A_p$ weight class can be characterised in terms of weighted $L^p$ boundedness of Hardy-Littlewood maximal functions and Riesz transforms. However, this fact does not always remain valid in the setting of other differential operators. For instance, in the Schr\"{o}dinger setting, Bongioanni, Cabral and Harboure \cite{BoCaHa} proved that the weighted $L^p$ boundedness of maximal operators associated with Schr\"{o}dinger operators can be characterized by a class of $A_p$ weights related to critical radius functions, while there exist weights $w$ not contained in this class of weights, for which the Riesz transforms associated with Schr\"{o}dinger operators are bounded on $L^p(\mathbb{R}^n,wdx)$, see \cite{Ba}. The same phenomenon also occurs in the setting of Bessel operator. This is precisely the differential operator of interest in this paper.

Let $\lambda>-1/2,\neq0$. The Bessel operators
\begin{align*}
\Delta_\lambda:=-\frac{d^2}{dx^2}-\frac{2\lambda}{x}\frac{d}{dx},
\end{align*}
emerged within the fields of complex analysis and PDEs. Muckenhoupt and Stein \cite{MuSt} established the $L^p(\mathbb{R}_+,x^{2\lambda}dx)$-boundedness of Bessel Riesz transforms $R_{\Delta_\lambda}:=\frac{d}{dx}(\Delta_\lambda)^{-1/2}$, conjugate function and fractional integrals. Sequently, extensive research concerning to the harmonic analysis has been conducted within the Bessel setting, see \cite{AlBeFaRo,AnKe,BeChFaRo,BeDuLeLiWi,BeHaNoVi,Bui,ChDuLiTY,DuoLiMaWuYa,GuLiYa,Ke,LiLiangWenWu,S,YY} and the references therein.

Here, and thereafter, we always denote $d\mu(x)=x^{2\lambda}dx$. With respect to the standard metric on $\mathbb{R}_+$ induced by the absolute value $|\cdot|$, the measure $\mu$ is doubling. Equivalently, the triple $(\mathbb{R}_+,|\cdot|,d\mu)$ is a space of homogeneous type in the sense of Coifman and Weiss \cite{CoWe}. Thus, the usual Muckenhoupt classes $A_p(d\mu)$, which is defined by
\begin{align*}
\sup_{I\subset\mathbb{R}_+}\Big(\frac{1}{\mu(I)}\int_Iw(x)d\mu(x)\Big)
\Big(\frac{1}{\mu(I)}\int_Iw(x)^{1-p'}d\mu(x)\Big)^{p-1}<\infty,
\end{align*}
are the nature classes to consider for weighted $L^p(\mathbb{R}_+,wd\mu)$ estimates of operators in the Bessel setting, where the weighted $L^p(\mathbb{R}_+,wd\mu)$ norm is defined by
\begin{align*}
\|f\|_{L^p(\mathbb{R}_+,wd\mu)}:=\Big(\int_{\mathbb{R}_+}|f(x)|^pw(x)d\mu(x)\Big)^{1/p}.
\end{align*}
We refer readers to see \cite{AlBeFaRo,BeFaSa,DuoGonKufLiWiYa} for the weighted $L^p(\mathbb{R}_+,wd\mu)$ bounds of harmonic analysis operators associated to $\Delta_\lambda$ with $w\in A_p(d\mu)$. However, the early work of Anderson and Kerman \cite{AnKe} revealed that this point of view may not fully capture the intrinsic structure of the Bessel operator. Let $d\nu_\lambda(x)=x^{2\lambda+1}dx$. In \cite{AnKe}, Anderson and Kerman proved that the Bessel Riesz transform $R_{\Delta_\lambda}$ is bounded on $L^p(\mathbb{R}_+,wdx)$ if and only if $w\in A_{p,\lambda}$, where the weighted $L^p(\mathbb{R}_+,wdx)$ norm is defined by
\begin{align*}
\|f\|_{L^p(\mathbb{R}_+,wdx)}:=\Big(\int_{\mathbb{R}_+}|f(x)|^pw(x)dx\Big)^{1/p}.
\end{align*}
and the weight class $A_{p,\lambda}$ is characterized by
\begin{align*}
[w]_{A_{p,\lambda}}:=\sup_{I\subset\mathbb{R}_+}\Big(\frac{1}{\nu_\lambda(I)}\int_Ix^pw(x)dx\Big)
\Big(\frac{1}{\nu_\lambda(I)}\int_Ix^{2\lambda p'}w(x)^{1-p'}dx\Big)^{p-1}<\infty.
\end{align*}
In view of this work, one may wonder whether the Hardy-Littlewood maximal operator $\mathcal{M}_\lambda$ associated to $\Delta_\lambda$
\begin{align*}
\mathcal{M}_\lambda f(x):=\sup_{I\ni x}\frac{1}{\nu_\lambda(I)}\int_I|f(y)|d\nu_\lambda(y)
\end{align*}
is bounded on $L^p(\mathbb{R}_+,wdx)$ if and only if $w\in A_{p,\lambda}$? Li et al. \cite{LiLiangLinShen} gave a negative answer, and they found that
\begin{align}\label{scalar characterization new weight}
w\in\widetilde{A}_{p,\lambda}\Leftrightarrow\|\mathcal{M}_\lambda f\|_{L^p(\mathbb{R}_+,wdx)}\lesssim\|f\|_{L^p(\mathbb{R}_+,wdx)},~1<p<\infty,
\end{align}
where the definition of $\widetilde{A}_{p,\lambda}$ is given below.
\begin{definition}{\rm(\cite{LiLiangLinShen,LiLiangShenW})}
Let $w$ be a weight defined on $\mathbb{R}_+$ and $\lambda>-1/2,\neq0$. We say that $w\in \widetilde{A}_{p,\lambda}$ with $1<p<\infty$ if
\begin{align*}
[w]_{\widetilde{A}_{p,\lambda}}:=\sup_{I\subset\mathbb{R}_+}\Big(\frac{1}{\nu_\lambda(I)}\int_Iw(x)dx\Big)
\Big(\frac{1}{\nu_\lambda(I)}\int_Ix^{(2\lambda+1)p'}w(x)^{1-p'}dx\Big)^{p-1}<\infty,
\end{align*}
where the supremum is taken over all intervals $I\subset\mathbb{R}_+$. We say $w\in \widetilde{A}_{1,\lambda}$ if
\begin{align*}
[w]_{\widetilde{A}_{1,\lambda}}:=\sup_{I\subset\mathbb{R}_+}\Big(\frac{1}{\nu_\lambda(I)}\int_Iw(x)dx\Big)
\|(\cdot)^{2\lambda+1}w(\cdot)^{-1}\|_{L^{\infty}(I,dx)}<\infty.
\end{align*}
We define $\widetilde{A}_{\infty,\lambda}:=\bigcup_{p\geq1}\widetilde{A}_{p,\lambda}$.
\end{definition}
\noindent Li et al. \cite{LiLiangLinShen} also showed that $\widetilde{A}_{p,\lambda}$ classes do not coincide with the previous weight classes $A_{p,\lambda}$ and $A_p(d\mu)$. This indicates that the weighted theory in the Bessel setting is essentially different from the classical case and far from being understood.

In this paper, one of our goals is to continue to study this new class of weights since there are many important theories yet to be established. In \cite{LiLiangShenW}, Li, Liang, Shen and Wick obtained the quantitative weighted estimates of Bessel Riesz transforms with respect to $\widetilde{A}_{p,\lambda-\frac{1}{2}}$. We also refer readers to see \cite{Wen} for the recent work of the sharp weighted bounds for the Bessel Riesz transforms and their commutators with respect to $A_{p,\lambda}$. It is well known that the quantitative weighted estimates for harmonic analysis operators origin from the classical work of Buckley \cite{Bu}, who proved the following sharp weighted $L^p$ bounds for Hardy-Littlewood maximal operator $M$,
\begin{align*}
\|Mf\|_{L^p(\mathbb{R}^n,wdx)}\lesssim[w]_{A_p}^{\frac{1}{p-1}}\|f\|_{L^p(\mathbb{R}^n,wdx)},
\end{align*}
where $1/(p-1)$ is sharp in the sense that it cannot be replaced by a smaller one.  It seems that the sharp weighted estimate of $\mathcal{M}_\lambda$ is missing in the previous works. Since $\widetilde{A}_{p,\lambda}$ is the appropriate class that ties to the intrinsic behavior of $\mathcal{M}_\lambda$, it is natural to ask the following question:

\textbf{Question 1.} Can we obtain the sharp weighted bounds for $\mathcal{M}_\lambda$ with respect to $\widetilde{A}_{p,\lambda}$?

Moreover, in 2013, Hyt\"{o}nen and P\'{e}rez \cite{HyPe} improved Buckley's result by giving the following mixed $A_p$-$A_\infty$ bound for $M$,
\begin{align*}
\|Mf\|_{L^p(\mathbb{R}^n,wdx)}\lesssim([w]_{A_p}[w^{1-p'}]_{A_\infty})^{\frac{1}{p}}
\|f\|_{L^p(\mathbb{R}^n,wdx)}\leq[w]_{A_p}^{\frac{1}{p-1}}\|f\|_{L^p(\mathbb{R}^n,wdx)}.
\end{align*}
The above result motivates our second question.

\textbf{Question 2.} Is it possible to give the mixed $\widetilde{A}_{p,\lambda}$-$\widetilde{A}_{\infty,\lambda}$ bound for $\mathcal{M}_\lambda$?

To solve this question, the first issue we are confronted with is how to define the $[w]_{\widetilde{A}_{\infty,\lambda}}$? We need to ensure that this definition is consistent with $\widetilde{A}_{\infty,\lambda}:=\bigcup_{p\geq1}\widetilde{A}_{p,\lambda}$ defined in \cite{LiLiangShenW}. Furthermore, the mixed $\widetilde{A}_{p,\lambda}$-$\widetilde{A}_{\infty,\lambda}$ bound for $\mathcal{M}_\lambda$ depends on the reverse H\"{o}lder inequality, the authors in \cite{LiLiangLinShen} only proved the reverse H\"{o}lder inequality (RHI) for $\widetilde{A}_{p,\lambda}$, which is not enough for establishing
the mixed weighted bound for $\mathcal{M}_\lambda$.

To give positive answers to Questions 1 and 2, we first define suitable weight constant of $\widetilde{A}_{\infty,\lambda}$, which is inspired by the classical Fujii-Wilson type $A_{\infty}$ constant in \cite{Fu,Wi}.
\begin{definition}\label{A}
Let $\lambda>-1/2,\neq0$ and $v_\lambda(t):=t^{2\lambda+1}$. We define Fujii-Wilson type $\widetilde{A}_{\infty,\lambda}$ constants as follows
\begin{align*}
[w]_{\widetilde{A}_{\infty,\lambda}}:=\sup_{I\subset\mathbb{R}_+}\frac{1}{w(I)}
\int_I\mathcal{M}_\lambda(wv_\lambda^{-1}\chi_I)(t)d\nu_\lambda(t)<\infty,
\end{align*}
where the supremum is taken over all intervals $I\subset\mathbb{R}_+$.
\end{definition}

Our first theorem is concerned with the characterization of RHI via $\widetilde{A}_{\infty,\lambda}$, which extends the RHI for $\widetilde{A}_{p,\lambda}$ obtained in \cite{LiLiangLinShen} to $\widetilde{A}_{\infty,\lambda}$.
\begin{theorem}\label{Reverse Holder inequality theorem}
Let $v_\lambda(t):=t^{2\lambda+1}$ with $\lambda>-1/2,\neq0$ and $I_0\subset\mathbb{R}_+$ be an interval.
\begin{itemize}
\item [(1)] Let $w\in\widetilde{A}_{\infty,\lambda}$. Then for any $\epsilon\in\Big(0,\frac{1}{2^{2\lambda+3}[w]_{\widetilde{A}_{\infty,\lambda}}-1}\Big]$,
\begin{align}\label{Reverse Holder inequality}
\Big(\frac{1}{\nu_\lambda(I_0)}\int_{I_0}(w(t)v_\lambda(t)^{-1})
^{1+\epsilon}d\nu_\lambda(t)\Big)^{\frac{1}{1+\epsilon}}\leq
\frac{2}{\nu_\lambda(I_0)}\int_{I_0}w(t)v_{\lambda}(t)^{-1}d\nu_\lambda(t).
\end{align}
\item [(2)] If \eqref{Reverse Holder inequality} holds, then $w\in\widetilde{A}_{\infty,\lambda}$.
\end{itemize}
\end{theorem}

The following theorem give several characterizations of $\widetilde{A}_{\infty,\lambda}$, which also illustrates that $\{w:[w]_{\widetilde{A}_{\infty,\lambda}}<\infty\}$ coincides $\widetilde{A}_{\infty,\lambda}£º=\bigcup_{p\geq1}\widetilde{A}_{p,\lambda}$ defined in \cite{LiLiangShenW}.

\begin{theorem}\label{characterization A}
Let $w$ be a weight on $\mathbb{R}_+$ and $\lambda>-1/2,\neq0$. Then the following statements are equivalent.
\begin{itemize}
\item [(1)] $w\in \widetilde{A}_{\infty,\lambda}$.
\item [(2)] There exist $c_1, c_2>0$ such that for each interval $I\subset\mathbb{R}_+$ and each measurable subset $J\subset I$,
    \begin{align*}
    \frac{w(J)}{w(I)}\lesssim c_1\Big(\frac{\nu_\lambda(J)}{\nu_\lambda(I)}\Big)^{c_2}.
    \end{align*}
\item [(3)] There exist $\alpha,\beta\in(0,1)$ such that for each interval $I\subset\mathbb{R}_+$ and each measurable subset $J\subset I$,
    \begin{align*}
    w(J)<\alpha w(I)\Rightarrow\nu_\lambda(J)<\beta\nu_\lambda(I).
    \end{align*}
\item [(4)] There exists $p\geq1$ such that $w\in\widetilde{A}_{p,\lambda}$.
\item [(5)] There exists $\epsilon>0$ such that
\begin{align*}
\Big(\frac{1}{\nu_\lambda(I_0)}\int_{I_0}(w(t)v_\lambda(t)^{-1})
^{1+\epsilon}d\nu_\lambda(t)\Big)^{\frac{1}{1+\epsilon}}\lesssim
\frac{1}{\nu_\lambda(I_0)}\int_{I_0}w(t)v_{\lambda}(t)^{-1}d\nu_\lambda(t).
\end{align*}
\end{itemize}
\end{theorem}

As an application of RHI, we have the following ``precise open property'' of $\widetilde{A}_{p,\lambda}$. This property is useful in obtaining the sharp $L^p(\mathbb{R}_+,wdx)$-bounds of $\mathcal{M}_\lambda$.
\begin{theorem}\label{open property of weight}
Let $1<p<\infty$, $w\in \widetilde{A}_{p,\lambda}$ with $\lambda>-1/2,\neq0$ and $v_\lambda(t):=t^{2\lambda+1}$. Then $w\in\widetilde{A}_{p-\epsilon,\lambda}$, where
\begin{align*}
\epsilon=\frac{p-1}{r_\sigma'}:=\frac{p-1}{2^{2\lambda+3}[\sigma]_{\widetilde{A}_{\infty,\lambda}}+1},~
\sigma:=v_\lambda^{p'}w^{1-p'}.
\end{align*}
Moreover,
\begin{align*}
[w]_{\widetilde{A}_{p-\epsilon,\lambda}}\leq 2^{p-1}[w]_{\widetilde{A}_{p,\lambda}}.
\end{align*}
\end{theorem}
Based on Theorems \ref{Reverse Holder inequality theorem}, \ref{open property of weight}, The following theorem gives the full answers of Questions 1 and 2.
\begin{theorem}\label{mixed quantitative weighted strong estimates}
Let $1<p<\infty$, $w\in \widetilde{A}_{p,\lambda}$ with $\lambda>-1/2,\neq0$ and $v_\lambda(t):=t^{2\lambda+1}$. Then
\begin{align*}
\|\mathcal{M}_\lambda f\|_{L^p(\mathbb{R}_+,wdx)}\lesssim([w]_{\widetilde{A}_{p,\lambda}}
[v_\lambda^{p'}w^{1-p'}]_{\widetilde{A}_{\infty,\lambda}})^{\frac{1}{p}}\|f\|_{L^p(\mathbb{R}_+,wdx)}\leq
[w]_{\widetilde{A}_{p,\lambda}}^{\frac{1}{p-1}}\|f\|_{L^p(\mathbb{R}_+,wdx)}.
\end{align*}
Furthermore, the exponent $1/(p-1)$ is sharp.
\end{theorem}

Our second goal of this paper is to establish the Rubio de Francia extrapolation theorem related to $\widetilde{A}_{p,\lambda}$. It is well known that extrapolation theorem of Rubio de Francia, announced in \cite{Ru1} and given with a detailed proof in \cite{Ru2}, is one of the most powerful and valuable tools in the theory of weighted norm inequalities. The extrapolation theorem states that if the operator $T$ is bounded on $L^{p_0}(\mathbb{R}^n,wdx)$ for some $1\leq p_0<\infty$ and every $w\in A_{p_0}$, then for every $1<p<\infty$, $T$ is bounded on $L^{p}(\mathbb{R}^n,wdx)$ for $w\in A_p$. It was later observed in \cite{CrPe} that the operator itself does not play any role and all the statements can be given in terms of families of pairs of nonnegative measurable functions. In order to study the sharp weighted $L^p$ bounds of some classical operators, Dragi\u{c}evi\'{c}, Grafakos, Pereyra and Petermichl \cite{DrGrPePe} established the sharp extrapolation theorem. We also refer readers to see the work of Duoandikoetxea \cite{D}, who presented an alternative, streamlined proof of the sharp extrapolation theorem. The sharp extrapolation theorem was utilized by Hyt\"{o}nen \cite{Hy} to prove the celebrated $A_p$ theorem for Calder\'{o}n-Zygmund operators. For more works on extrapolation theorem, one can see \cite{CaLiSiYa,CaMM,CaOl,CrMa,CMP11,HaMaSe,LiMaMaOmVu,LiMaOm} and the references therein.

\textbf{Question 3.} Can we establish extrapolation theorems in terms of $\widetilde{A}_{p,\lambda}$?

\medskip

This question will be addressed by our next theorem.

\begin{theorem}\label{Rubio de Francia extrapolation theorem for new weights}
Let $(f,g)\in\mathcal{F}$, where $\mathcal{F}$ is a family of pairs $(f,g)$ of non-negative, measurable functions that are not identically zero. Let $\lambda>-1/2,\neq0$. Assume that for some $p_0\in[1,\infty)$, there is a positive increasing function $N_{p_0}$ on $[1,\infty)$ such that for each $w\in \widetilde{A}_{p_0,\lambda}$,
\begin{align*}
\int_{\mathbb{R}_+}f(x)^{p_0}w(x)dx\leq N_{p_0}([w]_{\widetilde{A}_{p_0,\lambda}})\int_{\mathbb{R}_+}g(x)^{p_0}w(x)dx.
\end{align*}
Then for all $1<p<\infty$ and all $w\in\widetilde{A}_{p,\lambda}$,
\begin{align*}
\int_{\mathbb{R}_+}f(x)^{p}w(x)dx\leq N_{p}(p_0,p,\lambda,[w]_{\widetilde{A}_{p,\lambda}})\int_{\mathbb{R}_+}g(x)^{p}w(x)dx,
\end{align*}
where
\[
N_{p}(p_0,p,\lambda,[w]_{\widetilde{A}_{p,\lambda}})\\
:=
\begin{cases}
2^{\frac{p_0-p}{p_0}}N_{p_0}\Big(2^{p_0-p}
\|\mathcal{M}_\lambda\|_{\mathcal{B}(L^p(\mathbb{R}_+,wdx))}^{p_0-p}[w]_{\widetilde{A}_{p,\lambda}}\Big),& p<p_0,\\
2^{\frac{p-p_0}{p_0(p-1)}}N_{p_0}\Big(2^{\frac{p-p_0}{p-1}}
\|\mathcal{M}_\lambda\|_{\mathcal{B}(L^{p'}(\mathbb{R}_+,v_{\lambda}^{p'}w^{1-p'}dx))}^{\frac{p-p_0}{p-1}}
[w]_{\widetilde{A}_{p,\lambda}}
^{\frac{p_0-1}{p-1}}\Big),& p>p_0.
\end{cases}
\]
and
$$\|\mathcal{M}_\lambda\|_{\mathcal{B}(L^p(\mathbb{R}_+,wdx))}:=
\|\mathcal{M}_\lambda\|_{L^p(\mathbb{R}_+,wdx)\rightarrow
L^p(\mathbb{R}_+,wdx)},$$
$$\|\mathcal{M}_\lambda\|_{\mathcal{B}(L^{p'}(\mathbb{R}_+,v_{\lambda}^{p'}w^{1-p'}dx))}:=
\|\mathcal{M}_\lambda\|_{L^{p'}(\mathbb{R}_+,v_{\lambda}^{p'}w^{1-p'}dx)\rightarrow L^{p'}(\mathbb{R}_+,v_{\lambda}^{p'}w^{1-p'}dx)}.$$
\end{theorem}

\begin{corollary}\label{vector-value extrapolation}
Let $(f,g)\in\mathcal{F}$, where $\mathcal{F}$ is a family of pairs $(f,g)$ of non-negative, measurable functions that are not identically zero. Let $\lambda>-1/2,\neq0$. Assume that for some $p_0\in[1,\infty)$, and each $w\in \widetilde{A}_{p_0,\lambda}$,
\begin{align*}
\|f\|_{L^{p_0}(\mathbb{R}_+,wdx)}\lesssim \|g\|_{L^{p_0}(\mathbb{R}_+,wdx)}.
\end{align*}
Then for all $1<p,q<\infty$, all $w\in\widetilde{A}_{p,\lambda}$, and $\{(f_i,g_i)\}_i\subset\mathcal{F}$,
\begin{align*}
\Big\|\Big(\sum_if_i^q\Big)^{1/q}\Big\|_{L^p(\mathbb{R}_+,wdx)}\lesssim
\Big\|\Big(\sum_ig_i^q\Big)^{1/q}\Big\|_{L^p(\mathbb{R}_+,wdx)}.
\end{align*}
\end{corollary}

\begin{corollary}\label{scale}
Let $(f,g)\in\mathcal{F}$, where $\mathcal{F}$ is a family of pairs $(f,g)$ of non-negative, measurable functions that are not identically zero. Let $\lambda>-1/2,\neq0$. Assume that for some $r>0$ and some $p_0\in[r,\infty)$, and for each $w\in \widetilde{A}_{p_0/r,\lambda}$,
\begin{align*}
\|f\|_{L^{p_0}(\mathbb{R}_+,wdx)}\lesssim \|g\|_{L^{p_0}(\mathbb{R}_+,wdx)}.
\end{align*}
Then for all $p\in(r,\infty)$ and all $w\in\widetilde{A}_{p/r,\lambda}$,
\begin{align*}
\|f\|_{L^{p}(\mathbb{R}_+,wdx)}\lesssim \|g\|_{L^{p}(\mathbb{R}_+,wdx)}.
\end{align*}
\end{corollary}

\begin{corollary}\label{A extraapolaation}
Let $(f,g)\in\mathcal{F}$, where $\mathcal{F}$ is a family of pairs $(f,g)$ of non-negative, measurable functions that are not identically zero. Let $\lambda>-1/2,\neq0$. Assume that for some $p_0\in(0,\infty)$, and each $w\in \widetilde{A}_{\infty,\lambda}$,
\begin{align*}
\|f\|_{L^{p_0}(\mathbb{R}_+,wdx)}\lesssim \|g\|_{L^{p_0}(\mathbb{R}_+,wdx)}.
\end{align*}
Then for all $0<p<\infty$ and all $w\in\widetilde{A}_{\infty,\lambda}$,
\begin{align*}
\|f\|_{L^{p}(\mathbb{R}_+,wdx)}\lesssim \|g\|_{L^{p}(\mathbb{R}_+,wdx)}.
\end{align*}
\end{corollary}

\begin{corollary}\label{weak type extrapolation}
Let $(f,g)\in\mathcal{F}$, where $\mathcal{F}$ is a family of pairs $(f,g)$ of non-negative, measurable functions that are not identically zero. Let $\lambda>-1/2,\neq0$. Assume that for some $p_0\in[1,\infty)$, and each $w\in \widetilde{A}_{p_0,\lambda}$,
\begin{align*}
\|f\|_{L^{p_0,\infty}(\mathbb{R}_+,wdx)}\lesssim \|g\|_{L^{p_0}(\mathbb{R}_+,wdx)}.
\end{align*}
Then for all $1<p<\infty$ and all $w\in\widetilde{A}_{p,\lambda}$,
\begin{align*}
\|f\|_{L^{p,\infty}(\mathbb{R}_+,wdx)}\lesssim \|g\|_{L^{p}(\mathbb{R}_+,wdx)}.
\end{align*}
\end{corollary}

From Corollaries \ref{A extraapolaation}, \ref{weak type extrapolation}, we have the following two corollaries. Since the proofs are almost the same as Corollaries \ref{vector-value extrapolation}, \ref{A extraapolaation}, we omit their proofs.

\begin{corollary}\label{A extrapolaation vector-value}
Let $(f,g)\in\mathcal{F}$, where $\mathcal{F}$ is a family of pairs $(f,g)$ of non-negative, measurable functions that are not identically zero. Let $\lambda>-1/2,\neq0$. Assume that for some $p_0\in(0,\infty)$, and each $w\in \widetilde{A}_{\infty,\lambda}$,
\begin{align*}
\|f\|_{L^{p_0}(\mathbb{R}_+,wdx)}\lesssim \|g\|_{L^{p_0}(\mathbb{R}_+,wdx)}.
\end{align*}
Then for all $0<p,q<\infty$, all $w\in\widetilde{A}_{\infty,\lambda}$, and $\{(f_i,g_i)\}_i\subset\mathcal{F}$,
\begin{align*}
\Big\|\Big(\sum_if_i^q\Big)^{1/q}\Big\|_{L^{p}(\mathbb{R}_+,wdx)}\lesssim
\Big\|\Big(\sum_ig_i^q\Big)^{1/q}\Big\|_{L^p(\mathbb{R}_+,wdx)}.
\end{align*}
\end{corollary}

\begin{corollary}
Let $(f,g)\in\mathcal{F}$, where $\mathcal{F}$ is a family of pairs $(f,g)$ of non-negative, measurable functions that are not identically zero. Let $\lambda>-1/2,\neq0$. Assume that for some $p_0\in(0,\infty)$, and each $w\in \widetilde{A}_{\infty,\lambda}$
\begin{align*}
\|f\|_{L^{p_0,\infty}(\mathbb{R}_+,wdx)}\lesssim\|g\|_{L^{p_0}(\mathbb{R}_+,wdx)}.
\end{align*}
Then for all $0<p<\infty$, and all $w\in\widetilde{A}_{\infty,\lambda}$,
\begin{align*}
\|f\|_{L^{p,\infty}(\mathbb{R}_+,wdx)}\lesssim\|g\|_{L^p(\mathbb{R}_+,wdx)}.
\end{align*}
\end{corollary}

The last goal of this article is to study the restricted weighted weak type estimates of harmonic analysis operators associated to $\Delta_\lambda$. Let us recall the background of this topic. Let $1\leq p<\infty$. Chung, Hunt, Kurtz \cite{ChHuKu}, and Kerman, Torchinsky \cite{KT} introduced a new class of weights $A_p^\mathcal{R}$ for $1<p<\infty$, which is strictly containing Muckenhoupt $A_p$ weights, to characterize the restricted weighted weak $(p,p)$ estimate of Hardy-Littlewood maximal operator $M$:
\begin{align}\label{classical restricted weak type estimate}
\|M\|_{L^{p,1}(\mathbb{R}^n,wdx)\rightarrow L^{p,\infty}(\mathbb{R}^n,wdx)}\sim[w]_{A_p^\mathcal{R}}<\infty,
\end{align}
where
\begin{align*}
[w]_{A_p^\mathcal{R}}:=\sup_{Q}w(Q)^{1/p}\frac{\|\chi_Q w^{-1}\|_{L^{p',\infty}(\mathbb{R}^n,wdx)}}{|Q|}.
\end{align*}
Since $A_1^\mathcal{R}=A_1$, \eqref{classical restricted weak type estimate} is exactly the $A_1$ weights characterization in terms of weighted weak (1,1) estimate of $M$ when $p=1$. Recently, $A_p^\mathcal{R}$ weights have received considerable attention. We refer readers to see \cite{CaSiZh,PeR,R} for weighted inequalities concerned with $A_p^\mathcal{R}$, see \cite{CaGrSo,CaOm,LikLuOm,R} for extrapolation theory related to $A_p^\mathcal{R}$.

\textbf{Question 4.} In view of the early work of Kerman and Torchinsky \cite{KT}, it is natural to ask what is the appropriate weights to characterize the restricted weighted weak type estimates of $\mathcal{M}_\lambda$?

In the classical proof of characterization of the restricted weighted weak
type estimate of uncentered Hardy-Littlewood maximal operator, the strategy is to transform it to centered Hardy-Littlewood maximal operator and applying Besicovich covering lemma. This technique fails in the Bessel setting since the center of the interval shifts as the interval expands. Hence, we need to improve the existing method. Furthermore, can we establish quantitative restricted weighted weak type estimates of other harmonic analysis operators associated to $\Delta_\lambda$?

Before we give the characterization of restricted weighted weak type estimate of $\mathcal{M}_\lambda$, we first introduce a new class of weights.
\begin{definition}
Let $1\leq p<\infty$, $v_\lambda(t):=t^{2\lambda+1}$ with $\lambda>-1/2,\neq0$. Let $(w_1,w_2)$ be a pair of weights. We call $(w_1,w_2)\in \widetilde{A}_{p,\lambda}^\mathcal{R}$ if
\begin{equation*}
[w_1,w_2]_{\widetilde{A}_{p,\lambda}^\mathcal{R}}:=\sup_{I\subset\mathbb{R}_+}w_2(I)^{1/p}
\frac{\|\chi_Iv_\lambda w_1^{-1}\|_{L^{p',\infty}(\mathbb{R}_+,w_1dx)}}{\nu_\lambda(I)}<\infty.
\end{equation*}
If $w_1=w_2=w$, we denote $(w,w)\in\widetilde{A}_{p,\lambda}^\mathcal{R}$ by $w\in\widetilde{A}_{p,\lambda}^\mathcal{R}$.
\end{definition}

The next Theorems \ref{characterization restricted weak type of HL} and \ref{another characterization restricted weak type} provide several characterizations of $\widetilde{A}_{p,\lambda}^\mathcal{R}$.
\begin{theorem}\label{characterization restricted weak type of HL}
Let $1\leq p<\infty$, $\lambda>-1/2, \neq0$ and $w_1, w_2$ be weights. Then
\begin{equation*}
\|\mathcal{M}_\lambda \|_{L^{p,1}(\mathbb{R}_+, w_1dx)\rightarrow L^{p,\infty}(\mathbb{R}_+, w_2dx)}\sim [w_1,w_2]_{\widetilde{A}_{p,\lambda}^{\mathcal{R}}}.
\end{equation*}
\end{theorem}

\begin{theorem}\label{another characterization restricted weak type}
Let $1\leq p<\infty$, $\lambda>-1/2, \neq0$ and $w_1, w_2$ be weights. The following assertions are equivalent:
\begin{itemize}
\item [(1)]$\|\mathcal{M}_\lambda f\|_{L^{p,\infty}(\mathbb{R}_+, w_2dx)}\lesssim\|f\|_{L^{p,1}(\mathbb{R}_+, w_1dx)}$;
\item [(2)]$\|\mathcal{M}_\lambda(\chi_J)\|_{L^{p,\infty}(\mathbb{R}_+, w_2dx)}\lesssim w_1(J)^{1/p}$ for any measurable set $J\subseteq\mathbb{R}_+$;
\item [(3)]
$
\|w_1,w_2\|_{\widetilde{A}_{p, \lambda}^{\mathcal{R}}}:=\sup_{I\subset\mathbb{R}_+}w_2(I)^{1/p}
    \sup_{0<w_1(J)<\infty}\frac{\nu_\lambda(I\cap J)}{\nu_\lambda(I)}w_1(J)^{-1/p}
    <\infty;
$
\item [(4)]$(w_1,w_2)\in \widetilde{A}_{p,\lambda}^{\mathcal{R}}$.
\end{itemize}
\end{theorem}

Relied on Theorem \ref{characterization restricted weak type of HL}, we have the following quantitative restricted weighted weak type estimates for sparse operators.
\begin{theorem}\label{quantitative two weight restricted weak type of sparse operator}
Let $\mathcal{S}$ be an $\eta$-sparse family. Let $1<p<\infty$, $(w_1,w_2)\in\widetilde{A}_{p,\lambda-\frac{1}{2}}^\mathcal{R}$, $w_2\in\widetilde{A}_{\infty,\lambda-\frac{1}{2}}$. Then there exists $r\geq1$ such that
\begin{align*}
\|\mathcal{A}_\mathcal{S}f\|_{L^{p,\infty}(\mathbb{R}_+,w_2dx)}\lesssim
[w_2]_{\widetilde{A}_{r,\lambda-\frac{1}{2}}^{\mathcal{R}}}^{r}
[w_1,w_2]_{\widetilde{A}_{p,\lambda-\frac{1}{2}}^\mathcal{R}}\|f\|_{L^{p,1}(\mathbb{R}_+,w_1dx)},
\end{align*}
where
\begin{align*}
\mathcal{A}_\mathcal{S}f(x):=\sum_{I\in\mathcal{S}}\frac{1}{\mu(I)}\int_I |f(t)|d\mu(t)\chi_I(x).
\end{align*}
\end{theorem}

To establish quantitative restricted weighted weak type estimates for sparse operators associated to $BMO$ functions, we need more preparation tools, among which is the following characterizations of restricted weighted weak type estimates of Orlicz maximal functions associated to $\Delta_\lambda$. Unlike the classical case, we emphasize that extra doubling condition of $w_2$ is needed due to the intrinsic structure of Bessel operator.
\begin{theorem}\label{characterization two weight restricted weak type estimate of Orlicz maximal}
Let $1<p<\infty$, $\lambda>-1/2, \neq0$ and $\phi$ be a Young function. Let $w_1$ be a weight and $w_2$ be a doubling weight. Then
\begin{equation*}
\|\mathcal{M}_{\phi,\nu_\lambda}f\|_{L^{p,\infty}(\mathbb{R}_+,w_2dx)}\lesssim\|f\|_{L^{p,1}(\mathbb{R}_+,w_1dx)}
\end{equation*}
if and only if $(w_1,w_2)\in\widetilde{A}_{p,\lambda,\phi}^\mathcal{R}$, where the weight constant of $\widetilde{A}_{p,\lambda,\phi}^\mathcal{R}$ and $\mathcal{M}_{\phi,\nu_\lambda}$ are defined by
\begin{align*}
\|w_1,w_2\|_{\widetilde{A}_{p,\lambda,\phi}^\mathcal{R}}:=\sup_{I\subset\mathbb{R}_+}\sup_{J\subset I}\|\chi_J\|_{\phi,I,\nu_\lambda}\Big(\frac{w_2(I)}{w_1(J)}\Big)^{1/p}<\infty
\end{align*}
and
\begin{align*}
\mathcal{M}_{\phi,\nu_\lambda}f(x):=\sup_{I\ni x,I\subset\mathbb{R}_+}\|f\|_{\phi,I,\nu_\lambda}\chi_I(x),
\end{align*}
respectively. Moreover, if $(w_1,w_2)\in\widetilde{A}_{p,\lambda,\phi}^\mathcal{R}$ and $w_2\in \widetilde{A}_{\infty,\lambda}$, there exists $r\geq1$ such that
\begin{align}\label{quantitative two weight restricted weak type estimate of Orlicz maximal}
\|\mathcal{M}_{\phi,\nu_\lambda}f\|_{L^{p,\infty}(\mathbb{R}_+,w_2dx)}\lesssim
[w_2]_{\widetilde{A}_{r,\lambda}^\mathcal{R}}^{r/p}\|w_1,w_2\|_{\widetilde{A}_{p,\lambda,\phi}^\mathcal{R}}
\|f\|_{L^{p,1}(\mathbb{R}_+,w_1dx)}.
\end{align}
\end{theorem}

The following theorem is concerned with quantitative restricted weighted weak type estimates of sparse operators associated to $BMO$ functions.
\begin{theorem}\label{quantitative two weight restricted weak type estimate for sparse operator of commutator 2}
Let $1<p<\infty$, $\Phi(t)=t\log(e+t)^m$ with $m\in\mathbb{Z}^+$ and $\mathcal{S}$ be an $\eta$-sparse family. If $w_2\in\widetilde{A}_{\infty,\lambda-\frac{1}{2}}$, $(w_1,w_2)\in\widetilde{A}_{p,\lambda-\frac{1}{2},\Phi}^\mathcal{R}$ and $b\in BMO_{\Delta_\lambda}(\mathbb{R}_+)$. Then there exists $r\geq1$ such that
\begin{align*}
&\|\mathcal{A}_{\mathcal{S},b}^mf\|_{L^{p,\infty}(\mathbb{R}_+,w_2dx)}
+\|(\mathcal{A}_{\mathcal{S},b}^m)^{\ast}f\|_{L^{p,\infty}(\mathbb{R}_+,w_2dx)}\\
&\quad\lesssim\|b\|_{BMO_{\Delta_\lambda}(\mathbb{R}_+)}^m
[w_2]_{\widetilde{A}_{\infty,\lambda-\frac{1}{2}}}^m
[w_2]_{\widetilde{A}_{r,\lambda-\frac{1}{2}}^\mathcal{R}}^{r+\frac{r}{p}}
\|w_1,w_2\|_{\widetilde{A}_{p,\lambda-\frac{1}{2},\Phi}^\mathcal{R}}\|f\|_{L^{p,1}(\mathbb{R}_+,w_1dx)},
\end{align*}
where
\begin{align*}
\mathcal{A}_{\mathcal{S},b}^mf(x):=\sum_{I\in\mathcal{S}}\frac{1}{\mu(I)}|b(x)-b_{I,\mu}|^m
\int_I|f(t)|d\mu(t)\chi_I(x),~b_{I,\mu}:=\frac{1}{\mu(I)}\int_Ib(t)d\mu(t)
\end{align*}
and
\begin{align*}
(\mathcal{A}_{\mathcal{S},b}^m)^{\ast}f(x):=\sum_{I\in\mathcal{S}}\frac{1}{\mu(I)}
\int_I|b(z)-b_{I,\mu}|^m|f(z)|d\mu(z)\chi_I(x).
\end{align*}
\end{theorem}

We organize the rest of the paper as follows. We give some basic definitions and useful conclusions in Section $2$. In Section $3$, we give the proofs of Theorems \ref{Reverse Holder inequality theorem}-\ref{mixed quantitative weighted strong estimates}. In Section $4$, we prove Theorems \ref{Rubio de Francia extrapolation theorem for new weights}, Corollaries \ref{vector-value extrapolation}-\ref{weak type extrapolation}. The proofs of Theorems \ref{characterization restricted weak type of HL}-\ref{quantitative two weight restricted weak type estimate for sparse operator of commutator 2} are given in Section $5$. Finally, the applications of our theorems will be presented in Section $6$.

Throughout this paper, $d\mu(x)=x^{2\lambda}dx$, $d\nu_\lambda(x)=x^{2\lambda+1}dx$. We denote $f\lesssim g$, $f\thicksim g$ if $f\leq Cg$ and $f\lesssim g \lesssim f$, respectively. For $x,r,\sigma>0$, $I(x,r):=(x-r,x+r)\cap\mathbb{R}_+$, $I(x,\sigma r):=(x-\sigma r,x+\sigma r)\cap\mathbb{R}_+$. $\chi_I$ represents the characteristic function of measurable set $I$, $f_{I,\mu}$ represents the average $\mu(I)^{-1}\int_If(x)d\mu(x)$ of $f$ over the interval $I$.

\section{Preliminaries}
\subsection{A system of dyadic cubes and sparse family}
Let us recall the following definition of dyadic systems in a space of homogeneous type $X$.
\begin{definition}{\rm(\cite{Chr,HyK})}
Let $(X,d,\mu)$ be a space of homogeneous type. A countable family $\mathcal{D}:=\bigcup_{k\in\mathbb{Z}}\mathcal{D}_k$, $\mathcal{D}_k:=\{Q_\alpha^k:\alpha\in\mathcal{I}_k\}$, of Borel sets $Q_\alpha^k\subset X$ is called a system of dyadic cubes with parameters $0<a_1\leq A_1\leq\infty$ and $\delta\in(0,1)$ if the following properties hold:
\begin{itemize}
\item [(1)] $X=\bigcup_{\alpha\in\mathcal{I}_k}Q_\alpha^k~ (disjoint~ union) $ for all $k\in\mathbb{Z}$;
\item [(2)] if $l\geq k$, $\alpha\in\mathcal{I}_k,\tau\in\mathcal{I}_l$, then either $Q_{\tau}^l\subset Q_\alpha^k$ or $Q_{\tau}^l\bigcap Q_\alpha^k=\emptyset$;
\item [(3)] for every $k\in\mathbb{Z}$, $\alpha\in\mathcal{I}_k$ and each $l\leq k$, there exists a unique $\beta\in\mathcal{I}_k$ such that $Q_\alpha^k\subset Q_\beta^l$;
\item [(4)] for every $k\in\mathbb{Z}$ and $\alpha\in\mathcal{I}_k$, there exists at most $M(a~ fixed ~constant)$ $\beta$ such that $Q_\beta^{k+1}\subset Q_\alpha^k$, and $Q_\alpha^k=\bigcup_{Q\in\mathcal{D}_{k+1},Q\subset Q_\alpha^k}Q$;
\item [(5)] for every $k\in\mathbb{Z}$ and $\alpha\in\mathcal{I}_k$, there exists $x_\alpha^k$ such that
    \begin{align*}
    B(x_\alpha^k,a_1\delta^k)\subset Q_\alpha^k\subset B(x_\alpha^k,A_1\delta^k)=:B(Q_\alpha^k);
    \end{align*}
\item [(6)] if $l\geq k$ and $Q_\beta^l\subset Q_\alpha^k$, then $B(Q_\beta^l)\subset B(Q_{\alpha}^k)$.
\end{itemize}
The set $Q_\alpha^k$ is called a dyadic cube of generation $k$ with centre point $x_\alpha^k\in
Q_\alpha^k$ and sidelength $\delta^k$.
\end{definition}

\begin{definition}\rm(\cite{DuoGonKufLiWiYa})
Let $\eta\in(0,1)$. A collection $\mathcal{S}\subseteq\mathcal{D}$ of dyadic cubes is said to be an $\eta$-sparse family if for any cube $Q\in\mathcal{S}$, there exists a measurable subset $E_Q\subseteq Q$ such that $\mu(E_Q)\geq\eta\mu(Q)$, and the sets $\{E_Q\}_{Q\in\mathcal{S}}$ have only finite overlap.
\end{definition}

\subsection{Lorentz spaces and spaces of bounded mean oscillation}
Let $p>0$ and $w$ be a weight. $L^{p,1}(\mathbb{R}_+,wdx)$ is the space of functions such that
\begin{equation*}
\|f\|_{L^{p,1}(\mathbb{R}_+,wdx)}:=p\int_0^\infty w(\{x\in\mathbb{R}^n:|f(x)|>y\})^{1/p}dy<\infty.
\end{equation*}
We will ultimate the following duality inequality, see \cite[Theorem 1.4.16.(v)]{Gra}:
\begin{align*}
\frac{1}{p}\|g\|_{L^{p',\infty}(\mathbb{R}_+,wdx)}&\leq
\sup\Big\{\int_{\mathbb{R}_+}|f(x)g(x)|w(x)dx:\|f\|_{L^{p,1}(\mathbb{R}_+,wdx)}\leq1\Big\}\\
&\leq\|g\|_{L^{p',\infty}(\mathbb{R}_+,wdx)},~1<p<\infty.
\end{align*}
We shall also use the following summation lemma.
\begin{lemma}{\rm(\cite{R})}\label{summation lemma}
Let $0<p<\infty$ and $\gamma\geq\max\{p,1\}$. Given a measurable function $f$ and a weight $w$, if $\{E_j\}$ is a collection of measurable sets such that $\sum_{j\geq1}\chi_{E_j}\leq C$, then
\begin{equation*}
\sum_{j\geq1}\|f\chi_{E_j}\|_{L^{p,1}(\mathbb{R}_+,wdx)}^\gamma\leq C^{\gamma/p}\|f\|_{L^{p,1}(\mathbb{R}_+,wdx)}^\gamma.
\end{equation*}
\end{lemma}

Next, we recall the definition of $BMO_{\Delta_\lambda}(\mathbb{R}_+)$.
\begin{definition}{\rm(\cite{DuoLiWiYa})}
The function space $BMO_{\Delta_\lambda}(\mathbb{R}_+)$ is the set of $b\in L_{loc}^{1}(\mathbb{R}_+,\mu)$ such that
\begin{align*}
\|b\|_{BMO_{\Delta_\lambda}(\mathbb{R}_+)}:=\sup_{I\subset \mathbb{R}_+}\frac{1}{\mu(I)}\int_I|b(t)-b_{I,\mu}|d\mu(t)<\infty,
\end{align*}
where the supremum is taken over all intervals $I\subset\mathbb{R}_+$.
\end{definition}
Treating $(\mathbb{R}_+,|\cdot|,d\mu)$ as a space of homogeneous type, for $b\in BMO_{\Delta_\lambda}(\mathbb{R}_+)$, recall the following standard John-Nirenberg inequality in \cite{Kr}:
\begin{align}\label{John-Nirengerg inequality}
\mu(\{x\in I\subset \mathbb{R}_+:|b(x)-b_{I,\mu}|>\gamma\})\leq c_1e^{-\frac{\gamma}{c_2\|b\|_{BMO_{\Delta_\lambda}(\mathbb{R}_+)}}}\mu(I).
\end{align}
From this and layer-cake formula, there exists a positive constant $c$ such that for any interval $I\subset\mathbb{R}_+$,
\begin{align}\label{JN corollary}
|s|<(ce\|b\|_{BMO_{\Delta_\lambda}(\mathbb{R}_+)})^{-1}&\Rightarrow
\frac{1}{\mu(I)}\int_Ie^{s|b(x)-b_{I,\mu}|}d\mu(x)\\
&\qquad\leq 1+e\frac{(ce)^{-1}|s|\|b\|_{BMO_{\Delta_\lambda}(\mathbb{R}_+)}}
{1-(ce)^{-1}|s|\|b\|_{BMO_{\Delta_\lambda}(\mathbb{R}_+)}}.\nonumber
\end{align}
\subsection{Young functions, Luxemburg norms} Recall that $\Phi:[0,\infty)\rightarrow[0,\infty)$ is called a Young function if it is continuous, increasing, convex and satisfies $\Phi(0)=0$ and $\lim_{t\rightarrow\infty}\Phi(t)/t=\infty$. The corresponding complementary function of $\Phi$, denoted by $\bar{\Phi}:[0,\infty)\rightarrow[0,\infty)$, is given by
\begin{equation*}
\bar{\Phi}(t)=\sup_{s>0}\{st-\Phi(s)\}.
\end{equation*}
Let $\Phi$ be a Young function, $u$ be a measure. The Luxemburg norms $\|f\|_{\Phi,Q,u}$ is defined by
\begin{equation*}
\|f\|_{\Phi,Q,u}:=\inf\Big\{\lambda>0:\frac{1}{u(Q)}\int_Q\Phi\Big(\frac{|f(x)|}
{\lambda}\Big)du(x)\leq1\Big\}.
\end{equation*}
When $du=d\mu$ and $du=wdx$, we denote $\|f\|_{\Phi,Q,u}$ by $\|f\|_{\Phi,Q,\mu}$ and $\|f\|_{\Phi,Q,w}$, respectively. Let $A$ be a Young function, we will frequently use the following generalized H\"{o}lder's inequality,
\begin{align*}
\frac{1}{u(Q)}\int_Q|fg|du\leq2\|f\|_{A,Q,u}\|g\|_{\bar{A},Q,u}.
\end{align*}
If we take $\Phi(t)=e^t-1$, for $b\in BMO_{\Delta_\lambda}(\mathbb{R}_+)$, \eqref{JN corollary} implies that
\begin{equation*}
\|b-b_{I,\mu}\|_{\exp L,I,\mu}\lesssim\|b\|_{BMO_{\Delta_\lambda}(\mathbb{R}_+)}.
\end{equation*}

\section{Characterizations of \texorpdfstring{$\widetilde{A}_{\infty,\lambda}$} ~~weights and sharp weighted bounds of Hardy-Littlewood maximal operators associated to $\Delta_\lambda$}
In this section, we prove Theorems \ref{Reverse Holder inequality theorem}-\ref{mixed quantitative weighted strong estimates}. We will borrow some ideas from \cite{Gra,HyPe,HyPeRe}, however, as we stated in the introduction, the weight classes considered in our work differ from those in these references.

\begin{definition}
Let $\mathcal{D}_k|_{I_0}$ be a family of dyadic intervals on $I_0$ with side length $2^{-k}|I_0|$ and $f\in L_{loc}^1(\mathbb{R}_+,d\nu_\lambda)$. The dyadic maximal function of $f$ associated to the interval $I_0$ is defined by
\begin{equation*}
\mathcal{M}_{\lambda,I_0}^d(f)(x):=\sup_kE_{k,\lambda,I_0}f(x),
\end{equation*}
where
\begin{equation*}
E_{k,\lambda,I_0}f(x):=\sum_{I\in\mathcal{D}_k|_{I_0}}
\Big(\frac{1}{\nu_\lambda(I)}\int_I|f|d\nu_\lambda\Big)\chi_{I}(x).
\end{equation*}
\end{definition}

\begin{lemma}\label{reverse holder for M}
Let $w\in\widetilde{A}_{\infty,\lambda}$ with $\lambda>-1/2,\neq0$ and $I_0\subset\mathbb{R}_+$ be an interval. Then for any $\epsilon\in\Big(0,\frac{1}{2^{2\lambda+3}[w]_{\widetilde{A}_{\infty,\lambda}}-1}\Big]$,
\begin{align*}
\frac{1}{\nu_\lambda(I_0)}\int_{I_0}[\mathcal{M}_{\lambda,I_0}^d(wv_\lambda^{-1}\chi_{I_0})(x)]
^{1+\epsilon}d\nu_\lambda(x)\leq2[w]_{\widetilde{A}_{\infty,\lambda}}
\Big(\frac{1}{\nu_\lambda(I_0)}\int_{I_0}w(x)dx\Big)^{1+\epsilon}.
\end{align*}
\end{lemma}

\begin{proof}
Since all computations are localized to $I_0$, we can assume $w$ is supported on $I_0$. For brevity, denote
\begin{align*}
\Omega_t:=\{x\in I_0:\mathcal{M}_{\lambda,I_0}^d(wv_\lambda^{-1})(x)>t\}.
\end{align*}
It is direct that
\begin{align*}
&\int_{I_0}\mathcal{M}_{\lambda,I_0}^d(wv_\lambda^{-1})(x)^{1+\epsilon}d\nu_\lambda(x)\\
&\quad=\int_{I_0}\mathcal{M}_{\lambda,I_0}^d(wv_\lambda^{-1})(x)^{\epsilon}
\mathcal{M}_{\lambda,I_0}^d(wv_\lambda^{-1})(x)d\nu_\lambda(x)\\
&\quad=\int_0^\infty\epsilon t^{\epsilon-1}\Big(\int_{\Omega_t}\mathcal{M}_{\lambda,I_0}^d(wv_\lambda^{-1})(x)d\nu_\lambda(x)\Big)dt\\
&\quad\leq\int_0^{(wv_\lambda^{-1})_{I_0,\nu_\lambda}}\epsilon t^{\epsilon-1}\Big(\int_{I_0}\mathcal{M}_{\lambda,I_0}^d(wv_\lambda^{-1})(x)d\nu_\lambda(x)\Big)dt\\
&\qquad+
\int_{(wv_\lambda^{-1})_{I_0,\nu_\lambda}}^\infty\epsilon t^{\epsilon-1}\Big(\int_{\Omega_t}\mathcal{M}_{\lambda,I_0}^d(wv_\lambda^{-1})(x)d\nu_\lambda(x)\Big)dt\\
&\quad=:E+F.
\end{align*}

For the term $E$, since $w$ is supported on $I_0$, by Definition \ref{A}, it is easy to see that
\begin{align*}
E\leq[w]_{\widetilde{A}_{\infty,\lambda}}\Big(\frac{1}{\nu_\lambda(I_0)}\int_{I_0}w(x)dx\Big)^\epsilon w(I_0).
\end{align*}

Next, we tackle $F$. For
\begin{align*}
\frac{1}{\nu_\lambda(I_0)}\int_{I_0}w(x)dx\leq t,
\end{align*}
there exists a collection of maximal disjoint dyadic intervals $I_j\subset I_0$ such that $\Omega_t=\cup_jI_j$ and
\begin{align*}
\frac{1}{\nu_\lambda(I_j)}\int_{I_j}w(x)dx\geq t.
\end{align*}
To proceed with the proof, we first claim that for $x\in I_j$,
\begin{align*}
\mathcal{M}_{\lambda,I_0}^d(wv_\lambda^{-1})(x)=\mathcal{M}_{\lambda,I_0}^d(wv_\lambda^{-1}\chi_{I_j})(x).
\end{align*}
It is trivial that $\mathcal{M}_{\lambda,I_0}^d(wv_\lambda^{-1})(x)\geq
\mathcal{M}_{\lambda,I_0}^d(wv_\lambda^{-1}\chi_{I_j})(x)$. We turn to the converse. By maximal of the intervals in $\{I_j\}_j$, for $x\in I_j$,
\begin{align*}
&\mathcal{M}_{\lambda,I_0}^d(wv_\lambda^{-1})(x)\\
&\quad=\sup_k\sum_{I\in\mathcal{D}_k|_{I_0},I\subset I_j}\Big(\frac{1}{\nu_\lambda(I)}\int_{I}w(x)v_\lambda^{-1}(x)d\nu_\lambda(x)\Big)\chi_I(x)\\
&\quad=\sup_k\sum_{I\in\mathcal{D}_k|_{I_0},I\subset I_j}\Big(\frac{1}{\nu_\lambda(I)}\int_{I}w(x)v_\lambda^{-1}(x)\chi_{I_j}(x)d\nu_\lambda(x)\Big)\chi_I(x)\\
&\quad\leq\mathcal{M}_{\lambda,I_0}^d(wv_\lambda^{-1}\chi_{I_j})(x).
\end{align*}
This proves the claim.

Using the claim and Definition \ref{A} again,
\begin{align}\label{estimate of dyadic maximal}
&\int_{I_j}\mathcal{M}_{\lambda,I_0}^d(wv_\lambda^{-1})(x)d\nu_\lambda(x)\\
&\quad=\int_{I_j}\mathcal{M}_{\lambda,I_0}^d(wv_\lambda^{-1}\chi_{I_j})(x)d\nu_\lambda(x)\nonumber\\
&\quad\leq[w]_{\widetilde{A}_{\infty,\lambda}}w(I_j).\nonumber
\end{align}
Note that $\nu_\lambda(\widetilde{I}_j)\leq2^{2\lambda+2}\nu_\lambda(I_j)$, where $\widetilde{I}_j$ is the dyadic parent of $I_j$. From this and the maximality of the interval $I_j$, we have
\begin{align}\label{w transfrom to v}
w(I_j)\leq w(\widetilde{I}_j)\nu_\lambda(\widetilde{I}_j)^{-1}\nu_\lambda(\widetilde{I}_j)\leq t\nu_\lambda(\widetilde{I}_j)\leq2^{2\lambda+2}t\nu_\lambda(I_j).
\end{align}
This, together with \eqref{estimate of dyadic maximal}, allows us to get
\begin{align*}
F&=\int_{(wv_\lambda^{-1})_{I_0,\nu_\lambda}}^\infty\epsilon t^{\epsilon-1}
\Big(\sum_j\int_{I_j}\mathcal{M}_{\lambda,I_0}^d(wv_\lambda^{-1})(x)d\nu_\lambda(x)\Big)dt\\
&\leq2^{2\lambda+2}[w]_{\widetilde{A}_{\infty,\lambda}}\int_{(wv_\lambda^{-1})_{I_0,\nu_\lambda}}^\infty
\epsilon t^{\epsilon}\sum_j\nu_\lambda(I_j)dt\\
&=2^{2\lambda+2}\epsilon[w]_{\widetilde{A}_{\infty,\lambda}}\int_{(wv_\lambda^{-1})_{I_0,\nu_\lambda}}^\infty
t^{\epsilon}\nu_\lambda(\Omega_t)dt\\
&\leq\frac{2^{2\lambda+2}\epsilon[w]_{\widetilde{A}_{\infty,\lambda}}}{1+\epsilon}
\int_{I_0}\mathcal{M}_{\lambda,I_0}^d(wv_\lambda^{-1})(x)^{1+\epsilon}d\nu_\lambda(x).
\end{align*}

In summary, combined with the estimates of $E$, $F$, we obtain
\begin{align*}
&\int_{I_0}\mathcal{M}_{\lambda,I_0}^d(wv_\lambda^{-1})(x)^{1+\epsilon}d\nu_\lambda(x)\\
&\quad\leq[w]_{\widetilde{A}_{\infty,\lambda}}(wv_\lambda^{-1})_{I_0,\nu_\lambda}^\epsilon w(I_0)+
\frac{2^{2\lambda+2}\epsilon[w]_{\widetilde{A}_{\infty,\lambda}}}{1+\epsilon}
\int_{I_0}\mathcal{M}_{\lambda,I_0}^d(wv_\lambda^{-1})(x)^{1+\epsilon}d\nu_\lambda(x).
\end{align*}
Dividing both sides by $\nu_\lambda(I_0)$, using $\epsilon\in\Big(0,\frac{1}{2^{2\lambda+3}[w]_{\widetilde{A}_{\infty,\lambda}}-1}\Big]$, we get
\begin{align*}
&\frac{1}{\nu_\lambda(I_0)}\int_{I_0}\mathcal{M}_{\lambda,I_0}^d(wv_\lambda^{-1})(x)^{1+\epsilon}
d\nu_\lambda(x)\\
&\quad\leq[w]_{\widetilde{A}_{\infty,\lambda}}(wv_\lambda^{-1})_{I_0,\nu_\lambda}^{1+\epsilon}+
\frac{2^{2\lambda+2}\epsilon[w]_{\widetilde{A}_{\infty,\lambda}}}{1+\epsilon}\frac{1}{\nu_\lambda(I_0)}
\int_{I_0}\mathcal{M}_{\lambda,I_0}^d(wv_\lambda^{-1})(x)^{1+\epsilon}d\nu_\lambda(x)\\
&\quad\leq[w]_{\widetilde{A}_{\infty,\lambda}}(wv_\lambda^{-1})_{I_0,\nu_\lambda}^{1+\epsilon}+
\frac{1}{2\nu_\lambda(I_0)}
\int_{I_0}\mathcal{M}_{\lambda,I_0}^d(wv_\lambda^{-1})(x)^{1+\epsilon}d\nu_\lambda(x).
\end{align*}
Thus,
\begin{align*}
\frac{1}{\nu_\lambda(I_0)}\int_{I_0}[\mathcal{M}_{\lambda,I_0}^d(wv_\lambda^{-1}\chi_{I_0})(x)]
^{1+\epsilon}d\nu_\lambda(x)\leq2[w]_{\widetilde{A}_{\infty,\lambda}}
\Big(\frac{1}{\nu_\lambda(I_0)}\int_{I_0}w(x)dx\Big)^{1+\epsilon}.
\end{align*}
\end{proof}

\begin{proof}[Proof of Theorem \ref{Reverse Holder inequality theorem}]
We first prove (1)$\Rightarrow$(2). Let $\Omega_t$ and $I_j$ be stated as in Lemma \ref{reverse holder for M}. Again, we assume that $w$ is supported on $I_0$. Since
\begin{align*}
\int_{I_0}(w(x)v_\lambda(x)^{-1})^{1+\epsilon}d\nu_\lambda(x)&\leq\int_{I_0}
\mathcal{M}_{\lambda,I_0}^d(wv_{\lambda}^{-1})(x)^\epsilon w(x)dx\\
&=\int_0^\infty\epsilon t^{\epsilon-1}w(\Omega_t)dt\\
&=\int_0^{(wv_\lambda^{-1})_{I_0,\nu_\lambda}}\epsilon t^{\epsilon-1}w(\Omega_t)dt+\int_{(wv_\lambda^{-1})_{I_0,\nu_\lambda}}^\infty\epsilon t^{\epsilon-1}w(\Omega_t)dt\\
&=:G+H.
\end{align*}

By a direct computation,
\begin{equation*}
G\leq(wv_{\lambda}^{-1})_{I_0,\nu_\lambda}^\epsilon w(I_0).
\end{equation*}

In view of \eqref{w transfrom to v}, we deduce that
\begin{align*}
H&\leq\int_{(wv_\lambda^{-1})_{I_0,\nu_\lambda}}^\infty\epsilon t^{\epsilon-1}\sum_jw(I_j)dt\\
&\leq\epsilon2^{2\lambda+2}\int_{(wv_\lambda^{-1})_{I_0,\nu_\lambda}}^\infty t^{\epsilon}\sum_j\nu_\lambda(I_j)dt\\
&=\epsilon2^{2\lambda+2}\int_{(wv_\lambda^{-1})_{I_0,\nu_\lambda}}^\infty t^{\epsilon}\nu_\lambda(\Omega_t)dt\\
&\leq\frac{2^{2\lambda+2}\epsilon}{1+\epsilon}\int_{I_0}\mathcal{M}_{\lambda,I_0}^d
(wv_\lambda^{-1})(x)^{1+\epsilon}d\nu_\lambda(x).
\end{align*}

Therefore, from the estimates of $G$ and $H$, we have
\begin{align*}
&\int_{I_0}(w(x)v_\lambda(x)^{-1})^{1+\epsilon}d\nu_\lambda(x)\\
&\quad\leq(wv_{\lambda}^{-1})_{I_0,\nu_\lambda}^\epsilon w(I_0)+\frac{2^{2\lambda+2}\epsilon}{1+\epsilon}\int_{I_0}\mathcal{M}_{\lambda,I_0}^d
(wv_\lambda^{-1})(x)^{1+\epsilon}d\nu_\lambda(x).
\end{align*}
Dividing both sides by $\nu_\lambda(I_0)$, applying Lemma \ref{reverse holder for M} and $\epsilon\in\Big(0,\frac{1}{2^{2\lambda+3}[w]_{\widetilde{A}_{\infty,\lambda}}-1}\Big]$, we conclude that
\begin{align*}
&\frac{1}{\nu_\lambda(I_0)}\int_{I_0}(w(x)v_\lambda(x)^{-1})^{1+\epsilon}d\nu_\lambda(x)\\
&\quad\leq(wv_{\lambda}^{-1})_{I_0,\nu_\lambda}^{1+\epsilon} +\frac{2^{2\lambda+2}\epsilon}{1+\epsilon}\frac{1}{\nu_\lambda(I_0)}\int_{I_0}\mathcal{M}_{\lambda,I_0}^d
(wv_\lambda^{-1})(x)^{1+\epsilon}d\nu_\lambda(x)\\
&\quad\leq\Big(\frac{1}{\nu_\lambda(I_0)}\int_{I_0}w(x)v_{\lambda}(x)^{-1}d\nu_\lambda(x)\Big)^{1+\epsilon}\\
&\qquad+
\frac{2^{2\lambda+3}\epsilon[w]_{\widetilde{A}_{\infty,\lambda}}}{1+\epsilon}
\Big(\frac{1}{\nu_\lambda(I_0)}\int_{I_0}w(x)v_{\lambda}(x)^{-1}d\nu_\lambda(x)\Big)^{1+\epsilon}\\
&\quad\leq2\Big(\frac{1}{\nu_\lambda(I_0)}\int_{I_0}w(x)v_{\lambda}(x)^{-1}d\nu_\lambda(x)\Big)^{1+\epsilon},
\end{align*}
as we wanted.

Next, we prove (2)$\Rightarrow$(1). By H\"{o}lder's inequality and the boundedness of $\mathcal{M}_\lambda$ on $L^{1+\epsilon}(\mathbb{R}_+,d\nu_\lambda)$ (see \cite{CoWe}), we have
\begin{align*}
&\frac{1}{\nu_\lambda(I)}\int_I\mathcal{M}_\lambda(wv_\lambda^{-1}\chi_I)(x)d\nu_\lambda(x)\\
&\quad\leq\Big(\frac{1}{\nu_\lambda(I)}\int_I\mathcal{M}_\lambda(wv_\lambda^{-1}\chi_I)(x)^{1+\epsilon}
d\nu_\lambda(x)\Big)^{1/(1+\epsilon)}\\
&\quad\leq\Big(\frac{1}{\nu_\lambda(I)}\int_{\mathbb{R}_+}\mathcal{M}_\lambda(wv_\lambda^{-1}\chi_I)(x)^
{1+\epsilon}
d\nu_\lambda(x)\Big)^{1/(1+\epsilon)}\\
&\quad\lesssim\Big(\frac{1}{\nu_\lambda(I)}\int_{I}(w(x)v_\lambda(x)^{-1})^{1+\epsilon}
d\nu_\lambda(x)\Big)^{1/(1+\epsilon)}\\
&\quad\lesssim\frac{1}{\nu_\lambda(I)}\int_{I}w(x)v_\lambda(x)^{-1}d\nu_\lambda(x)=v_\lambda(I)^{-1}w(I).
\end{align*}
Rearrange the above inequality, and then take the supremum over all intervals $I\subset\mathbb{R}_+$, we get
\begin{align*}
\sup_{I\subset\mathbb{R}_+}\frac{1}{w(I)}
\int_I\mathcal{M}_\lambda(wv_\lambda^{-1}\chi_I)(x)d\nu_\lambda(x)<\infty.
\end{align*}
\end{proof}

The following corollary improves the $\widetilde{A}_{p,\lambda}$ condition in \cite{LiLiangLinShen} to $\widetilde{A}_{\infty,\lambda}$.
\begin{corollary}\label{Ratio property}
Let $w\in\widetilde{A}_{\infty,\lambda}$ with $\lambda>-1/2,\neq0$. Then there exists $c_\lambda>0$ such that for each interval $I\subset\mathbb{R}_+$ and each measurable subset $J\subset I$,
\begin{align*}
\frac{w(J)}{w(I)}\leq2\Big(\frac{\nu_\lambda(J)}{\nu_\lambda(I)}\Big)^
{\frac{1}{c_\lambda[w]_{\widetilde{A}_{\infty,\lambda}}}}.
\end{align*}
\end{corollary}
\begin{proof}
Let $r_{w,\lambda}:=1+\frac{1}{c_\lambda[w]_{\widetilde{A}_{\infty,\lambda}}}$. By H\"{o}lder's inequality and \eqref{Reverse Holder inequality}, we have
\begin{align*}
w(J)&=\nu_\lambda(I)\frac{1}{\nu_\lambda(I)}\int_Iw(x)v_\lambda(x)^{-1}\chi_J(x)d\nu_\lambda(x)\\
&\leq\nu_\lambda(I)
\Big(\frac{1}{\nu_\lambda(I)}\int_I(w(x)v_\lambda(x)^{-1})^{r_{w,\lambda}}d\nu_\lambda(x)\Big)^{1/r_{w,\lambda}}
\Big(\frac{1}{\nu_\lambda(I)}\int_I\chi_J(x)d\nu_\lambda(x)\Big)^{1/r_{w,\lambda}'}\\
&\leq2\nu_\lambda(I)\Big(\frac{1}{\nu_\lambda(I)}\int_Iw(x)dx\Big)
\Big(\frac{\nu_\lambda(J)}{\nu_\lambda(I)}\Big)^{\frac{1}{r_{w,\lambda}'}}.
\end{align*}

Note that $r_{w,\lambda}'\sim c_\lambda[w]_{\widetilde{A}_{\infty,\lambda}}$, we get
\begin{align*}
\frac{w(J)}{w(I)}\leq2\Big(\frac{\nu_\lambda(J)}{\nu_\lambda(I)}\Big)^
{\frac{1}{r_{w,\lambda}'}}\leq2\Big(\frac{\nu_\lambda(J)}{\nu_\lambda(I)}\Big)^
{\frac{1}{c_\lambda[w]_{\widetilde{A}_{\infty,\lambda}}}}.
\end{align*}
\end{proof}

We need the following lemma to prove the characterization theorem for $\widetilde{A}_{\infty,\lambda}$. The proof is almost the same as in \cite[Corollary 7.2.4]{Gra}, we omit the details.
\begin{lemma}\label{geneal reverse Holder}
Let $w$ be a weight and $\mu$ be a doubling measure on $\mathbb{R}_+$. Assume that there are $\alpha,\beta\in(0,1)$ such that
\begin{align*}
u(J)\leq\alpha u(I)\Rightarrow\int_Jw(t)du(t)\leq\beta\int_Iw(t)du(t),
\end{align*}
where $J$ is an arbitrary $\mu$-measurable subset of an interval $I\subset\mathbb{R}_+$. Then there exist $\gamma>0$ such that for any interval $I\subset\mathbb{R}_+$,
\begin{align*}
\Big(\frac{1}{u(I)}\int_Iw(t)^{1+\gamma}du(t)\Big)^{\frac{1}{1+\gamma}}\lesssim
\frac{1}{u(I)}\int_Iw(t)du(t).
\end{align*}
\end{lemma}

\begin{proof}[Proof of Theorem \ref{characterization A}]
We prove Theorem \ref{characterization A} by going:
\begin{align*}
(1)\Rightarrow(2)\Rightarrow(3)\Rightarrow(4)\Rightarrow(5)\Rightarrow(1).
\end{align*}
(1)$\Rightarrow$(2) and (5)$\Rightarrow$(1) have been proved in Theorem \ref{Ratio property} and Theorem \ref{Reverse Holder inequality theorem}, respectively. (4)$\Rightarrow$(5) is proved in \cite{LiLiangLinShen}. It remains to prove (2)$\Rightarrow$(3) and (3)$\Rightarrow$(4).

We show (2)$\Rightarrow$(3) first. Choose an $\alpha'\in(0,1)$ small enough such that $\beta'=c_1\alpha'^{c_2}<1$. By (2), we have
\begin{align*}
\nu_\lambda(J)<\alpha'\nu_\lambda(I)\Rightarrow w(J)<\beta'w(I)
\end{align*}
for any intervals $I\subset\mathbb{R}_+$ and any measurable set $J\subset I$. Replacing $J$ by $I\backslash J$, it follows that
\begin{align*}
\nu_\lambda(J)\geq(1-\alpha')\nu_\lambda(I)\Rightarrow w(J)\geq(1-\beta')w(I),
\end{align*}
which implies that
\begin{align*}
w(J)<(1-\beta')w(I)\Rightarrow \nu_\lambda(J)<(1-\alpha')\nu_\lambda(I)
\end{align*}
by the contraposition. This proves (3).

Next, we show (3)$\Rightarrow$(4). Define the measure $du(t)=w(t)dt$. Firstly, we show $u$ is doubling. Note that
\begin{align}\label{contraposition of (3)}
\nu_\lambda(J)\geq\beta\nu_\lambda(I)\Rightarrow w(J)\geq\alpha w(I).
\end{align}
Hence, by the doubling property of $\nu_\lambda$:
\begin{equation*}
\nu_\lambda(2I)\leq 2^{2\lambda+2}\nu_\lambda(I),
\end{equation*}
and replacing $J$ and $I$ by $I$ and $2I$ in \eqref{contraposition of (3)}, respectively. We prove that $u$ is doubling by \eqref{contraposition of (3)}.

Observe that the condition (3) can be reformulated as
\begin{align*}
u(J)<\alpha u(I)\Rightarrow\int_J(wv_\lambda^{-1})^{-1}du(t)<\int_I(wv_\lambda^{-1})^{-1}du(t).
\end{align*}
This, together with the doubling property of $u$, shows that $(wv_\lambda^{-1})^{-1}$ satisfies a reverse H\"{o}lder estimate with respect to the measure $u$ by Lemma \ref{geneal reverse Holder}, that is, there exists $\gamma>0$ such that for any $I\subset \mathbb{R}_+$,
\begin{align*}
\Big(\frac{1}{w(I)}\int_I(w(t)v_\lambda(t)^{-1})^{-1-\gamma}w(t)dt\Big)^{\frac{1}{1+\gamma}}\lesssim
\frac{1}{w(I)}\int_I(w(t)v_\lambda(t)^{-1})^{-1}w(t)dt,
\end{align*}
which is equivalent to
\begin{align*}
\Big(\frac{1}{\nu_\lambda(I)}\int_Iw(t)dt\Big)^{\frac{\gamma}{\gamma+1}}
\Big(\frac{1}{\nu_\lambda(I)}\int_Iv_\lambda(t)^{1+\lambda}w(t)^{-\gamma}dt\Big)^{\frac{1}{1+\gamma}}
\lesssim1.
\end{align*}
Denote $p=1+(1/\gamma)$. Then we proved that for any $I\subset\mathbb{R}_+$,
\begin{align*}
\Big(\frac{1}{\nu_\lambda(I)}\int_Iw(t)dt\Big)^{1/p}
\Big(\frac{1}{\nu_\lambda(I)}\int_Iv_\lambda(t)^{p'}w(t)^{1-p'}dt\Big)^{1/p'}\lesssim1.
\end{align*}
This proves $w\in\widetilde{A}_{p,\lambda}$.
\end{proof}

Now, we undertake the proof of sharp weighted bounds of $\mathcal{M}_\lambda$. We first need to prove Theorem \ref{open property of weight}.
\begin{proof}[Proof of Theorem \ref{open property of weight}]
Note that $1-(p-\epsilon)'=(1-p')r_\sigma$ and
\begin{align*}
0<r_\sigma:=1+\frac{1}{2^{2\lambda+3}[\sigma]_{\widetilde{A}_{\infty,\lambda}}}
<1+\frac{1}{2^{2\lambda+3}[\sigma]_{\widetilde{A}_{\infty,\lambda}}-1}.
\end{align*}
Then for any interval $I\subset\mathbb{R}_+$, by Theorem \ref{Reverse Holder inequality},
\begin{align*}
&\Big(\frac{1}{\nu_\lambda(I)}\int_Iv_\lambda(t)^{(p-\epsilon)'}w(t)^{1-(p-\epsilon)'}dt\Big)^{p-\epsilon-1}\\
&\quad=\Big(\frac{1}{\nu_\lambda(I)}\int_I(w(t)v_\lambda(t)^{-1})^{(1-p')r_\sigma}d\nu_\lambda(t)\Big)
^{\frac{p-1}{r_\sigma}}\\
&\quad=\Big(\frac{1}{\nu_\lambda(I)}\int_I(w(t)^{1-p'}v_\lambda(t)^{p'}v_\lambda(t)^{-1})^{r_\sigma}
d\nu_\lambda(t)\Big)^{\frac{p-1}{r_\sigma}}\\
&\quad\leq\Big(\frac{2}{\nu_\lambda(I)}\int_Iv_\lambda(t)^{p'}w(t)^{1-p'}dt\Big)^{p-1}.
\end{align*}
Hence,
\begin{align*}
&\Big(\frac{1}{\nu_\lambda(I)}\int_Iw(t)dt\Big)
\Big(\frac{1}{\nu_\lambda(I)}\int_Iv_\lambda(t)^{(p-\epsilon)'}w(t)^{1-(p-\epsilon)'}dt\Big)^{p-\epsilon-1}\\
&\quad\leq2^{p-1}\Big(\frac{1}{\nu_\lambda(I)}\int_Iw(t)dt\Big)
\Big(\frac{1}{\nu_\lambda(I)}\int_Iv_\lambda(t)^{p'}w(t)^{1-p'}dt\Big)^{p-1}\\
&\quad\leq2^{p-1}[w]_{\widetilde{A}_{p,\lambda}}.
\end{align*}
Taking the supremum over all intervals $I\subset\mathbb{R}_+$, we get the desired result.
\end{proof}

Analogous to the classical weight case, we can also define another $\widetilde{A}_{\infty,\lambda}$ constant.
\begin{definition}
Let $\lambda>-1/2,\neq0$ and $v_\lambda(t):=t^{2\lambda+1}$. We define
\begin{align*}
[w]_{\widetilde{A}_{\infty,\lambda}}':=\sup_{I\subset\mathbb{R}_+}\Big(\frac{1}{\nu_\lambda(I)}
\int_Iw(t)dt\Big)\Big(e^{\frac{1}{\nu_\lambda(I)}\int_I
\log(w(t)v_\lambda(t)^{-1})^{-1}d\nu_\lambda(t)}\Big)<\infty.
\end{align*}
\end{definition}

The following lemma says that $[w]_{\widetilde{A}_{\infty,\lambda}}$ is better than $[w]_{\widetilde{A}_{\infty,\lambda}}'$, and $[w]_{\widetilde{A}_{\infty,\lambda}}'$ serve as an intermediate bridge, through which the quantitative relationship between $[w]_{\widetilde{A}_{\infty,\lambda}}$ and $[w]_{\widetilde{A}_{p,\lambda}}$ can be obtained.
\begin{lemma}\label{A compare}
Let $1\leq p<\infty$ and $\lambda>-1/2,\neq0$. Then $[w]_{\widetilde{A}_{\infty,\lambda}}\lesssim[w]_{\widetilde{A}_{\infty,\lambda}}'\leq
[w]_{\widetilde{A}_{p,\lambda}}$.
\end{lemma}
\begin{proof}
Fix an interval $I\subset\mathbb{R}_+$. We first show $[w]_{\widetilde{A}_{\infty,\lambda}}'\leq
[w]_{\widetilde{A}_{p,\lambda}}$. It suffices to consider the case $1<p<\infty$ since $[w]_{\widetilde{A}_{p,\lambda}}\leq[w]_{\widetilde{A}_{1,\lambda}}$. By Jesen's inequality,
\begin{align*}
\Big(\int_I(wv_\lambda^{-1})^{-\frac{1}{p-1}}\frac{d\nu_\lambda}{\nu_\lambda(I)}\Big)^{p-1}\geq e^{\int_I\log(wv_\lambda^{-1})^{-1}\frac{d\nu_\lambda}{\nu_\lambda(I)}},
\end{align*}
so
\begin{align*}
&\Big(\frac{1}{\nu_\lambda(I)}\int_Iw(t)dt\Big)
e^{\int_I\log(wv_\lambda^{-1})^{-1}\frac{d\nu_\lambda}{\nu_\lambda(I)}}\\
&\quad\leq\Big(\frac{1}{\nu_\lambda(I)}\int_Iw(t)dt\Big)
\Big(\int_I(wv_\lambda^{-1})^{-\frac{1}{p-1}}\frac{d\nu_\lambda}{\nu_\lambda(I)}\Big)^{p-1}\\
&\quad=\Big(\frac{1}{\nu_\lambda(I)}\int_Iw(t)dt\Big)
\Big(\frac{1}{\nu_\lambda(I)}\int_Iv_\lambda(t)^{p'}w(t)^{-\frac{1}{p-1}}dt\Big)^{p-1}.
\end{align*}
Taking the supremum over all $I\subset\mathbb{R}_+$, we get $[w]_{\widetilde{A}_{\infty,\lambda}}'\leq
[w]_{\widetilde{A}_{p,\lambda}}$.

Next, we prove $[w]_{\widetilde{A}_{\infty,\lambda}}\lesssim[w]_{\widetilde{A}_{\infty,\lambda}}'$. Let $I,R$ be intervals in $\mathbb{R}_+$. From the definition of $[w]_{\widetilde{A}_\infty,\lambda}'$,
\begin{align}\label{compare orgin}
\frac{1}{\nu_\lambda(R)}\int_Rwv_\lambda d\nu_\lambda\leq[w]_{\widetilde{A}_\infty,\lambda}'e^{\frac{1}{\nu_\lambda(R)}\int_R
\log(wv_\lambda^{-1})d\nu_\lambda}.
\end{align}
By the fact
\begin{align*}
\mathcal{M}_\lambda(wv_\lambda^{-1}\chi_I)(x)=\sup_{R\ni x,R\subset I}\frac{1}{\nu_\lambda(R)}\int_Rwv_\lambda^{-1} d\nu_\lambda,~x\in I.
\end{align*}
and taking the supremum over $R$ on both sides of \eqref{compare orgin}, we have
\begin{align}\label{HL compare}
\mathcal{M}_\lambda(wv_\lambda^{-1}\chi_I)(x)\leq[w]_{\widetilde{A}_\infty,\lambda}'
\widetilde{\mathcal{M}}_{\lambda}(wv_\lambda^{-1}\chi_I)(x),~x\in I,
\end{align}
where
\begin{align*}
\widetilde{\mathcal{M}}_{\lambda}(f):=\sup_I
e^{\frac{1}{\nu_\lambda(I)}\int_I\log(|f|)}\chi_I.
\end{align*}
We claim that $\|\widetilde{\mathcal{M}}_\lambda f\|_{L^1(\mathbb{R}_+,d\nu_\lambda)}\lesssim\|f\|_{L^1(\mathbb{R}_+,d\nu_\lambda)}$. In fact, by the Jesen inequality,
\begin{align*}
\int_I|f|\frac{d\nu_\lambda}{\nu_\lambda(I)}\geq e^{\int_I\log|f|\frac{d\nu_\lambda}{\nu_\lambda(I)}},
\end{align*}
then $\widetilde{\mathcal{M}_\lambda} f\leq \mathcal{M}_\lambda f$. Observe that $\widetilde{\mathcal{M}_\lambda} f=[\widetilde{\mathcal{M}}_\lambda(|f|^{1/2})]^2$, so
\begin{align*}
\int_{\mathbb{R}_+}\widetilde{\mathcal{M}_\lambda} fd\nu_\lambda&=\int_{\mathbb{R}_+}\widetilde{\mathcal{M}_\lambda}(|f|^{1/2})^2d\nu_\lambda\\
&\leq\int_{\mathbb{R}_+}\mathcal{M}_\lambda(|f|^{1/2})^2d\nu_\lambda\\
&\lesssim\int_{\mathbb{R}_+}|f|d\nu_\lambda.
\end{align*}
This proves the claim.

By this claim and \eqref{HL compare}, we deduce that
\begin{align*}
\int_I\mathcal{M}_\lambda(wv_\lambda^{-1}\chi_I)d\nu_\lambda&\leq[w]_{\widetilde{A}_{\infty,\lambda}}'\int_I
\widetilde{\mathcal{M}}_\lambda(wv_\lambda^{-1}\chi_I)d\nu_\lambda\\
&\lesssim[w]_{\widetilde{A}_{\infty,\lambda}}'\int_I wv_\lambda^{-1} d\nu_\lambda\\
&=[w]_{\widetilde{A}_{\infty,\lambda}}'w(I),
\end{align*}
which gives $[w]_{\widetilde{A}_{\infty,\lambda}}\lesssim[w]_{\widetilde{A}_{\infty,\lambda}}'$.
\end{proof}

We end this section by giving the proof of Theorem \ref{mixed quantitative weighted strong estimates}.
\begin{proof}[Proof of Theorem \ref{mixed quantitative weighted strong estimates}]
Observe that
\begin{align*}
\{x\in\mathbb{R}_+:\mathcal{M}_\lambda f(x)>2t\}\subset\{x\in\mathbb{R}_+:\mathcal{M}_\lambda(f\chi_{|f|>t})(x)>t\},~t>0.
\end{align*}
Then
\begin{align}\label{distribution represention of Lp of M}
\|\mathcal{M}_\lambda f\|_{L^p(\mathbb{R}_+,wdx)}^p&=p\int_0^\infty t^{p-1}w(\{x\in\mathbb{R}_+:\mathcal{M}_\lambda f(x)>t\})dt\\
&=p2^p\int_0^\infty t^{p-1}w(\{x\in\mathbb{R}_+:\mathcal{M}_\lambda f(x)>2t\})dt\nonumber\\
&\leq p2^p\int_0^\infty t^{p-1}w(\{x\in\mathbb{R}_+:\mathcal{M}_\lambda(f\chi_{|f|>t})(x)>t\})dt.\nonumber
\end{align}

Now we temporarily stop the estimation of \eqref{distribution represention of Lp of M}. By Theorem \ref{open property of weight}, $w\in \widetilde{A}_{p-\epsilon,\lambda}$, we turn to prove
\begin{align}\label{quantitative weak type estimate of HLO}
\|\mathcal{M}_\lambda f\|_{L^{p-\epsilon,\infty}(\mathbb{R}_+,wdx)}\leq 5^{2\lambda+2}[w]_{\widetilde{A}_{p-\epsilon,\lambda}}^{\frac{1}{p-\epsilon}}\|f\|_{L^{p-\epsilon}(\mathbb{R}_+,wdx)}.
\end{align}
Denote
\begin{equation*}
E_t:=\{x\in\mathbb{R}_+: \mathcal{M}_\lambda f(x)>t\}.
\end{equation*}
Let $F$ be a compact subset of $E_t$. For any $x\in F$, one can choose an open interval $I_x\ni x$ for which the following inequality holds:
\begin{align*}
\frac{1}{\nu_\lambda(I_x)}\int_{I_x}|f(y)|d\nu_\lambda(y)>t.
\end{align*}
Due to the compactness of $F$, there is a finite subcover $\{I_{x_1},\cdots,I_{x_k}\}$ of $F$. In virtue of the well known one-dimensional Vitali covering lemma, we find a subcollection of mutually disjoint intervals $I_{x_{j_1}},\cdots,I_{x_{j_l}}$ such that $F\subset\bigcup_{i=1}^{l}5I_{x_{j_i}}$. As a consequence of H\"{o}lder's inequality and $w\in\widetilde{A}_{p-\epsilon,\lambda}$, we deduce that
\begin{align*}
w(F)&\leq\sum_{i=1}^{l}w(5I_{x_{j_i}})\\
&\leq t^{-(p-\epsilon)}\sum_{i=1}^lw(5I_{x_{j_i}})\Big(\frac{1}{\nu_{\lambda}(I_{x_{j_i}})}\int_{I_{x_{j_i}}}
|f(y)|y^{2\lambda+1}dy\Big)^{p-\epsilon}\\
&=t^{-(p-\epsilon)}\sum_{i=1}^lw(5I_{x_{j_i}})\Big(\frac{\nu_{\lambda}(5I_{x_{j_i}})}
{\nu_{\lambda}(I_{x_{j_i}})}\Big)^{p-\epsilon}
\Big(\frac{1}{\nu_{\lambda}(5I_{x_{j_i}})}
\int_{I_{x_{j_i}}}
|f(y)|y^{2\lambda+1}dy\Big)^{p-\epsilon}\\
&\leq5^{(2\lambda+2)(p-\epsilon)} t^{-(p-\epsilon)}\sum_{i=1}^l
\left(\frac{\int_{5I_{x_{j_i}}}v_\lambda(y)^{(p-\epsilon)'}
w(y)^{-\frac{(p-\epsilon)'}{p-\epsilon}}dy}{\nu_{\lambda}(5I_{x_{j_i}})}\right)^{p-\epsilon-1}
\\
&\qquad\qquad\qquad\qquad\qquad\times\frac{w(5I_{x_{j_i}})}{\nu_\lambda(5I_{x_{j_i}})}
\|f\chi_{I_{x_{j_i}}}\|_{L^{p-\epsilon}(\mathbb{R}_+,wdx)}^{p-\epsilon}\\
&\leq5^{(2\lambda+2)(p-\epsilon)} t^{-(p-\epsilon)}[w]_{\widetilde{A}_{p-\epsilon,\lambda}}\sum_{i=1}^l
\|f\chi_{I_{x_{j_i}}}\|_{L^{p-\epsilon}(\mathbb{R}_+,wdx)}^{p-\epsilon}\\
&\leq 5^{(2\lambda+2)(p-\epsilon)} t^{-(p-\epsilon)}[w]_{\widetilde{A}_{p-\epsilon,\lambda}}\|f\|_{L^{p-\epsilon}
(\mathbb{R}_+,wdx)}^{p-\epsilon}.
\end{align*}
By taking the supremum over all compact $F\subset E_t$, we obtain \eqref{quantitative weak type estimate of HLO}.

We return to estimate \eqref{distribution represention of Lp of M}. By \eqref{quantitative weak type estimate of HLO},
\begin{align*}
&t^{p-\epsilon}w(\{x\in\mathbb{R}_+:\mathcal{M}_\lambda(f\chi_{|f|>t})(x)>t\})\\
&\quad\leq5^{(2\lambda+2)(p-\epsilon)}[w]_{\widetilde{A}_{p-\epsilon,\lambda}}
\int_{\mathbb{R}_+}|f(x)|^{p-\epsilon}\chi_{\{x:|f(x)|>t\}}w(x)dx.
\end{align*}
From this and Theorem \ref{open property of weight}, we have
\begin{align*}
&\int_0^\infty t^{p-1}w(\{x\in\mathbb{R}_+:\mathcal{M}_\lambda(f\chi_{|f|>t})(x)>t\})dt\\
&\quad=\int_0^\infty t^{\epsilon-1}t^{p-\epsilon}w(\{x\in\mathbb{R}_+:\mathcal{M}_\lambda(f\chi_{|f|>t})(x)>t\})dt\\
&\quad\leq5^{(2\lambda+2)(p-\epsilon)}[w]_{\widetilde{A}_{p-\epsilon,\lambda}}\int_0^\infty t^{\epsilon-1}\int_{\{x:|f(x)|>t\}}|f(x)|^{p-\epsilon}w(x)dxdt\\
&\quad= 5^{(2\lambda+2)(p-\epsilon)}\frac{[w]_{\widetilde{A}_{p-\epsilon,\lambda}}}{\epsilon}
\int_{\mathbb{R}_+}|f(x)|^pw(x)dx\\
&\quad\leq\frac{5^{(4\lambda+6)p}2^{p-1}}{p-1}[v_\lambda^{p'}w^{1-p'}]_{\widetilde{A}_{\infty,\lambda}}
[w]_{\widetilde{A}_{p,\lambda}}
\|f\|_{L^p(\mathbb{R}_+,wdx)}^p.
\end{align*}
Inserting this estimate into \eqref{distribution represention of Lp of M}, we get
\begin{align*}
\|\mathcal{M}_\lambda f\|_{L^p(\mathbb{R}_+,wdx)}\lesssim([w]_{\widetilde{A}_{p,\lambda}}
[v_\lambda^{p'}w^{1-p'}]_{\widetilde{A}_{\infty,\lambda}})^{\frac{1}{p}}\|f\|_{L^p(\mathbb{R}_+,wdx)}.
\end{align*}
Note that $v_\lambda^{p'}w^{1-p'}\in \widetilde{A}_{p',\lambda}$ and $[v_\lambda^{p'}w^{1-p'}]_{\widetilde{A}_{p',\lambda}}=
[w]_{\widetilde{A}_{p,\lambda}}^{1/(p-1)}$, by Lemma \ref{A compare}, we get
\begin{align*}
\|\mathcal{M}_\lambda f\|_{L^p(\mathbb{R}_+,wdx)}\lesssim([w]_{\widetilde{A}_{p,\lambda}}
[v_\lambda^{p'}w^{1-p'}]_{\widetilde{A}_{\infty,\lambda}})^{\frac{1}{p}}\|f\|_{L^p(\mathbb{R}_+,wdx)}\leq
[w]_{\widetilde{A}_{p,\lambda}}^{\frac{1}{p-1}}\|f\|_{L^p(\mathbb{R}_+,wdx)}.
\end{align*}

Finally, we prove that the exponent $1/(p-1)$ is sharp. Let $x\in\mathbb{R}_+$, $\epsilon\in(0,1)$. Denote $w_\epsilon(x):=x^{(2\lambda+2)p-1-\epsilon}$ and $f_\epsilon(x):=x^{-(2\lambda+2)+\frac{\epsilon}{p-1}}\chi_{(0,1)}(x)$. Then
\begin{align*}
\|f_\epsilon\|_{L^p(\mathbb{R}_+,w_\epsilon dx)}&=\int_0^1x^{-p(2\lambda+2)+\frac{\epsilon p}{p-1}}x^{(2\lambda+2)p-1-\epsilon}dx\\
&=\int_0^1x^{\frac{\epsilon}{p-1}-1}dx=\frac{p-1}{\epsilon},
\end{align*}
so $f_\epsilon\in L^p(\mathbb{R}_+,w_\epsilon dx)$. For $x\in(0,1)$, one can check that
\begin{align*}
\mathcal{M}_\lambda f_\epsilon(x)&\geq\frac{1}{\int_0^xy^{2\lambda+1}dy}\int_0^xy^{-(2\lambda+2)+\frac{\epsilon}{p-1}}
y^{2\lambda+1}dy\\
&=\frac{2\lambda+2}{x^{2\lambda+2}}\int_0^xy^{\frac{\epsilon}{p-1}-1}dy\\
&=\frac{(2\lambda+2)(p-1)}{\epsilon}x^{-(2\lambda+2)+\frac{\epsilon}{p-1}}=\frac{(2\lambda+2)(p-1)}{\epsilon}
f_\epsilon(x),
\end{align*}
which implies that
\begin{align}\label{lower bound of M}
\|\mathcal{M}_\lambda f_\epsilon\|_{L^p(\mathbb{R}_+,w_\epsilon dx)}\geq\frac{(2\lambda+2)(p-1)}{\epsilon}\|f_\epsilon\|_{L^p(\mathbb{R}_+,w_\epsilon dx)}.
\end{align}

In the following, we prove that
\begin{align}\label{weight constant of test weight}
[w_\epsilon]_{\widetilde{A}_{p,\lambda}}^{\frac{1}{p-1}}=
\Big(\frac{2\lambda+2}{(2\lambda+2)p-\epsilon}\Big)^{\frac{1}{p-1}}\frac{(2\lambda+2)(p-1)}{\epsilon}.
\end{align}
Fix an interval $I:=(a,b)\subset\mathbb{R}_+$. One can check that
\begin{align*}
\frac{1}{\nu_\lambda(I)}\int_Iw_\epsilon(x)dx&=\frac{2\lambda+2}{b^{2\lambda+2}-a^{2\lambda+2}}
\frac{b^{(2\lambda+2)p-\epsilon}-a^{(2\lambda+2)p-\epsilon}}{(2\lambda+2)p-\epsilon}\\
&=\frac{1}{b^{2\lambda+2}(1-\frac{a^{2\lambda+2}}{b^{2\lambda+2}})}
\frac{b^{(2\lambda+2)p-\epsilon}
[1-(\frac{a^{2\lambda+2}}{b^{2\lambda+2}})^{p-\frac{\epsilon}{2\lambda+2}}]}
{p-\frac{\epsilon}{2\lambda+2}}\\
&=b^{(2\lambda+2)(p-1-\frac{\epsilon}{2\lambda+2})}
\frac{[1-(\frac{a^{2\lambda+2}}{b^{2\lambda+2}})^{p-\frac{\epsilon}{2\lambda+2}}]}
{(p-\frac{\epsilon}{2\lambda+2})(1-\frac{a^{2\lambda+2}}{b^{2\lambda+2}})}
\end{align*}
and
\begin{align*}
\frac{1}{\nu_\lambda(I)}\int_Ix^{(2\lambda+1)p'}w_\epsilon(x)^{1-p'}dx&=
\frac{2\lambda+2}{b^{2\lambda+2}-a^{2\lambda+2}}\int_a^bx^{\epsilon p'-\epsilon-1}dx\\
&=\frac{2\lambda+2}{b^{2\lambda+2}-a^{2\lambda+2}}\frac{b^{\epsilon/(p-1)}-a^{\epsilon/(p-1)}}{\epsilon/(p-1)}\\
&=\frac{b^{-(2\lambda+2)\big[1-\frac{\epsilon}{(2\lambda+2)(p-1)}\big]}}{1-\frac{a^{2\lambda+2}}{b^{2\lambda+2}}}
\frac{[1-(\frac{a^{2\lambda+2}}{b^{2\lambda+2}})^{\frac{\epsilon}{(2\lambda+2)(p-1)}}]}
{\epsilon/[(2\lambda+2)(p-1)]}.
\end{align*}
For brevity, denote $\alpha:=a^{2\lambda+2}/b^{2\lambda+2}$, $\beta:=p-1-(\epsilon/(2\lambda+2))$. We rewrite
\begin{align*}
\frac{1}{\nu_\lambda(I)}\int_Iw_\epsilon(x)dx=b^{(2\lambda+2)\beta}\frac{1-\alpha^{1+\beta}}
{(1+\beta)(1-\alpha)},
\end{align*}
\begin{align*}
\frac{1}{\nu_\lambda(I)}\int_Ix^{(2\lambda+1)p'}w_\epsilon(x)^{1-p'}dx=b^{-\frac{(2\lambda+2)\beta}{p-1}}
\frac{1-\alpha^{1-\frac{\beta}{p-1}}}{(1-\alpha)(1-\frac{\beta}{p-1})}.
\end{align*}
Hence,
\begin{align}\label{represention of weight constant}
&\Big(\frac{1}{\nu_\lambda(I)}\int_Iw_\epsilon(x)dx\Big)\Big(\frac{1}{\nu_\lambda(I)}\int_I
x^{(2\lambda+1)p'}w_\epsilon(x)^{1-p'}dx\Big)^{p-1}\\
&\quad=\frac{1-\alpha^{1+\beta}}{(1+\beta)(1-\alpha)}
\Big[\frac{1-\alpha^{1-\frac{\beta}{p-1}}}{(1-\alpha)(1-\frac{\beta}{p-1})}\Big]^{p-1}\nonumber.
\end{align}
When $I=(0,b)$, then $\alpha=0$. It is trivial that
\begin{align}\label{lower bound of weight constant}
[w_\epsilon]_{\widetilde{A}_{p,\lambda}}^{\frac{1}{p-1}}&\geq\Big(\frac{1}{1+\beta}\Big)^{\frac{1}{p-1}}
\Big(\frac{1}{1-\frac{\beta}{p-1}}\Big)\\
&=\left(\frac{2\lambda+2}{(2\lambda+2)p-\epsilon}\right)^{\frac{1}{p-1}}\frac{(2\lambda+2)(p-1)}{\epsilon}.\nonumber
\end{align}
We turn to the upper bound of $[w_\epsilon]_{\widetilde{A}_{p,\lambda}}^{\frac{1}{p-1}}$. Consider the function
\begin{equation*}
f(t):=\log(1-\alpha^t),~0<\alpha<1,~t>0.
\end{equation*}
Note that $f''(t)=-\frac{(\log \alpha)^2\alpha^t}{(1-\alpha^t)^2}<0$ and
\begin{align*}
\frac{1}{p}(1+\beta)+\frac{p-1}{p}\Big(1-\frac{\beta}{p-1}\Big)=1.
\end{align*}
Applying the Jesen inequality, we have
\begin{align*}
\frac{1}{p}f(1+\beta)+\frac{p-1}{p}f\Big(1-\frac{\beta}{p-1}\Big)\leq f(1),
\end{align*}
which yields that
\begin{align*}
(1-\alpha^{1+\beta})(1-\alpha^{1-\frac{\beta}{p-1}})^{p-1}\leq(1-\alpha)^p.
\end{align*}
From this and \eqref{represention of weight constant}, we deduce that
\begin{align*}
&\Big[\Big(\frac{1}{\nu_\lambda(I)}\int_I w_\epsilon(x)dx\Big)\Big(\frac{1}{\nu_\lambda(I)}\int_Ix^{(2\lambda+1)p'}w_\epsilon(x)^{1-p'}dx\Big)^{p-1}
\Big]^{\frac{1}{p-1}}\\
&\quad\leq\Big(\frac{1}{1+\beta}\Big)^{\frac{1}{p-1}}\left(\frac{1}{1-\frac{\beta}{p-1}}\right)=
\left(\frac{2\lambda+2}{(2\lambda+2)p-\epsilon}\right)^{\frac{1}{p-1}}\frac{(2\lambda+2)(p-1)}{\epsilon}.
\end{align*}
Therefore,
\begin{align*}
[w_\epsilon]_{\widetilde{A}_{p,\lambda}}^{\frac{1}{p-1}}\leq
\left(\frac{2\lambda+2}{(2\lambda+2)p-\epsilon}\right)^{\frac{1}{p-1}}\frac{(2\lambda+2)(p-1)}{\epsilon}.
\end{align*}
This, together with \eqref{lower bound of weight constant}, gives \eqref{weight constant of test weight}.

In the end, we prove the exponent $1/(p-1)$ is sharp by contradiction. Assume that the exponent of $[w]_{\widetilde{A}_{p,\lambda}}$ in Theorem \ref{mixed quantitative weighted strong estimates} is $\delta$ with $\delta<1/(p-1)$. Then by \eqref{lower bound of M},
\begin{align*}
c[w_\epsilon]_{\widetilde{A}_{p,\lambda}}^\delta\geq\|\mathcal{M}_\lambda\|_{L^p(\mathbb{R}_+,w_\epsilon dx)\rightarrow L^p(\mathbb{R}_+,w_\epsilon dx)}\geq\frac{(2\lambda+2)(p-1)}{\epsilon},
\end{align*}
which, together with \eqref{weight constant of test weight}, gives
\begin{align*}
\Big(\frac{(2\lambda+2)(p-1)}{\epsilon}\Big)^{1-(p-1)\delta}\leq c\Big(\frac{2\lambda+2}{(2\lambda+2)p-\epsilon}\Big)^{\delta}.
\end{align*}
Note that $1-(p-1)\delta>0$, the above inequality cannot hold when $\epsilon\rightarrow0$. This finish the proof.
\end{proof}

\section{Rubio de Francia extrapolation theorem in terms of a new class of weights}
In this section, we prove extrapolation theorems related to $\widetilde{A}_{p,\lambda}$. Our ideas come from \cite{CMP11,DrGrPePe,D,Ru1,Ru2}. To prove these extrapolation theorems, we first need to prove the following two lemmas.
\begin{lemma}\label{Rubio de Francia operator1}
Let $1<p<\infty$, $v_\lambda(x)=x^{2\lambda+1}$ with $\lambda>-1/2,\neq0$. Let $w\in\widetilde{A}_{p,\lambda}$ and $0\leq f\in L^p(\mathbb{R}_+,wdx)$. Define
\begin{equation*}
Rf(x):=v_\lambda(x)\sum_{k=0}^{\infty}\frac{\mathcal{M}_\lambda^kf(x)}
{(2\|\mathcal{M}_\lambda\|_{\mathcal{B}(L^p(\mathbb{R}_+,wdx))})^k},~x\in\mathbb{R}_+,
\end{equation*}
where $\mathcal{M}_\lambda^k:=\mathcal{M}_\lambda\circ \mathcal{M}_\lambda\circ\cdots\circ\mathcal{M}_\lambda$ is $k$ iterations of $\mathcal{M}_\lambda$ and $\mathcal{M}_\lambda^0$ is the identity operator. Then $f\leq Rf\cdot v_\lambda^{-1}$, $\|Rf\cdot v_\lambda^{-1}\|_{L^p(\mathbb{R}_+,wdx)}\leq2\|f\|_{{L^p(\mathbb{R}_+,wdx)}}$, and $Rf\in \widetilde{A}_{1,\lambda}$ with $[Rf]_{\widetilde{A}_{1,\lambda}}\leq2\|\mathcal{M}_\lambda\|_{\mathcal{B}(L^p(\mathbb{R}_+,wdx))}$.
\end{lemma}
\begin{proof}
Since $\mathcal{M}_\lambda^0f=f$, it is easy to see that $f\leq v_\lambda^{-1}Rf$. From Theorem \ref{mixed quantitative weighted strong estimates}, one has
\begin{align*}
\|v_\lambda^{-1}Rf\|_{L^p(\mathbb{R}_+,wdx)}&\leq\sum_{k=0}^\infty
\frac{\|\mathcal{M}_\lambda^kf\|_{L^p(\mathbb{R}_+,wdx)}}
{(2\|\mathcal{M}_\lambda\|_{\mathcal{B}(L^p(\mathbb{R}_+,wdx))})^k}\\
&\leq\sum_{k=0}^\infty 2^{-k}\|f\|_{L^p(\mathbb{R}_+,wdx)}=2\|f\|_{L^p(\mathbb{R}_+,wdx)}.
\end{align*}
For the last one, one can check that
\begin{align*}
\mathcal{M}_\lambda(v_\lambda^{-1}Rf)&\leq\sum_{k=0}^\infty
\frac{2\|\mathcal{M}_\lambda\|_{\mathcal{B}(L^p(\mathbb{R}_+,wdx))}\mathcal{M}_\lambda^{k+1}f}
{(2\|\mathcal{M}_\lambda\|_{\mathcal{B}(L^p(\mathbb{R}_+,wdx))})^{k+1}}\\
&=2\|\mathcal{M}_\lambda\|_{\mathcal{B}(L^p(\mathbb{R}_+,wdx))}\sum_{k=1}^\infty\frac{\mathcal{M}_\lambda^kf}
{(2\|\mathcal{M}_\lambda\|_{\mathcal{B}(L^p(\mathbb{R}_+,wdx))})^k}\\
&\leq2\|\mathcal{M}_\lambda\|_{\mathcal{B}(L^p(\mathbb{R}_+,wdx))}v_\lambda^{-1}Rf.
\end{align*}
This proves $[Rf]_{\widetilde{A}_{1,\lambda}}\leq2\|\mathcal{M}_\lambda\|_{\mathcal{B}(L^p(\mathbb{R}_+,wdx))}$.
\end{proof}

\begin{lemma}\label{Rubio de Francia operator2}
Let $1<p<\infty$, $v_\lambda(x)=x^{2\lambda+1}$ with $\lambda>-1/2,\neq0$. Let $w\in\widetilde{A}_{p,\lambda}$ and $0\leq f\in L^{p'}(\mathbb{R}_+,wdx)$. Define
\begin{equation*}
R'f(x):=v_\lambda(x)\sum_{k=0}^{\infty}\frac{\mathcal{M}_\lambda'^{k}f(x)}
{(2\|\mathcal{M}_\lambda'\|_{\mathcal{B}(L^{p'}(\mathbb{R}_+,wdx))})^k},~x\in\mathbb{R}_+,
\end{equation*}
where
\begin{equation*}
\mathcal{M}_\lambda'f(x):=\frac{\mathcal{M}_\lambda(fwv_\lambda^{-1})(x)v_\lambda(x)}{w(x)},
\end{equation*}
where $\mathcal{M}_\lambda'^k:=\mathcal{M}_\lambda'\circ \mathcal{M}_\lambda'\circ\cdots\circ\mathcal{M}_\lambda'$ is $k$ iterations of $\mathcal{M}_\lambda'$ and $\mathcal{M}_\lambda'^0$ is the identity operator.. Then $f\leq R'f\cdot v_\lambda^{-1}$, $\|R'f\cdot v_\lambda^{-1}\|_{L^{p'}(\mathbb{R}_+,wdx)}\leq2\|f\|_{{L^{p'}(\mathbb{R}_+,wdx)}}$, and $wv_\lambda^{-1}R'f\in \widetilde{A}_{1,\lambda}$ with $[wv_\lambda^{-1}R'f]_{\widetilde{A}_{1,\lambda}}\leq
2\|\mathcal{M}_\lambda'\|_{\mathcal{B}(L^{p'}(\mathbb{R}_+,wdx))}$.
\end{lemma}
\begin{proof}
The first result follows by the same reason as in Lemma \ref{Rubio de Francia operator1}. Using $w\in\widetilde{A}_{p,\lambda}\Leftrightarrow v_{\lambda}^{p'}w^{1-p'}\in\widetilde{A}_{p',\lambda}$ and the boundedness of $\mathcal{M}_\lambda$ on $L^{p'}(\mathbb{R}_+,v_{\lambda}^{p'}w^{1-p'}dx)$, we obtain
\begin{align}\label{variant HL boundedness}
\|\mathcal{M}_{\lambda}'f\|_{L^{p'}(\mathbb{R}_+,wdx)}&=
\|\mathcal{M}_\lambda(fwv_{\lambda}^{-1})\|_{L^{p'}(\mathbb{R}_+,v_{\lambda}^{p'}w^{1-p'}dx)}\\
&\leq\|\mathcal{M}_\lambda\|_{\mathcal{B}(L^{p'}(\mathbb{R}_+,v_{\lambda}^{p'}w^{1-p'}dx))}
\|f\|_{L^{p'}(\mathbb{R}_+,wdx)}\nonumber.
\end{align}
Thus, the rest of the proof of the second result is similar to that of Lemma \ref{Rubio de Francia operator1}, we omit the details.

To obtain the last result. Similar to the procedure of Lemma \ref{Rubio de Francia operator1}, we can get
\begin{align*}
\mathcal{M}_{\lambda'}(R'f\cdot v_\lambda^{-1})\leq2\|\mathcal{M}_{\lambda'}\|_{\mathcal{B}(L^{p'}(\mathbb{R}_+,wdx))}
R'f\cdot v_\lambda^{-1}.
\end{align*}
Multiply both sides by $wv_\lambda^{-1}$ and observe that
\begin{align*}
wv_\lambda^{-1}\mathcal{M}_{\lambda'}(R'f\cdot v_\lambda^{-1})=
\mathcal{M}_\lambda(wv_\lambda^{-2}R'f),
\end{align*}
we obtain
\begin{align*}
\mathcal{M}_\lambda(wv_\lambda^{-2}R'f)\leq
2\|\mathcal{M}_{\lambda'}\|_{\mathcal{B}(L^{p'}(\mathbb{R}_+,wdx))}wv_\lambda^{-2}R'f.
\end{align*}
This completes the proof.
\end{proof}

\begin{proof}[Proof of Theorem \ref{Rubio de Francia extrapolation theorem for new weights}]
We prove it by considering two cases: $1\leq p_0<p<\infty$ and $1<p<p_0$.

\textbf{Case 1. $1\leq p_0<p<\infty$}:\quad Assume $p_0=1$ first. By duality, Lemma \ref{Rubio de Francia operator2} and \eqref{variant HL boundedness}, there is a function $0\leq h\in L^{p'}(\mathbb{R}_+,wdx)$ with $\|h\|_{L^{p'}(\mathbb{R}_+,wdx)}=1$ such that
\begin{align*}
\|f\|_{L^p(\mathbb{R}_+,wdx)}&=\int_{\mathbb{R}_+}f(x)h(x)w(x)dx\\
&\leq\int_{\mathbb{R}_+}f(x)R'h(x)v_\lambda(x)^{-1}w(x)dx\\
&\leq N_1\Big(2\|\mathcal{M}_\lambda\|_{\mathcal{B}(L^{p'}(\mathbb{R}_+,v_\lambda^{p'}w^{1-p'}dx))}\Big)
\int_{\mathbb{R}_+}
g(x)R'h(x)v_\lambda(x)^{-1}w(x)dx\\
&\leq N_1\Big(2\|\mathcal{M}_\lambda\|_{\mathcal{B}(L^{p'}(\mathbb{R}_+,v_\lambda^{p'}w^{1-p'}dx))}\Big)
\Big(\int_{\mathbb{R}_+}g(x)^pw(x)dx\Big)^{1/p}\\
&\qquad\times\Big(\int_{\mathbb{R}_+}(R'h(x)v_\lambda(x)^{-1})^{p'}w(x)dx\Big)^{1/p'}\\
&\leq2N_1\Big(2\|\mathcal{M}_\lambda\|_{\mathcal{B}(L^{p'}(\mathbb{R}_+,v_\lambda^{p'}w^{1-p'}dx))}\Big)
\|g\|_{L^p(\mathbb{R}_+,wdx)}.
\end{align*}

Next, we deal with the case $1<p_0<p<\infty$. Again by duality and Lemma \ref{Rubio de Francia operator2},
\begin{align}\label{case 1 origin}
\|f\|_{L^p(\mathbb{R}_+,wdx)}&=\int_{\mathbb{R}_+}f(x)h(x)w(x)dx\\
&=\int_{\mathbb{R}_+}f(x)h(x)^{\frac{p-p_0}{p_0(p-1)}}h(x)^{\frac{(p_0-1)p}{p_0(p-1)}}w(x)dx\nonumber\\
&\leq\int_{\mathbb{R}_+}f(x)(R'h(x)v_{\lambda}(x)^{-1})^{\frac{p-p_0}{p_0(p-1)}}
h(x)^{\frac{(p_0-1)p}{p_0(p-1)}}w(x)dx.\nonumber
\end{align}

Denote $W=(v_\lambda^{-1}R'h)^{\frac{p-p_0}{p-1}}w$, where $w\in \widetilde{A}_{p,\lambda}$. In the following, we show that $W\in\widetilde{A}_{p_0,\lambda}$. Fix an interval $I\subset\mathbb{R}_+$. By Lemma \ref{Rubio de Francia operator2},
\begin{align*}
[v_\lambda^{-1}wR'h]_{\widetilde{A}_{1,\lambda}}\leq2\|\mathcal{M}_\lambda'\|_{\mathcal{B}(L^{p'}(\mathbb{R}_+,wdx))},
\end{align*}
which yields that for a.e. $x\in I$,
\begin{align*}
\frac{1}{\nu_\lambda(I)}\int_IR'h(x)w(x)v_\lambda(x)^{-1}dx\leq2
\|\mathcal{M}_\lambda'\|_{\mathcal{B}(L^{p'}(\mathbb{R}_+,wdx))}
R'h(x)w(x)v_\lambda(x)^{-2}.
\end{align*}
From this, we have
\begin{align}\label{weight pro 1}
&\Big(\frac{1}{\nu_\lambda(I)}\int_Iv_{\lambda}(x)^{p_0'}W(x)^{1-p_0'}dx\Big)^{p_0-1}\\
&\quad=\Big(\frac{1}{\nu_\lambda(I)}\int_Iv_{\lambda}(x)^{p_0'}(R'h(x)v_{\lambda}(x)^{-1})
^{-\frac{p-p_0}{(p_0-1)(p-1)}}w(x)^{1-p_0'}dx\Big)^{p_0-1}\nonumber\\
&\quad=\Big(\frac{1}{\nu_\lambda(I)}\int_Iv_{\lambda}(x)^{p_0'}(R'h(x)w(x)v_{\lambda}(x)^{-2})
^{-\frac{p-p_0}{(p_0-1)(p-1)}}\nonumber\\
&\qquad\quad\times(v_\lambda(x)w(x)^{-1})^{-\frac{p-p_0}{(p_0-1)(p-1)}}w(x)^{1-p_0'}dx
\Big)^{p_0-1}\nonumber\\
&\quad\leq\Big(2\|\mathcal{M}_\lambda'\|_{\mathcal{B}(L^{p'}(\mathbb{R}_+,wdx))}\Big)^{\frac{p-p_0}{p-1}}
\Big(\frac{1}{\nu_\lambda(I)}\int_IR'h(x)w(x)v_\lambda(x)^{-1}dx\Big)^{-\frac{p-p_0}{p-1}}\nonumber\\
&\quad\qquad\times\Big(\frac{1}{\nu_\lambda(I)}\int_I
v_{\lambda}(x)^{p_0'}(v_\lambda(x)w(x)^{-1})^{-\frac{p-p_0}{(p_0-1)(p-1)}}w(x)^{1-p_0'}dx\Big)^{p_0-1}
\nonumber\\
&\quad=\Big(2\|\mathcal{M}_\lambda'\|_{B(L^{p'}(\mathbb{R}_+,wdx))}\Big)^{\frac{p-p_0}{p-1}}
\Big(\frac{1}{\nu_\lambda(I)}\int_IR'h(x)w(x)v_\lambda(x)^{-1}dx\Big)^{-\frac{p-p_0}{p-1}}\nonumber\\
&\quad\qquad\times\Big[\Big(\frac{1}{\nu_\lambda(I)}\int_I
v_\lambda(x)^{p'}w(x)^{1-p'}dx\Big)^{p-1}\Big]^{\frac{p_0-1}{p-1}}.\nonumber
\end{align}

On the other hand, applying H\"{o}lder's inequality,
\begin{align}\label{weight pro 2}
&\frac{1}{\nu_\lambda(I)}\int_IW(x)dx\\
&\quad=\frac{1}{\nu_\lambda(I)}\int_I(R'h(x)
v_\lambda(x)^{-1})^{\frac{p-p_0}{p-1}}w(x)dx\nonumber\\
&\quad\leq\Big(\frac{1}{\nu_\lambda(I)}\int_IR'h(x)v_\lambda(x)^{-1}w(x)dx\Big)
^{\frac{p-p_0}{p-1}}\Big(\frac{1}{\nu_\lambda(I)}\int_Iw(x)dx\Big)^{\frac{p_0-1}{p-1}}\nonumber.
\end{align}
Therefore, combined with \eqref{weight pro 1} and \eqref{weight pro 2}, we deduce that
\begin{align*}
&\Big(\frac{1}{\nu_\lambda(I)}\int_IW(x)dx\Big)
\Big(\frac{1}{\nu_\lambda(I)}\int_Iv_{\lambda}(x)^{p_0'}W(x)^{1-p_0'}dx\Big)^{p_0-1}\\
&\quad\leq\Big(2\|\mathcal{M}_\lambda'\|_{\mathcal{B}(L^{p'}(\mathbb{R}_+,wdx))}\Big)^{\frac{p-p_0}{p-1}}
[w]_{\widetilde{A}_{p,\lambda}}^{\frac{p_0-1}{p-1}},
\end{align*}
which gives
\begin{align}\label{weight constant compare}
[W]_{\widetilde{A}_{p_0,\lambda}}\leq
\Big(2\|\mathcal{M}_\lambda'\|_{\mathcal{B}(L^{p'}(\mathbb{R}_+,wdx))}\Big)^{\frac{p-p_0}{p-1}}
[w]_{\widetilde{A}_{p,\lambda}}^{\frac{p_0-1}{p-1}}.
\end{align}

Now, we return to prove the result in case $1<p_0<p<\infty$. By Lemma \ref{Rubio de Francia operator2},
\begin{align*}
\|f\|_{L^{p_0}(\mathbb{R}_+,Wdx)}&=\Big(\int_{\mathbb{R}_+}f(x)^{p_0}(R'h(x)v_{\lambda}(x)^{-1})
^{\frac{p-p_0}{p-1}}w(x)dx\Big)^{1/p_0}\\
&\leq\Big[\Big(\int_{\mathbb{R}_+}f(x)^{p}w(x)dx\Big)^{\frac{p_0}{p}}
\Big(\int_{\mathbb{R}_+}
(R'h(x)v_\lambda(x)^{-1})^{p'}w(x)dx\Big)^{\frac{p-p_0}{(p-1)p'}}\Big]^{1/p_0}\nonumber\\
&\leq2^{\frac{p-p_0}{(p-1)p_0}}\Big[\Big(\int_{\mathbb{R}_+}f(x)^{p}w(x)dx\Big)^{\frac{p_0}{p}}
\Big(\int_{\mathbb{R}_+}
h(x)^{p'}w(x)dx\Big)^{\frac{p-p_0}{(p-1)p'}}\Big]^{1/p_0}\nonumber\\
&=2^{\frac{p-p_0}{(p-1)p_0}}\|f\|_{L^p(\mathbb{R}_+,wdx)}<\infty.
\end{align*}
Thus, by Holder's inequality, \eqref{variant HL boundedness}, \eqref{case 1 origin}, \eqref{weight constant compare}, the assumption and Lemma \ref{Rubio de Francia operator2}, we have
\begin{align*}
&\|f\|_{L^p(\mathbb{R}_+,wdx)}\\
&\quad\leq\int_{\mathbb{R}_+}f(x)W(x)^{\frac{1}{p_0}}
h(x)^{\frac{p}{p_0'(p-1)}}w(x)^{\frac{1}{p_0'}}dx\\
&\quad\leq\Big(\int_{\mathbb{R}_+}f(x)^{p_0}W(x)dx\Big)^{\frac{1}{p_0}}
\Big(\int_{\mathbb{R}_+}h(x)^{p'}w(x)dx\Big)^{\frac{1}{p_0'}}\\
&\quad\leq N_{p_0}([W]_{\widetilde{A}_{p_0,\lambda}})\Big(\int_{\mathbb{R}_+}g(x)^{p_0}
(R'h(x)v_\lambda(x)^{-1})^{\frac{p-p_0}{p-1}}
w(x)dx\Big)^{\frac{1}{p_0}}\\
&\quad\leq N_{p_0}\Big(2^{\frac{p-p_0}{p-1}}\|\mathcal{M}_\lambda\|_{\mathcal{B}(L^{p'}
(\mathbb{R}_+,v_\lambda^{p'}w^{1-p'}dx))}^{\frac{p-p_0}{p-1}}
[w]_{\widetilde{A}_{p,\lambda}}^{\frac{p_0-1}{p-1}}\Big)
\Big(\int_{\mathbb{R}_+}g(x)^{p_0\cdot\frac{p}{p_0}}w(x)dx\Big)^{\frac{p_0}{p}\cdot\frac{1}{p_0}}\\
&\quad\quad\times\Big(\int_{\mathbb{R}_+}(R'h(x)v_\lambda(x)^{-1})^{\frac{p-p_0}{p-1}\cdot\frac{p}{p-p_0}}
w(x)dx\Big)^{\frac{p-p_0}{pp_0}}\\
&\quad=N_{p_0}\Big(2^{\frac{p-p_0}{p-1}}
\|\mathcal{M}_\lambda\|_{\mathcal{B}(L^{p'}(\mathbb{R}_+,v_\lambda^{p'}w^{1-p'}dx))}^{\frac{p-p_0}{p-1}}
[w]_{\widetilde{A}_{p,\lambda}}^{\frac{p_0-1}{p-1}}\Big)
\Big(\int_{\mathbb{R}_+}g(x)^pw(x)dx\Big)^{\frac{1}{p}}\\
&\quad\quad\times\Big(\int_{\mathbb{R}_+}(R'h(x)v_\lambda(x)^{-1})^{p'}w(x)dx\Big)
^{\frac{1}{p'}\frac{p'(p-p_0)}{pp_0}}\\
&\quad\leq2^{\frac{p-p_0}{p_0(p-1)}}N_{p_0}\Big(2^{\frac{p-p_0}{p-1}}
\|\mathcal{M}_\lambda\|_{\mathcal{B}(L^{p'}
(\mathbb{R}_+,v_\lambda^{p'}w^{1-p'}dx))}^{\frac{p-p_0}{p-1}}
[w]_{\widetilde{A}_{p,\lambda}}^{\frac{p_0-1}{p-1}}\Big)\|g\|_{L^p(\mathbb{R}_+,wdx)}.
\end{align*}

\textbf{Case 2. $1<p<p_0$}:\quad Since $(f,g)\in\mathcal{F}$, we can define
\begin{equation*}
h(x):=\epsilon\frac{f(x)}{\|f\|_{L^p(\mathbb{R}_+,wdx)}}+\frac{g(x)}{\|g\|_{L^p(\mathbb{R}_+,wdx)}},~
\epsilon>0.
\end{equation*}
Then $h\in L^p(\mathbb{R}_+,wdx)$ with $\|h\|_{L^p(\mathbb{R}_+,wdx)}\leq1+\epsilon$. Let $R$ be defined in Lemma \ref{Rubio de Francia operator1}. By H\"{o}lder's inequality and Lemma \ref{Rubio de Francia operator1}, we have
\begin{align}\label{case 2 pro}
&\|f\|_{L^p(\mathbb{R}_+,wdx)}\\
&\quad=\Big(\int_{\mathbb{R}_+}f(x)^p(Rh(x)v_\lambda(x)^{-1})
^{\frac{(p-p_0)p}{p_0}}(Rh(x)v_\lambda(x)^{-1})
^{\frac{(p_0-p)p}{p_0}}w(x)dx\Big)^{\frac{1}{p}}\nonumber\\
&\quad\leq\Big(\int_{\mathbb{R}_+}f(x)^{p\cdot\frac{p_0}{p}}
(Rh(x)v_\lambda(x)^{-1})^{\frac{(p-p_0)p}{p_0}\cdot\frac{p_0}{p}}
w(x)dx\Big)^{\frac{1}{p}\cdot\frac{p}{p_0}}\nonumber\\
&\qquad\quad\times\Big(\int_{\mathbb{R}_+}
(Rh(x)v_\lambda(x)^{-1})^{\frac{(p_0-p)p}{p_0}\cdot\frac{p_0}{p_0-p}}
w(x)dx\Big)^{\frac{1}{p}\cdot\frac{p_0-p}{p_0}}\nonumber\\
&\quad\leq2^{\frac{p_0-p}{p_0}}\|h\|_{L^p(\mathbb{R}_+,wdx)}^{\frac{p_0-p}{p_0}}
\Big(\int_{\mathbb{R}_+}f(x)^{p_0}(Rh(x)v_\lambda(x)^{-1})^{p-p_0}w(x)dx\Big)^{\frac{1}{p_0}}\nonumber\\
&\quad\leq[2(1+\epsilon)]^{\frac{p_0-p}{p_0}}
\Big(\int_{\mathbb{R}_+}f(x)^{p_0}(Rh(x)v_\lambda(x)^{-1})^{p-p_0}w(x)dx\Big)^{\frac{1}{p_0}}.\nonumber
\end{align}

Again by Lemma \ref{Rubio de Francia operator1} and $p<p_0$,
\begin{align*}
&\int_{\mathbb{R}_+}f(x)^{p_0}(Rh(x)v_\lambda(x)^{-1})^{p-p_0}w(x)dx\\
&\quad\leq\int_{\mathbb{R}_+}f(x)^{p_0}h(x)^{p-p_0}w(x)dx\\
&\quad\leq\int_{\mathbb{R}_+}f(x)^{p_0}\Big(\epsilon f(x)/\|f\|_{L^p(\mathbb{R}_+,wdx)}\Big)^{p-p_0}w(x)dx\\
&\quad=\epsilon^{p-p_0}\|f\|_{L^p(\mathbb{R}_+,wdx)}^{p_0}<\infty.
\end{align*}
If we can prove $(Rh\cdot v_\lambda^{-1})^{p-p_0}w\in \widetilde{A}_{p_0,\lambda}$ with
\begin{align}\label{weight constant compare1}
[(Rh\cdot v_\lambda^{-1})^{p-p_0}w]_{\widetilde{A}_{p_0,\lambda}}\leq2^{p_0-p}
\|\mathcal{M}_\lambda\|_{\mathcal{B}(L^p(\mathbb{R}_+,wdx))}^{p_0-p}
[w]_{\widetilde{A}_{p,\lambda}}.
\end{align}
Then by \eqref{case 2 pro}, \eqref{weight constant compare1} and the assumption, we deduce that
\begin{align*}
\|f\|_{L^p(\mathbb{R}_+,wdx)}&\leq[2(1+\epsilon)]^{\frac{p_0-p}{p_0}}
N_{p_0}([(Rh\cdot v_\lambda^{-1})^{p-p_0}w]_{\widetilde{A}_{p_0,\lambda}})\\
&\quad\times\Big(\int_{\mathbb{R}_+}g(x)^{p_0}(Rh\cdot v_\lambda^{-1})^{p-p_0}w(x)dx\Big)^{\frac{1}{p_0}}
\\
&\leq[2(1+\epsilon)]^{\frac{p_0-p}{p_0}}
N_{p_0}(2^{p_0-p}\|\mathcal{M}_\lambda\|_{\mathcal{B}(L^p(\mathbb{R}_+,wdx))}^{p_0-p}[w]_{\widetilde{A}_{p,\lambda}})\\
&\quad\times\Big(\int_{\mathbb{R}_+}g(x)^{p_0}(Rh\cdot v_\lambda^{-1})^{p-p_0}w(x)dx\Big)^{\frac{1}{p_0}}.
\end{align*}
According to the definition of $h$ and Lemma \ref{Rubio de Francia operator1},
\begin{align*}
&\Big(\int_{\mathbb{R}_+}g(x)^{p_0}(Rh(x)v_\lambda(x)^{-1})^{p-p_0}w(x)dx\Big)^{\frac{1}{p_0}}\\
&\quad\leq\Big(\int_{\mathbb{R}_+}g(x)^{p_0}h(x)^{p-p_0}w(x)dx\Big)^{\frac{1}{p_0}}\\
&\quad\leq\Big(\int_{\mathbb{R}_+}g(x)^{p_0}(g(x)/\|g\|_{L^p(\mathbb{R}_+,wdx)})^{p-p_0}w(x)dx\Big)
^{\frac{1}{p_0}}\\
&\quad=\|g\|_{L^p(\mathbb{R}_+,wdx)}.
\end{align*}
From this, we get
\begin{align*}
\|f\|_{L^p(\mathbb{R}_+,wdx)}\leq[2(1+\epsilon)]^{\frac{p_0-p}{p_0}}
N_{p_0}(2^{p_0-p}\|\mathcal{M}_\lambda\|_{\mathcal{B}(L^p(\mathbb{R}_+,wdx))}^{p_0-p}[w]_{\widetilde{A}_{p,\lambda}})
\|g\|_{L^p(\mathbb{R}_+,wdx)}.
\end{align*}
By letting $\epsilon\rightarrow0$, we get the conclusion.

As we stated before, the conclusion holds provided that \eqref{weight constant compare1} is true. Thus, we need to prove \eqref{weight constant compare1}. Fix an interval $I\subset\mathbb{R}_+$, Lemma \ref{Rubio de Francia operator1} shows that for a.e. $x\in I$
\begin{align*}
\frac{1}{\nu_\lambda(I)}\int_IRh(x)v_\lambda(x)^{-1}d\nu_\lambda(x)\leq
2\|\mathcal{M}_\lambda\|_{\mathcal{B}(L^p(\mathbb{R}_+,wdx))}
Rh(x)v_\lambda(x)^{-1}.
\end{align*}
Hence, for a.e. $x\in I$,
\begin{align*}
(Rh(x)v_\lambda(x)^{-1})^{p-p_0}\leq\Big(2\|\mathcal{M}_\lambda\|_{\mathcal{B}(L^p(\mathbb{R}_+,wdx))}\Big)
^{p_0-p}\nu_\lambda(I)^{p_0-p}\Big(\int_IRh(x)dx\Big)^{p-p_0}.
\end{align*}
From this, we have
\begin{align}\label{weight pro3}
&\frac{1}{\nu_\lambda(I)}\int_I(Rh(x)v_\lambda(x)^{-1})^{p-p_0}w(x)dx\\
&\quad\leq\Big(2\|\mathcal{M}_\lambda\|_{\mathcal{B}(L^p(\mathbb{R}_+,wdx))}\Big)
^{p_0-p}
\Big(\frac{1}{\nu_\lambda(I)}\int_Iw(x)dx\Big)
\Big(\frac{1}{\nu_\lambda(I)}\int_IRh(x)dx\Big)^{p-p_0}\nonumber.
\end{align}
On the other hand, by H\"{o}lder's inequality, we have
\begin{align}\label{weight pro4}
&\frac{1}{\nu_\lambda(I)}\int_Iv_\lambda(x)^{p_0'}(Rh(x)v_\lambda(x)^{-1})^{-\frac{p-p_0}{p_0-1}}
w(x)^{-\frac{1}{p_0-1}}dx\\
&\quad=\frac{1}{\nu_\lambda(I)}\int_Iv_\lambda(x)^{\frac{p'(p-1)}{p_0-1}}
v_\lambda(x)^{\frac{p_0-p}{p_0-1}}(Rh(x)v_\lambda(x)^{-1})^{-\frac{p-p_0}{p_0-1}}
w(x)^{-\frac{1}{p_0-1}}dx\nonumber\\
&\quad=\frac{1}{\nu_\lambda(I)}\int_Iv_\lambda(x)^{\frac{p'(p-1)}{p_0-1}}
Rh(x)^{\frac{p_0-p}{p_0-1}}w(x)^{-\frac{1}{p_0-1}}dx\nonumber\\
&\quad\leq\Big(\frac{1}{\nu_\lambda(I)}\int_Iv_\lambda(x)^{p'}w(x)^{-\frac{1}{p-1}}dx\Big)
^{\frac{p-1}{p_0-1}}
\Big(\frac{1}{\nu_\lambda(I)}\int_IRh(x)dx\Big)^{\frac{p_0-p}{p_0-1}}.\nonumber
\end{align}
Therefore, by \eqref{weight pro3} and \eqref{weight pro4}, we obtain
\begin{align*}
&\Big(\frac{1}{\nu_\lambda(I)}\int_I\Big(\frac{Rh(x)}{v_\lambda(x)}\Big)^{p-p_0}w(x)dx\Big)
\Big(\frac{1}{\nu_\lambda(I)}\int_Iv_\lambda(x)^{p_0'}
\Big(\frac{Rh(x)}{v_\lambda(x)}\Big)^{-\frac{p-p_0}{p_0-1}}
w(x)^{-\frac{1}{p_0-1}}dx\Big)^{p_0-1}\\
&\quad\leq2^{p_0-p}
\|\mathcal{M}_\lambda\|_{\mathcal{B}(L^p(\mathbb{R}_+,wdx))}^{p_0-p}[w]_{\widetilde{A}_{p,\lambda}}.
\end{align*}
This gives \eqref{weight constant compare1}. The proof of Theorem \ref{Rubio de Francia extrapolation theorem for new weights} is completed.
\end{proof}

As corollaries of Theorem \ref{Rubio de Francia extrapolation theorem for new weights}, we present the proofs of Corollaries \ref{vector-value extrapolation}-\ref{weak type extrapolation}.
\begin{proof}[Proof of Corollary \ref{vector-value extrapolation}]
Denote
\begin{align*}
F(x):=\Big(\sum_if_i(x)^q\Big)^{1/q},\qquad G(x):=\Big(\sum_ig_i(x)^q\Big)^{1/q}.
\end{align*}
By Theorem \ref{Rubio de Francia extrapolation theorem for new weights} with $p=q$, we obtain
\begin{align*}
\|F\|_{L^q(\mathbb{R}_+,wdx)}^q&=\sum_i\int_{\mathbb{R}_+}f_i(x)^qw(x)dx\\
&\lesssim\sum_i\int_{\mathbb{R}_+}g_i(x)^qw(x)dx=\|G\|_{L^q(\mathbb{R}_+,wdx)}^q.
\end{align*}
This shows that $(F,G)\in\mathcal{F}$ satisfies the assumption of Theorem \ref{Rubio de Francia extrapolation theorem for new weights} with $p_0=q$. Hence, applying Theorem \ref{Rubio de Francia extrapolation theorem for new weights} with $(F,G)\in\mathcal{F}$, we have that for all $w\in \widetilde{A}_{p,\lambda}$, $\|F\|_{L^p(\mathbb{R}_+,wdx)}\lesssim\|G\|_{L^p(\mathbb{R}_+,wdx)}$, this is the desired result.
\end{proof}

\begin{proof}[Proof of Corollary \ref{scale}]
Since
\begin{align*}
\int_{\mathbb{R}_+}(f(x)^r)^{p_0/r}w(x)dx&=\int_{\mathbb{R}_+}f(x)^{p_0}w(x)dx\\
&\lesssim\int_{\mathbb{R}_+}g(x)^{p_0}w(x)dx=\int_{\mathbb{R}_+}(g(x)^r)^{p_0/r}w(x)dx.
\end{align*}
Thus, applying Theorem \ref{Rubio de Francia extrapolation theorem for new weights} with $(f^r,g^r)\in\mathcal{F}$, we have that for all $w\in\widetilde{A}_{p/r,\lambda}$,
\begin{align*}
\|f\|_{L^{p}(\mathbb{R}_+,wdx)}^p=\int_{\mathbb{R}_+}(f(x)^r)^{p/r}w(x)dx\lesssim
\int_{\mathbb{R}_+}(g(x)^r)^{p/r}w(x)dx=\|g\|_{L^{p}(\mathbb{R}_+,wdx)}^p.
\end{align*}
\end{proof}

\begin{proof}[Proof of Corollary \ref{A extraapolaation}]
Fix $1<r<\infty$. Using the assumption, for each $w\in\widetilde{A}_{r,\lambda}\subset\widetilde{A}_{\infty,\lambda}$, $(f^{p_0/r},g^{p_0/r})\in\mathcal{F}$, we have
\begin{align*}
\int_{\mathbb{R}_+}(f(x)^{p_0/r})^rw(x)dx&=\int_{\mathbb{R}_+}f(x)^{p_0}w(x)dx\\
&\lesssim\int_{\mathbb{R}_+}g(x)^{p_0}w(x)dx=\int_{\mathbb{R}_+}(g(x)^{p_0/r})^rw(x)dx.
\end{align*}
The above inequality is just the hypothesis of Theorem \ref{Rubio de Francia extrapolation theorem for new weights} with $r$ in place of $p_0$, then for all $1<s<\infty$ and $w\in\widetilde{A}_{s,\lambda}$,
\begin{align*}
\int_{\mathbb{R}_+}(f(x)^{p_0/r})^sw(x)dx\lesssim\int_{\mathbb{R}_+}(g(x)^{p_0/r})^sw(x)dx.
\end{align*}

Fix $p>0$ and $w\in\widetilde{A}_{\infty,\lambda}$. Since the $\widetilde{A}_{p,\lambda}$ classes are nested (see \cite{LiLiangLinShen}), without loss of generality, we may assume that $w\in\widetilde{A}_{s,\lambda}$ with $s>p/p_0$. Thus, we can choose $r>1$ such that $p_0s/r=p$, and the above inequality gives the desired result.
\end{proof}

\begin{proof}[Proof of Corollary \ref{weak type extrapolation}]
Denote $f_t:=t\chi_{\{x\in\mathbb{R}_+:f(x)>t\}},t>0$. Then for each $w\in\widetilde{A}_{p_0,\lambda}$ and pair $(f_t,g)$, the assumption yields that
\begin{align*}
\|f_t\|_{L^{p_0}(\mathbb{R}_+,wdx)}&=tw(\{x\in\mathbb{R}_+:f(x)>t\})^{1/p_0}\\
&\leq\|f\|_{L^{p_0}(\mathbb{R}_+,wdx)}\lesssim\|g\|_{L^{p_0}(\mathbb{R}_+,wdx)}.
\end{align*}
From this, for any $t>0$, we can use Theorem \ref{Rubio de Francia extrapolation theorem for new weights} with $(f_t,g)\in\mathcal{F}$ to get
\begin{align*}
tw(\{x\in\mathbb{R}_+:f(x)>t\})^{1/p}
=\|f_t\|_{L^{p}(\mathbb{R}_+,wdx)}\lesssim\|g\|_{L^{p}(\mathbb{R}_+,wdx)},
\end{align*}
which implies that
\begin{align*}
\|f\|_{L^{p,\infty}(\mathbb{R}_+,wdx)}=\sup_{t>0}tw(\{x\in\mathbb{R}_+:f(x)>t\})^{1/p}
\lesssim\|g\|_{L^{p}(\mathbb{R}_+,wdx)}.
\end{align*}
\end{proof}

\section{Restricted weighted weak type estimates for operators in the Bessel setting}
In this section, we give the proofs of Theorems \ref{characterization restricted weak type of HL}-\ref{quantitative two weight restricted weak type estimate for sparse operator of commutator 2}. Though our ideas origins from \cite{ChHuKu,KT,R}, we should improve the methods in \cite{ChHuKu,KT,R} and sometimes need more information about the weights due to the intrinsic structure of the Bessel operator.

We begin by proving some properties of the new class of weights $\widetilde{A}_{p,\lambda}^\mathcal{R}$, which are independent of interests. The following theorem illustrates the relationship between $\widetilde{A}_{p,\lambda}$ and $\widetilde{A}_{p,\lambda}^\mathcal{R}$.
\begin{theorem}\label{compare weight constant1}
Let $1\leq p<q<\infty$ and $\lambda>-1/2, \neq 0$. Then $\widetilde{A}_{1,\lambda}=\widetilde{A}_{1,\lambda}^\mathcal{R}$, $\widetilde{A}_{p,\lambda}\subset\widetilde{A}_{p,\lambda}^\mathcal{R}\subset\widetilde{A}_{q,\lambda}$ and $\widetilde{A}_{\infty,\lambda}=\bigcup_{p\geq1}\widetilde{A}_{p,\lambda}=
\bigcup_{p\geq1}\widetilde{A}_{p,\lambda}^\mathcal{R}$. Moreover,
\begin{align*}
[w]_{\widetilde{A}_{p,\lambda}^\mathcal{R}}\leq[w]_{\widetilde{A}_{p,\lambda}}^{1/p},\quad [w]_{\widetilde{A}_{q,\lambda}}\leq\Big(\frac{p'}{p'-q'}\Big)^{q-1}[w]_{\widetilde{A}_{p,\lambda}^\mathcal{R}}^q.
\end{align*}
\end{theorem}
\begin{proof}
Obviously, $\widetilde{A}_{1,\lambda}=\widetilde{A}_{1,\lambda}^\mathcal{R}$. For the remaining claims, it is enough to establish the inequalities
\begin{align*}
[w]_{\widetilde{A}_{p,\lambda}^\mathcal{R}}\leq[w]_{\widetilde{A}_{p,\lambda}}^{1/p},\quad [w]_{\widetilde{A}_{q,\lambda}}\leq\Big(\frac{p'}{p'-q'}\Big)^{q-1}[w]_{\widetilde{A}_{p,\lambda}^\mathcal{R}}^q.
\end{align*}
We rewrite the definition of $[w]_{\widetilde{A}_{p,\lambda}}^{1/p}$ by
\begin{equation*}
[w]_{\widetilde{A}_{p,\lambda}}^{1/p}=\sup_{I\subset\mathbb{R}_+}\Big(\frac{w(I)}{\nu_\lambda(I)}\Big)^{1/p}
\Big(\frac{\|\chi_I(\cdot)^{2\lambda+1}w^{-1}\|_{L^{p'}(\mathbb{R}_+,wdx)}}{\nu_\lambda(I)}\Big)^{1/p'}.
\end{equation*}
Then $[w]_{\widetilde{A}_{p,\lambda}^\mathcal{R}}\leq[w]_{\widetilde{A}_{p,\lambda}}^{1/p}$ follows by
\begin{align*}
\|\chi_I(\cdot)^{2\lambda+1}w^{-1}\|_{L^{p',\infty}(\mathbb{R}_+,wdx)}\leq
\|\chi_I(\cdot)^{2\lambda+1}w^{-1}\|_{L^{p'}(\mathbb{R}_+,wdx)}.
\end{align*}

For each interval $I\subset\mathbb{R}_+$, $f\in L^{p',\infty}(\mathbb{R}_+,wdx)$. Recall that
\begin{align*}
\int_I|f(x)|^{q'}w(x)dx\leq\frac{p'}{p'-q'}w(I)^{1-\frac{q'}{p'}}\|f\|_{L^{p',\infty}(\mathbb{R}_+,wdx)}^{q'},
\,0<q'<p',
\end{align*}
see \cite[Exercise 1.1.11]{Gra}. Applying the inequality above with $f=\chi_Iv_\lambda w^{-1}$, one has
\begin{align*}
&\Big(\frac{1}{\nu_\lambda(I)}\int_Iw(x)dx\Big)\Big(\frac{1}{\nu_\lambda(I)}\int_Ix^{(2\lambda+1)q'}
w(x)^{1-q'}dx\Big)^{q-1}\\
&\quad=\frac{w(I)}{\nu_\lambda(I)^q}\Big(\int_I(v_\lambda(x)w(x)^{-1})^{q'}w(x)dx\Big)^{q-1}\\
&\quad\leq\Big(\frac{p'}{p'-q'}\Big)^{q-1}\frac{w(I)w(I)^{\frac{q}{p}-1}
\|\chi_Iv_\lambda w^{-1}\|_{L^{p',\infty}(\mathbb{R}_+,wdx)}^q}{\nu_\lambda(I)^q}\\
&\quad\leq\Big(\frac{p'}{p'-q'}\Big)^{q-1}[w]_{\widetilde{A}_{p,\lambda}^\mathcal{R}}^q.
\end{align*}
Therefore,
\begin{align*}
[w]_{\widetilde{A}_{q,\lambda}}\leq\Big(\frac{p'}{p'-q'}\Big)^{q-1}[w]_{\widetilde{A}_{p,\lambda}^\mathcal{R}}^q
\end{align*}
follows by taking the supremum over all intervals $I\subset\mathbb{R}_+$.
\end{proof}

The following example, together with Theorem \ref{compare weight constant1}, says that $\widetilde{A}_{p,\lambda}$ is a proper subset of $\widetilde{A}_{p,\lambda}^\mathcal{R}$.
\begin{example}
Let $1<p<\infty$ and $\lambda>-1/2, \neq0$. $|x|^{p-1+(2\lambda+1)p}$ belongs to $\widetilde{A}_{p,\lambda}^\mathcal{R}$ but not to $\widetilde{A}_{p,\lambda}$.
\end{example}
\begin{proof}
It was proved in \cite{LiLiangLinShen} that $|x|^{p-1+(2\lambda+1)p}\not\in\widetilde{A}_{p,\lambda}$. We are devoted to proving $|x|^{p-1+(2\lambda+1)p}\in\widetilde{A}_{p,\lambda}^\mathcal{R}$. For any interval $I:=(a,b)\subset\mathbb{R}_+$, we claim that
\begin{align}\label{vital for example}
&\|\chi_{(a,b)}(\cdot)^{(2\lambda+2)(1-p)}\|_{L^{p',\infty}(\mathbb{R}_+,|x|^{p-1+(2\lambda+1)p}dx)}\\
&\quad\leq b^{(1-p)(2\lambda+2)}\Big(\int_a^b x^{p-1+(2\lambda+1)p}dx\Big)^{1/p'}\nonumber.
\end{align}
Indeed, one can see that
\begin{align}\label{local sup}
&\sup_{0<t<b^{(1-p)(2\lambda+2)}}t\Big(\int_{\{x\in(a,b)\subset\mathbb{R}_+:x^{(2\lambda+2)(1-p)}>t\}}
x^{p-1+(2\lambda+1)p}dx\Big)^{1/p'}\\
&\quad=\sup_{0<t<b^{(1-p)(2\lambda+2)}}t\Big(\int_{a}^{\min\big\{b,t^{-\frac{1}{(p-1)(2\lambda+2)}}\big\}}
x^{p-1+(2\lambda+1)p}dx\Big)^{1/p'}\nonumber\\
&\quad\leq b^{(1-p)(2\lambda+2)}\Big(\int_a^b x^{p-1+(2\lambda+1)p}dx\Big)^{1/p'}\nonumber.
\end{align}

On the other hand, note that for $s>a$,
\begin{equation*}
f(s)=s^{(1-p)(2\lambda+2)}\Big(\int_a^sx^{p-1+(2\lambda+1)p}dx\Big)^{1/p'}
\end{equation*}
is an increasing function. In fact,
\begin{align*}
f(s)=s^{(1-p)(2\lambda+2)}\Big(\frac{s^{(2\lambda+2)p}-a^{(2\lambda+2)p}}{(2\lambda+2)p}\Big)^{1/p'},
\end{align*}
then
\begin{align*}
\log f=-(p-1)(2\lambda+2)\log s+\frac{1}{p'}\log\Big(\frac{s^{(2\lambda+2)p}-a^{(2\lambda+2)p}}{(2\lambda+2)p}\Big).
\end{align*}
It follows that
\begin{align*}
\frac{f'(s)}{f(s)}&=-\frac{(p-1)(2\lambda+2)}{s}+\frac{[(2\lambda+2)p]^2s^{(2\lambda+2)p-1}}
{p'(s^{(2\lambda+2)p}-a^{(2\lambda+2)p})}\\
&=-\frac{(p-1)(2\lambda+2)}{s}+\frac{(p-1)(2\lambda+2)}{s}\frac{s^{(2\lambda+2)p}(2\lambda+2)p}
{s^{(2\lambda+2)p}-a^{(2\lambda+2)p}}\\
&=\frac{(p-1)(2\lambda+2)}{s}\frac{a^{(2\lambda+2)p}+s^{(2\lambda+2)p}[(2\lambda+2)p-1]}
{s^{(2\lambda+2)p}-a^{(2\lambda+2)p}}>0.
\end{align*}
This proves $f'(s)>0$. Using this fact, for $a^{(1-p)(2\lambda+2)}>t\geq b^{(1-p)(2\lambda+2)}$, we have
\begin{align*}
&t\Big(\int_{a}^{\min\big\{b,t^{-\frac{1}{(p-1)(2\lambda+2)}}\big\}}
x^{p-1+(2\lambda+1)p}dx\Big)^{1/p'}\\
&\quad=t\Big(\int_a^{t^{-\frac{1}{(p-1)(2\lambda+2)}}}x^{p-1+(2\lambda+1)p}dx\Big)^{1/p'}\\
&\quad=s^{(1-p)(2\lambda+2)}\Big(\int_a^sx^{p-1+(2\lambda+1)p}dx\Big)^{1/p'}\\
&\quad\leq b^{(1-p)(2\lambda+2)}\Big(\int_a^bx^{p-1+(2\lambda+1)p}dx\Big)^{1/p'},
\end{align*}
where $s=t^{-1/(p-1)(2\lambda+2)}$. Hence,
\begin{align*}
&\sup_{b^{(1-p)(2\lambda+2)}\leq t<a^{(1-p)(2\lambda+2)} }t\Big(\int_{\{x\in(a,b)\subset\mathbb{R}_+:x^{(2\lambda+2)(1-p)}>t\}}
x^{p-1+(2\lambda+1)p}dx\Big)^{1/p'}\\
&\quad=t\Big(\int_{a}^{\min\big\{b,t^{-\frac{1}{(p-1)(2\lambda+2)}}\big\}}
x^{p-1+(2\lambda+1)p}dx\Big)^{1/p'}\\
&\quad\leq b^{(1-p)(2\lambda+2)}\Big(\int_a^bx^{p-1+(2\lambda+1)p}dx\Big)^{1/p'},
\end{align*}
this, together with \eqref{local sup}, yields \eqref{vital for example}.

Now, we are in the position to prove $|x|^{p-1+(2\lambda+1)p}\in\widetilde{A}_{p,\lambda}^\mathcal{R}$. For any interval $(a,b)\subset\mathbb{R}_+$, by making use of \eqref{vital for example}, we have
\begin{align*}
&\frac{\Big(\int_a^bx^{p-1+(2\lambda+1)p}dx\Big)^{1/p}
\|\chi_{(a,b)}(\cdot)^{(2\lambda+2)(1-p)}\|_{L^{p',\infty}(\mathbb{R}_+,x^{p-1+(2\lambda+1)p}dx)}}
{\int_a^bx^{2\lambda+1}dx}\\
&\quad\leq\frac{b^{(1-p)(2\lambda+2)}\Big(\int_a^bx^{p-1+(2\lambda+1)p}dx\Big)(2\lambda+2)}
{b^{2\lambda+2}-a^{2\lambda+2}}\\
&\quad\sim\frac{b^{(1-p)(2\lambda+2)}(b^{(2\lambda+2)p}-a^{(2\lambda+2)p})}{p(b^{2\lambda+2}-a^{2\lambda+2})}\\
&\quad=\frac{b^{2\lambda+2}-a^{(2\lambda+2)p}b^{(p-1)(2\lambda+2)}}{p(b^{2\lambda+2}-a^{2\lambda+2})}\leq1.
\end{align*}
This proves $|x|^{p-1+(2\lambda+1)p}\in\widetilde{A}_{p,\lambda}^\mathcal{R}$.
\end{proof}

From Theorem \ref{compare weight constant1}, we see that $w$ satisfies the doubling condition when $w\in\widetilde{A}_{p,\lambda}^R$. A precise dependency of weight constant is provided below.
\begin{lemma}\label{quantitative doubling}
Let $1\leq p<\infty$, $\lambda>-1/2, \neq0$ and $\eta>1$. If $w\in\widetilde{A}_{p,\lambda}^\mathcal{R}$, then for any interval $I\subset\mathbb{R}_+$,
\begin{align*}
w(\eta I)\lesssim[w]_{\widetilde{A}_{p,\lambda}^\mathcal{R}}^pw(I).
\end{align*}
\end{lemma}
\begin{proof}
We only prove the case of $1<p<\infty$, since the case $p=1$ can be similarly proved. For $1<p<\infty$, using H\"{o}lder's inequality and $w\in\widetilde{A}_{p,\lambda}^\mathcal{R}$, it is direct that
\begin{align}\label{use for doubling}
&\Big(\frac{1}{\nu_\lambda}\int_I|f(x)|d\nu_\lambda(x)\Big)^p\\
&\quad=\Big(\frac{1}{\nu_\lambda}\int_{\mathbb{R}_+}|f(x)|\chi_I(x)x^{2\lambda+1}w(x)^{-1}w(x)dx\Big)^p
\nonumber\\
&\quad\leq\Big(\frac{1}{\nu_\lambda}\|f\|_{L^{p,1}(\mathbb{R}_+,wdx)}
\|\chi_I(\cdot)^{2\lambda+1}w^{-1}\|_{L^{p',\infty}(\mathbb{R}_+,wdx)}\Big)^p\nonumber\\
&\quad=\frac{\|f\|_{L^{p,1}(\mathbb{R}_+,wdx)}^p}{w(I)}w(I)
\frac{\|\chi_I(\cdot)^{2\lambda+1}w^{-1}\|_{L^{p',\infty}(\mathbb{R}_+,wdx)}^p}{\nu_\lambda(I)^p}
\nonumber\\
&\quad\leq[w]_{\widetilde{A}_{p,\lambda}^\mathcal{R}}^p\frac{\|f\|_{L^{p,1}(\mathbb{R}_+,wdx)}^p}{w(I)}
\nonumber.
\end{align}

Applying \eqref{use for doubling} to the function $f=\chi_I$ and putting $\eta I$ in the place of $I$ in \eqref{use for doubling} with $\eta>1$, we get
\begin{align*}
w(\eta I)\lesssim\Big(\frac{\nu_\lambda(\eta I)}{\nu_\lambda(I)}\Big)^p[w]_{\widetilde{A}_{p,\lambda}^\mathcal{R}}^pw(I)
\lesssim[w]_{\widetilde{A}_{p,\lambda}^\mathcal{R}}^pw(I),
\end{align*}
where we use $\|\chi_I\|_{L^{p,1}(\mathbb{R}_+,wdx)}=p\omega(I)^{1/p}$, the doubling property of $x^{2\lambda+1}dx$ in the first and second inequality, respectively.
\end{proof}

\begin{proof}[Proof of Theorem \ref{characterization restricted weak type of HL}]
We begin with the proof of
\begin{align*}
[w_1, w_2]_{\widetilde{A}_{p,\lambda}^\mathcal{R}}\lesssim\|\mathcal{M}_\lambda\|_{L^{p,1}(\mathbb{R}_+, w_1dx)\rightarrow L^{p,\infty}(\mathbb{R}_+, w_2dx)}.
\end{align*}
Considering two cases: $1<p<\infty$ and $p=1$.

\textbf{Case 1. $1<p<\infty$}:\quad Fixed an interval $I\subset\mathbb{R}_+$ and $\epsilon>1$. Since
\begin{align*}
\frac{1}{p}\|g\|_{L^{p',\infty}(\mathbb{R}_+,wdx)}&\leq\sup\Big\{\int_{\mathbb{R}_+}|f(x)g(x)|w(x)dx: \|f\|_{L^{p,1}(\mathbb{R}_+,wdx)}\leq1\Big\}\\
&\leq\|g\|_{L^{p',\infty}(\mathbb{R}_+,wdx)},
\end{align*}
there exists a non-negative function $f$ such that $\|f\|_{L^{p,1}(\mathbb{R}_+, w_1dx)}=1$ and
\begin{align*}
\frac{1}{\epsilon p}\|\chi_I(\cdot)^{2\lambda+1}w_1^{-1}\|_{L^{p',\infty}(\mathbb{R}_+, w_1dx)}&\leq\int_{\mathbb{R}_+}f(x)\chi_I(x)x^{2\lambda+1}w_1(x)^{-1}w_1(x)dx\\
&=\int_If(x)d\nu_\lambda(x).
\end{align*}
From this and the definition of $\mathcal{M}_\lambda$, we get
\begin{align*}
&\frac{1}{\epsilon p\nu_\lambda(I)}\|\chi_I(\cdot)^{2\lambda+1}w_1^{-1}\|_{L^{p',\infty}(\mathbb{R}_+,w_1dx)}\chi_I\\
&\quad\leq\frac{1}{\nu_\lambda(I)}\int_If(x)d\nu_\lambda(x)\chi_I\leq\mathcal{M}_\lambda f.
\end{align*}
Hence,
\begin{align*}
&\Big\|\frac{1}{\epsilon p\nu_\lambda(I)}\|\chi_I(\cdot)^{2\lambda+1}w_1^{-1}\|_{L^{p',\infty}(\mathbb{R}_+,w_1dx)}\chi_I\Big\|
_{L^{p,\infty}(\mathbb{R}_+, w_2dx)}\\
&\quad\leq\|\mathcal{M}_\lambda f\|_{L^{p,\infty}(\mathbb{R}_+, w_2dx)}\lesssim\|\mathcal{M}_\lambda\|_{L^{p,1}(\mathbb{R}_+,w_1dx)\rightarrow L^{p,\infty}(\mathbb{R}_+,w_2dx)},
\end{align*}
which yields that
\begin{equation*}
[w_1, w_2]_{\widetilde{A}_{p,\lambda}^\mathcal{R}}\lesssim\|\mathcal{M}_\lambda\|_{L^{p,1}(\mathbb{R}_+, w_1dx)\rightarrow L^{p,\infty}(\mathbb{R}_+, w_2dx)}.
\end{equation*}

\textbf{Case 2. $p=1$}:\quad Let $I$ be an interval in $\mathbb{R}_+$ and $t>0$. Since $\mathcal{M}_\lambda f(x)\geq\nu_\lambda(I)^{-1}\int_I|f|d\nu_\lambda$ for all $x\in I$ and
\begin{equation*}
tw_2(\{x\in\mathbb{R}_+:\mathcal{M}_\lambda f(x)>t\})\lesssim\int_{\mathbb{R}_+}|f(x)|w_1(x)dx,
\end{equation*}
we have that for all $t<\nu_\lambda(I)^{-1}\int_I|f|d\nu_\lambda$,
\begin{equation*}
w_2(I)\leq w_2(\{x\in\mathbb{R}_+:\mathcal{M}_\lambda f(x)>t\})\lesssim \frac{1}{t}\int_{\mathbb{R}_+}|f(x)|w_1(x)dx.
\end{equation*}
Replace $f$ in the above inequality by $\chi_E$, where $E$ is an arbitrary measurable subset of $I$. Then
\begin{align*}
\frac{1}{\nu_\lambda(I)}\int_I\chi_Ed\nu_\lambda\lesssim\frac{\int_{\mathbb{R}_+}\chi_E(x)w_1(x)dx}{w_2(I)},
\end{align*}
that is,
\begin{align}\label{use for A1}
\frac{\nu_\lambda(E)}{\nu_\lambda(I)}\lesssim\frac{w_1(E)}{w_2(I)}.
\end{align}
Let $\epsilon>0$. Denote $a:=\essi_{x\in I}\big(w_1(x)x^{-(2\lambda+1)}\big)$. Then there exists $E_\epsilon\subset I$ such that for any $x\in E_\epsilon$,
\begin{align*}
w_1(x)x^{-(2\lambda+1)}<a+\epsilon.
\end{align*}
Hence, \eqref{use for A1} implies that
\begin{align*}
\frac{1}{\nu_\lambda(I)}\int_Iw_2(x)dx&\lesssim\frac{1}{\nu_\lambda(E_\epsilon)}\int_{E_\epsilon}w_1(x)dx\\
&=\frac{1}{\nu_\lambda(E_\epsilon)}\int_{E_\epsilon}w_1(x)x^{-(2\lambda+1)}d\nu_\lambda(x)\\
&\leq a+\epsilon=\essi_{x\in I}\big(w_1(x)x^{-(2\lambda+1)}\big)+\epsilon,
\end{align*}
from which it follows that
\begin{equation*}
\frac{w_2(I)}{\nu_\lambda(I)}\lesssim \essi_{x\in I}w_1(x)x^{-(2\lambda+1)}.
\end{equation*}
This proves $(w_1, w_2)\in \widetilde{A}_{1,\lambda}^\mathcal{R}$.

Next, we proceed to prove $\|\mathcal{M}_\lambda\|_{L^{p,1}(\mathbb{R}_+,w_1dx)\rightarrow L^{p,\infty}(\mathbb{R}_+,w_2dx)}\lesssim[w_1,w_2]_{\widetilde{A}_{p,\lambda}^\mathcal{R}}$. We still use the notations used in the proof of Theorem \ref{mixed quantitative weighted strong estimates}. In virtue of covering lemma, H\"{o}lder's inequality and Lemma \ref{summation lemma}, we deduce that
\begin{align}\label{quantitative information}
w_2(F)&\leq\sum_{i=1}^{l}w_2(5I_{x_{j_i}})\\
&\leq t^{-p}\sum_{i=1}^lw_2(5I_{x_{j_i}})\Big(\frac{1}{\nu_{\lambda}(I_{x_{j_i}})}\int_{I_{x_{j_i}}}
|f(y)|y^{2\lambda+1}dy\Big)^p\nonumber\\
&=t^{-p}\sum_{i=1}^lw_2(5I_{x_{j_i}})\Big(\frac{\nu_{\lambda}(5I_{x_{j_i}})}{\nu_{\lambda}(I_{x_{j_i}})}\Big)^p
\Big(\frac{1}{\nu_{\lambda}(5I_{x_{j_i}})}
\int_{I_{x_{j_i}}}
|f(y)|y^{2\lambda+1}dy\Big)^p\nonumber\\
&\lesssim t^{-p}\sum_{i=1}^lw_2(5I_{x_{j_i}})
\frac{\|\chi_{5I_{x_{j_i}}}(\cdot)^{2\lambda+1}w_1^{-1}\|_{L^{p',\infty}(\mathbb{R}_+,w_1dx)}^p}
{\nu_{\lambda}(5I_{x_{j_i}})^{p}}\|f\chi_{I_{x_{j_i}}}\|_{L^{p,1}(\mathbb{R}_+,w_1dx)}^p\nonumber\\
&\leq t^{-p}[w_1,w_2]_{\widetilde{A}_{p,\lambda}^\mathcal{R}}^p\sum_{i=1}^l
\|f\chi_{I_{x_{j_i}}}\|_{L^{p,1}(\mathbb{R}_+,w_1dx)}^p\nonumber\\
&\lesssim t^{-p}[w_1,w_2]_{\widetilde{A}_{p,\lambda}^\mathcal{R}}^p\|f\|_{L^{p,1}(\mathbb{R}_+,w_1dx)}^p\nonumber.
\end{align}
Thus, the conclusion follows by taking the supremum over all compact $F\subset E_t$.
\end{proof}

\begin{proof}[Proof of Theorem \ref{another characterization restricted weak type}]
(1)$\Rightarrow$(2) follows by taking $f=\chi_J$ and using
\begin{align*}
\|\chi_J\|_{L^{p,1}(\mathbb{R}_+,w_1dx)}=pw_1(J)^{1/p}.
\end{align*}
(4)$\Rightarrow$(1) is proved in  Theorem \ref{characterization restricted weak type of HL}. It remains to demonstrate (2)$\Rightarrow$(3) and (3)$\Rightarrow$(4).

We first give the proof of (2)$\Rightarrow$(3). Fix an interval $I\subset\mathbb{R}_+$ and a measurable set $J\subset\mathbb{R}_+$ with $0<w_1(J)<\infty$. It is trivial that
\begin{align*}
\frac{\nu_\lambda(I\cap J)}{\nu_\lambda(I)}\chi_I=\Big(\frac{1}{\nu_\lambda(I)}\int_I\chi_Jd\nu_\lambda\Big)\chi_I\leq
\mathcal{M}_\lambda(\chi_J).
\end{align*}
It follows that
\begin{align*}
w_2(I)^{1/p}\frac{\nu_\lambda(I\cap J)}{\nu_\lambda(I)}&\sim\Big\|\frac{\nu_\lambda(I\cap J)}{\nu_\lambda(I)}\chi_I\Big\|_{L^{p,\infty}(\mathbb{R}_+,w_2dx)}\\
&\leq\|\mathcal{M}_\lambda(\chi_J)\|_{L^{p,\infty}(\mathbb{R}_+,w_2dx)}\lesssim w_1(J)^{1/p}.
\end{align*}
This proves $\|w_1,w_2\|_{\widetilde{A}_{p,\lambda}^\mathcal{R}}<\infty$.

We turn to (3)$\Rightarrow$(4). Fix an interval $I\subset\mathbb{R}_+$ and $t>0$. Denote
\begin{equation*}
J:=\{x\in I\subset\mathbb{R}_+: x^{2\lambda+1}w_1(x)^{-1}>t\}.
\end{equation*}
We claim that
\begin{align}\label{claim for 3-4}
tw_1(J)\leq\|w_1,w_2\|_{\widetilde{A}_{p,\lambda}^\mathcal{R}}\frac{w_1(J)^{1/p}\nu_\lambda(I)}{w_2(I)^{1/p}}.
\end{align}
Indeed, by Chebyshev's inequality and the definition $J$,
\begin{align}\label{use for claim 3-4}
tw_1(J)&=\int_Jtw_1(x)dx\\
&\leq\int_Jx^{2\lambda+1}w_1(x)^{-1}w_1(x)dx=\nu_\lambda(J)\nonumber.
\end{align}
Note that
\begin{equation}\label{sup compare}
\sup_{0<w_1(J)<\infty}\frac{\nu_\lambda(I\cap J)}{w_1(J)^{1/p}}=\sup_{J\subset I}\frac{\nu_\lambda(J)}{w_1(J)^{1/p}}.
\end{equation}
Then for any $J\subset I$, from the condition (3), we have
\begin{align*}
\nu_\lambda(J)\leq\|w_1,w_2\|_{\widetilde{A}_{p,\lambda}^\mathcal{R}}\frac{w_1(J)^{1/p}
\nu_\lambda(I)}{w_2(I)^{1/p}}.
\end{align*}
This, together with \eqref{use for claim 3-4}, allows us to get \eqref{claim for 3-4}.

Now, we return to prove (3)$\Rightarrow$(4). We prove it by considering two cases: $1<p<\infty$ and $p=1$. When $1<p<\infty$. In virtue of \eqref{claim for 3-4} and the definition $J$, we have
\begin{align*}
w_2(I)^{1/p}\frac{tw_1(J)^{1/p'}}{\nu_\lambda(I)}\leq\|w_1,w_2\|_{\widetilde{A}_{p,\lambda}^\mathcal{R}}.
\end{align*}
Taking the supremum over all $t>0$, and then taking the supremum over all intervals $I\subset\mathbb{R}_+$, we get
\begin{align}\label{compare}
[w_1,w_2]_{\widetilde{A}_{p,\lambda}^\mathcal{R}}\leq\|w_1,w_2\|_{\widetilde{A}_{p,\lambda}^\mathcal{R}}.
\end{align}
When $p=1$, for $t<\|\chi_Iw_1^{-1}(\cdot)^{2\lambda+1}\|_{L^\infty(\mathbb{R}_+, w_1dx)}$, according to \eqref{claim for 3-4} again,
\begin{align*}
t\leq\|w_1,w_2\|_{\widetilde{A}_{p,\lambda}^\mathcal{R}}\frac{\nu_\lambda(I)}{w_2(I)}.
\end{align*}
Thus, by letting $t$ increase to $\|\chi_Iw_1^{-1}(\cdot)^{2\lambda+1}\|_{L^\infty(\mathbb{R}_+, w_1dx)}$, it yields that
\begin{align*}
w_2(I)\frac{\|\chi_Iw_1^{-1}(\cdot)^{2\lambda+1}\|_{L^\infty(\mathbb{R}_+, w_1dx)}}{\nu_\lambda(I)}\leq\|w_1,w_2\|_{\widetilde{A}_{1,\lambda}^\mathcal{R}}.
\end{align*}
Taking the supremum over all intervals $I\subset \mathbb{R}_+$, we obtain
\begin{align*}
[w_1,w_2]_{\widetilde{A}_{1,\lambda}^\mathcal{R}}\leq\|w_1,w_2\|_{\widetilde{A}_{1,\lambda}^\mathcal{R}}.
\end{align*}
\end{proof}

From the proof of Theorem \ref{another characterization restricted weak type}, we know that $[w_1,w_2]_{\widetilde{A}_{p,\lambda}^\mathcal{R}}$ and $\|w_1,w_2\|_{\widetilde{A}_{p,\lambda}^\mathcal{R}}$ is equivalent. The following proposition further demonstrates the equivalence between the two from a quantitative perspective.
\begin{proposition}\label{weight constant compare2}
Let $1\leq p<\infty$, $\lambda>-1/2, \neq0$ and $(w_1,w_2)$ be a pair of weights. Then
\begin{align*}
[w_1,w_2]_{\widetilde{A}_{p,\lambda}^\mathcal{R}}\leq\|w_1,w_2\|_{\widetilde{A}_{p,\lambda}^\mathcal{R}}\leq p
[w_1,w_2]_{\widetilde{A}_{p,\lambda}^\mathcal{R}}.
\end{align*}
\end{proposition}
\begin{proof}
In view of \eqref{compare}, it suffices to prove $\|w_1,w_2\|_{\widetilde{A}_{p,\lambda}^\mathcal{R}}\leq p
[w_1,w_2]_{\widetilde{A}_{p,\lambda}^\mathcal{R}}$. Recall the following Kolmogorov's inequality: for any measurable set $E\subset\mathbb{R}_+$,
\begin{align*}
\int_E|g(x)|w(x)dx\leq pw(E)^{1/p}\|g\|_{L^{p',\infty}(\mathbb{R}_+,wdx)},~1\leq p<\infty.
\end{align*}
Fix an interval $I\subset\mathbb{R}_+$, for any measurable set $J\subset I$, by taking $E=J$, $g=\chi_I|\cdot|^{2\lambda+1}w_1^{-1}$, $w=w_1$ in the Kolmogorov inequality, we get
\begin{align*}
\frac{\nu_\lambda(J)}{w_1(J)^{1/p}}\leq p\big\|\chi_I|\cdot|^{2\lambda+1}w_1^{-1}\big\|_{L^{p',\infty}(\mathbb{R}_+,w_1dx)},
\end{align*}
which, together with \eqref{sup compare}, implies that
\begin{align*}
\sup_{0<w_1(J)<\infty}\frac{\nu_\lambda(I\bigcap J)}{w_1(J)^{1/p}}=\sup_{J\subset I}\frac{\nu_\lambda(J)}{w_1(J)^{1/p}}\leq p\big\|\chi_I|\cdot|^{2\lambda+1}w_1^{-1}\big\|_{L^{p',\infty}(\mathbb{R}_+,w_1dx)}.
\end{align*}
Multiply both sides by $w_2(I)^{1/p}/\nu_\lambda(I)$ and then take the supremum over all intervals $I\subset \mathbb{R}_+$, we get the desired result.
\end{proof}
\begin{remark}
In particular, when $w_1=w_2=w$, we denote $\|w,w\|_{\widetilde{A}_{p,\lambda}^R}$ by $\|w\|_{\widetilde{A}_{p,\lambda}^R}$.
\end{remark}
\begin{proof}[Proof of Theorem \ref{quantitative two weight restricted weak type of sparse operator}]
Since $w_2\in \widetilde{A}_{\infty,\lambda-\frac{1}{2}}$, there exists $r\geq1$ such that $w_2\in \widetilde{A}_{r,\lambda-\frac{1}{2}}^\mathcal{R}$ by Theorem \ref{compare weight constant1}. From the proof of Theorem \ref{another characterization restricted weak type}, we know that
\begin{align*}
\|w_2\|_{\widetilde{A}_{r,\lambda}^\mathcal{R}}=\sup_{I\subset\mathbb{R}_+}\sup_{J\subset I}\frac{\nu_\lambda(J)}{\nu_\lambda(I)}\frac{w_2(I)^{1/r}}{w_2(J)^{1/r}}.
\end{align*}
Using $\mathcal{S}$ is an $\eta$-sparse family, for each interval $I\in\mathcal{S}$, there is a subset $E_I$ of $I$ such that
\begin{align}\label{sparse property}
\frac{w_2(I)}{w_2(E_I)}=\frac{\mu(I)^r}{\mu(E_I)^r}\frac{\mu(E_I)^r}{\mu(I)^r}\frac{w_2(I)}{w_2(E_I)}\leq
\eta^{-r}\|w_2\|_{\widetilde{A}_{r,\lambda-\frac{1}{2}}^\mathcal{R}}^r.
\end{align}

By duality and \eqref{sparse property},
\begin{align*}
&\|\mathcal{A}_\mathcal{S}f\|_{L^{p,\infty}(\mathbb{R}_+,w_2dx)}\\
&\quad\sim
\underset{\|g\|_{L^{p',1}(\mathbb{R}_+,w_2dx)=1}}{\sup_{0\leq g\in L^{p',1}(\mathbb{R}_+,w_2dx)}}\int_{\mathbb{R}_+}
\mathcal{A}_\mathcal{S}f(x)g(x)w_2(x)dx\\
&\quad=\underset{\|g\|_{L^{p',1}(\mathbb{R}_+,w_2dx)=1}}{\sup_{0\leq g\in L^{p',1}(\mathbb{R}_+,w_2dx)}}\sum_{I\in\mathcal{S}}\Big(\frac{1}{\mu(I)}\int_I|f(x)|d\mu(x)\Big)
\Big(\int_Ig(x)w_2(x)dx\Big)\\
&\quad=\underset{\|g\|_{L^{p',1}(\mathbb{R}_+,w_2dx)=1}}{\sup_{0\leq g\in L^{p',1}(\mathbb{R}_+,w_2dx)}}\sum_{I\in\mathcal{S}}\Big(\frac{1}{\mu(I)}\int_I|f(x)|d\mu(x)\Big)
\Big(\frac{1}{w_2(I)}\int_Ig(x)w_2(x)dx\Big)\\
&\qquad\qquad\qquad\qquad\qquad\times\frac{w_2(I)}{w_2(E_I)}w_2(E_I)\\
&\quad\lesssim\|w_2\|_{\widetilde{A}_{r,\lambda-\frac{1}{2}}^\mathcal{R}}^r\underset{\|g\|_{L^{p',1}
(\mathbb{R}_+,w_2dx)=1}}{\sup_{0\leq g\in L^{p',1}(\mathbb{R}_+,w_2dx)}}
\sum_{I\in\mathcal{S}}\int_{E_I}\mathcal{M}_{\mu}^{\mathcal{S}}f(x)M_{w_2}^{\mathcal{S}}
g(x)w_2(x)dx\\
&\quad\lesssim\|w_2\|_{\widetilde{A}_{r,\lambda-\frac{1}{2}}^\mathcal{R}}^r\underset{\|g\|_{L^{p',1}
(\mathbb{R}_+,w_2dx)=1}}{\sup_{0\leq g\in L^{p',1}(\mathbb{R}_+,w_2dx)}}
\int_{\mathbb{R}_+}\mathcal{M}_{\mu}^{\mathcal{S}}f(x)M_{w_2}^{\mathcal{S}}
g(x)w_2(x)dx\\
&\quad\leq\|w_2\|_{\widetilde{A}_{r,\lambda-\frac{1}{2}}^\mathcal{R}}^r\underset{\|g\|_{L^{p',1}
(\mathbb{R}_+,w_2dx)=1}}{\sup_{0\leq g\in L^{p',1}(\mathbb{R}_+,w_2dx)}}
\|\mathcal{M}_{\mu}^{\mathcal{S}}f\|_{L^{p,\infty}(\mathbb{R}_+.w_2dx)}\|M_{w_2}^{\mathcal{S}}
g\|_{L^{p',1}(\mathbb{R}_+,w_2dx)},
\end{align*}
where
\begin{equation*}
\mathcal{M}_{\mu}^\mathcal{S}f(x):=\sup_{I\ni x,I\in\mathcal{S}}\frac{1}{\mu(I)}\int_I|f(x)|d\mu(x)
\end{equation*}
and
\begin{equation*}
M_{w}^\mathcal{S}g(x):=\sup_{I\ni x,I\in\mathcal{S}}\frac{1}{w(I)}\int_I|g(x)|w(x)dx.
\end{equation*}

It is easy to see that $\|M_{w_2}^\mathcal{S}g\|_{L^{\infty}(\mathbb{R}_+,w_2dx)}\leq\|g\|_{L^{\infty}(\mathbb{R}_+,w_2dx)}$. In \cite{LerNa}, the authors proved that for any weight $w$,
\begin{align*}
\|M_{w}^\mathcal{S}g\|_{L^{1,\infty}(\mathbb{R}_+,wdx)}\leq\|g\|_{L^{1}(\mathbb{R}_+,wdx)}.
\end{align*}
Then by the Marcinkiewicz interpolation theorem \cite[Theorem 4.13]{BeSh}, we obtain
\begin{align}\label{M_w}
\|M_{w_2}^\mathcal{S}g\|_{L^{p,1}(\mathbb{R}_+,w_2dx)}\lesssim\|g\|_{L^{p,1}(\mathbb{R}_+,w_2dx)},
\end{align}
where the implicit constant is independent of $w_2$. From this, Theorem \ref{characterization restricted weak type of HL} and Proposition \ref{weight constant compare2}, we conclude that
\begin{align*}
\|\mathcal{A}_\mathcal{S}f\|_{L^{p,\infty}(\mathbb{R}_+,w_2dx)}\lesssim
[w_2]_{\widetilde{A}_{r,\lambda-\frac{1}{2}}^{\mathcal{R}}}^{r}
[w_1,w_2]_{\widetilde{A}_{p,\lambda-\frac{1}{2}}^\mathcal{R}}\|f\|_{L^{p,1}(\mathbb{R}_+,w_1dx)}.
\end{align*}
\end{proof}

\begin{proof}[Proof of Theorem \ref{characterization two weight restricted weak type estimate of Orlicz maximal}]
We first prove the necessity. Let $x\in\mathbb{R}_+$. Fix an interval $I\subset\mathbb{R}_+$ that contains $x$ and a measurable set $J\subset I$. By the definition of $\mathcal{M}_{\phi,\nu_\lambda}$, we see that
\begin{align*}
\|\chi_J\|_{\phi,I,\nu_\lambda}\chi_I(x)\leq\mathcal{M}_{\phi,\nu_\lambda}(\chi_J)(x).
\end{align*}
Taking the $\|\cdot\|_{L^{p,\infty}(\mathbb{R}_+,w_2dx)}$ on both sides and using the assumption, we get
\begin{align*}
\|\chi_J\|_{\phi,I,\nu_\lambda}\|\chi_I\|_{L^{p,\infty}(\mathbb{R}_+,w_2dx)}&\leq
\|\mathcal{M}_{\phi,\nu_\lambda}(\chi_J)\|_{L^{p,\infty}(\mathbb{R}_+,w_2dx)}\\
&\lesssim\|\chi_J\|_{L^{p,1}(\mathbb{R}_+,w_1dx)}\sim w_1(J)^{1/p},
\end{align*}
which implies that $\|\chi_J\|_{\phi,I,\nu_\lambda}(w_2(I)/w_1(J))^{1/p}\lesssim1$. So $\|w_1,w_2\|_{\widetilde{A}_{p,\lambda,\phi}}<\infty$.

Next, we concentrate on the proof of the sufficiency. We only prove the quantitative weighted bounds provided that $(w_1,w_2)\in \widetilde{A}_{p,\lambda,\phi}^\mathcal{R}$ and $w_2\in \widetilde{A}_{\infty,\lambda}^\mathcal{R}$, since the case of $(w_1,w_2)\in \widetilde{A}_{p,\lambda,\phi}^\mathcal{R}$ and $w_2$ is a doubling weight can be similar proved. Define
\begin{equation*}
M_{w_1,w_2}f(x):=\sup_{I\ni x,I\subset\mathbb{R}_+}\frac{1}{w_2(I)}\int_I|f(x)|w_1(x)dx.
\end{equation*}
In virtue of Vitali covering lemma and Lemma \ref{quantitative doubling}, adopting similar techniques as in the proof of Theorem \ref{characterization restricted weak type of HL}, there exists $r\geq1$ such that $w_2\in\widetilde{A}_{r,\lambda}^\mathcal{R}$ and
\begin{equation}\label{weak type (1,1) of two-weight HL}
\|M_{w_1,w_2}f\|_{L^{1,\infty}(\mathbb{R}_+,w_2dx)}\lesssim[w_2]_{A_{r,\lambda}^\mathcal{R}}^r
\|f\|_{L^1(\mathbb{R}_+,w_1dx)}.
\end{equation}
Note that for every interval $I\subset\mathbb{R}_+$ and every measurable set $J$, $(w_1,w_2)\in\widetilde{A}_{p,\lambda,\phi}^\mathcal{R}$ shows that
\begin{align*}
\|\chi_J\cap \chi_I\|_{\phi,I,\nu_\lambda}\leq\|w_1,w_2\|_{\widetilde{A}_{p,\lambda,\phi}^\mathcal{R}}
\Big(\frac{w_1(I\cap J)}{w_2(I)}\Big)^{1/p}.
\end{align*}
From this, one can easily get
\begin{align}\label{Orlicz maximal domanated by two-weight HL}
\mathcal{M}_{\phi,\nu_\lambda}(\chi_J)\leq\|w_1,w_2\|_{\widetilde{A}_{p,\lambda,\phi}^\mathcal{R}}
M_{w_1,w_2}(\chi_J)^{1/p}.
\end{align}
Thus, by \eqref{weak type (1,1) of two-weight HL}, we have
\begin{align}\label{conclusion for characteristic function}
\|\mathcal{M}_{\phi,\nu_\lambda}(\chi_J)\|_{L^{p,\infty}(\mathbb{R}_+,w_2dx)}&\leq
\|w_1,w_2\|_{\widetilde{A}_{p,\lambda,\phi}^\mathcal{R}}\|M_{w_1,w_2}(\chi_J)^{1/p}\|_{L^{p,\infty}(\mathbb{R}_+,w_2dx)}\\
&=\|w_1,w_2\|_{\widetilde{A}_{p,\lambda,\phi}^\mathcal{R}}\|M_{w_1,w_2}(\chi_J)\|_{L^{1,\infty}(\mathbb{R}_+,w_2dx)}^{1/p}\nonumber\\
&\lesssim [w_2]_{A_{r,\lambda}^\mathcal{R}}^{r/p}\|w_1,w_2\|_{\widetilde{A}_{p,\lambda,\phi}^\mathcal{R}}w_1(J)^{1/p}.\nonumber
\end{align}
This proves the conclusion for characteristic function.

For more general functions $f\in L^{p,1}(\mathbb{R}_+,w_1dx)$, denote $J_k:=\{x\in\mathbb{R}_+:2^k<|f(x)|\leq2^{k+1}\}, k\in\mathbb{Z}$. Therefore, by \eqref{conclusion for characteristic function} and $|f|\leq2\sum_{k\in\mathbb{Z}}2^k\chi_{J_k}$, we deduce that
\begin{align}\label{quanality weighted estimate of Orlicz maximal}
&\|\mathcal{M}_{\phi,\nu_\lambda}f\|_{L^{p,\infty}(\mathbb{R}_+,w_2dx)}\\
&\quad\lesssim\sum_{k\in\mathbb{Z}}2^k
\|\mathcal{M}_{\phi,\nu_\lambda}(\chi_{J_k})\|_{L^{p,\infty}(\mathbb{R}_+,w_2dx)}\nonumber\\
&\quad\lesssim[w_2]_{A_{r,\lambda}^\mathcal{R}}^{r/p}
\|w_1,w_2\|_{\widetilde{A}_{p,\lambda,\phi}^\mathcal{R}}\sum_{k\in\mathbb{Z}}2^k
w_1(\{x\in\mathbb{R}_+:|f(x)|>2^k\})^{1/p}\nonumber\\
&\quad\lesssim[w_2]_{A_{r,\lambda}^\mathcal{R}}^{r/p}
\|w_1,w_2\|_{\widetilde{A}_{p,\lambda,\phi}^\mathcal{R}}\sum_{k\in\mathbb{Z}}\int_{2^{k-1}}^{2^k}
w_1(\{x\in\mathbb{R}_+:|f(x)|>t\})^{1/p}dt\nonumber\\
&\quad\lesssim[w_2]_{A_{r,\lambda}^\mathcal{R}}^{r/p}
\|w_1,w_2\|_{\widetilde{A}_{p,\lambda,\phi}^\mathcal{R}}\|f\|_{L^{p,1}(\mathbb{R}_+,w_1dx)}\nonumber.
\end{align}
\end{proof}

Now, we are in the position to prove Theorem \ref{quantitative two weight restricted weak type estimate for sparse operator of commutator 2}. We prove it by dividing into two parts. The first part is the following quantitative restricted two-weight weak type estimate for $(\mathcal{A}_{\mathcal{S},b}^m)^{\ast}$.
\begin{theorem}\label{quantitative two weight restricted weak type estimate for sparse operator of commutator 1}
Let $1<p<\infty$, $\Phi(t)=t\log(e+t)^m$ with $m\in\mathbb{Z}^+$ and $\mathcal{S}$ be an $\eta$-sparse family. If $b\in BMO_{\Delta_\lambda}(\mathbb{R}_+)$, $w_2\in\widetilde{A}_{\infty,\lambda-\frac{1}{2}}$ and $(w_1,w_2)\in\widetilde{A}_{p,\lambda-\frac{1}{2},\Phi}^\mathcal{R}$. Then there exists $r\geq1$ such that
\begin{align*}
\|(\mathcal{A}_{\mathcal{S},b}^m)^{\ast}f\|_{L^{p,\infty}(\mathbb{R}_+,w_2dx)}
\lesssim\|b\|_{BMO_{\Delta_\lambda}(\mathbb{R}_+)}^m
[w_2]_{\widetilde{A}_{r,\lambda-\frac{1}{2}}^\mathcal{R}}^{r+\frac{r}{p}}
\|w_1,w_2\|_{\widetilde{A}_{p,\lambda-\frac{1}{2},\Phi}^\mathcal{R}}\|f\|_{L^{p,1}(\mathbb{R}_+,w_1dx)}.
\end{align*}
\end{theorem}
\begin{proof}
By duality,
\begin{align}\label{duality for sparse operator of commutator 1}
&\|(\mathcal{A}_{\mathcal{S},b}^m)^{\ast}f\|_{L^{p,\infty}(\mathbb{R}_+,w_2dx)}\\
&\quad\sim\underset{\|g\|_{L^{p',1}(\mathbb{R}_+,w_2dx)=1}}{\sup_{0\leq g\in L^{p',1}(\mathbb{R}_+,w_2dx)}}\int_{\mathbb{R}_+}
(\mathcal{A}_{\mathcal{S},b}^m)^{\ast}f(x)g(x)w_2(x)dx\nonumber\\
&\quad=\underset{\|g\|_{L^{p',1}(\mathbb{R}_+,w_2dx)=1}}{\sup_{0\leq g\in L^{p',1}(\mathbb{R}_+,w_2dx)}}\sum_{I\in\mathcal{S}}
\Big(\frac{1}{\mu(I)}\int_I|b(x)-b_{I,\mu}|^m|f(x)|d\mu(x)\Big)\nonumber\\
&\qquad\qquad\qquad\qquad\qquad\times\Big(\int_{I}g(x)w_2(x)dx\Big).\nonumber
\end{align}

Since $w_2\in\widetilde{A}_{\infty,\lambda-\frac{1}{2}}$, there exists $r\geq1$ such that $w_2\in\widetilde{A}_{r,\lambda-\frac{1}{2}}^\mathcal{R}$. Then by \eqref{sparse property}, generalized H\"{o}lder's inequality and John-Nirenberg's inequality, we have
\begin{align*}
&\sum_{I\in\mathcal{S}}
\Big(\frac{1}{\mu(I)}\int_I|b(x)-b_{I,\mu}|^m|f(x)|d\mu(x)\Big)\Big(\int_{I}g(x)w_2(x)dx\Big)\\
&\quad\lesssim\sum_{I\in\mathcal{S}}\||b-b_{I,\mu}|^m\|_{\exp L^{1/m},I,\mu}\|f\|_{L(\log L)^m,I,\mu}
\frac{1}{w_2(I)}\int_Ig(x)w_2(x)dxw_2(I)\\
&\quad\lesssim\|b\|_{BMO_{\Delta_\lambda}(\mathbb{R}_+)}^m
[w_2]_{\widetilde{A}_{r,\lambda-\frac{1}{2}}^\mathcal{R}}^r
\sum_{I\in\mathcal{S}}\|f\|_{L(\log L)^m,I,\mu}
\frac{1}{w_2(I)}\int_Ig(x)w_2(x)dxw_2(E_I)\\
&\quad\leq\|b\|_{BMO_{\Delta_\lambda}(\mathbb{R}_+)}^m
[w_2]_{\widetilde{A}_{r,\lambda-\frac{1}{2}}^\mathcal{R}}^r
\sum_{I\in\mathcal{S}}\int_{E_I}\mathcal{M}_{L(\log L)^m,\mu}f(x)M_{w_2}^\mathcal{S}g(x)w_2(x)dx\\
&\quad\lesssim\|b\|_{BMO_{\Delta_\lambda}(\mathbb{R}_+)}^m
[w_2]_{\widetilde{A}_{r,\lambda-\frac{1}{2}}^\mathcal{R}}^r\|\mathcal{M}_{L(\log L)^m,\mu}f\|_{L^{p,\infty}(\mathbb{R}_+,w_2dx)}\|M_{w_2}^\mathcal{S}g\|_{L^{p',1}(\mathbb{R}_+,w_2dx)}
\end{align*}

In view of \eqref{quantitative two weight restricted weak type estimate of Orlicz maximal}, \eqref{M_w} and \eqref{duality for sparse operator of commutator 1}, we deduce that
\begin{align*}
\|(\mathcal{A}_{\mathcal{S},b}^m)^{\ast}f\|_{L^{p,\infty}(\mathbb{R}_+,w_2dx)}
\lesssim\|b\|_{BMO_{\Delta_\lambda}(\mathbb{R}_+)}^m
[w_2]_{\widetilde{A}_{r,\lambda-\frac{1}{2}}^\mathcal{R}}^{r+\frac{r}{p}}
\|w_1,w_2\|_{\widetilde{A}_{p,\lambda-\frac{1}{2},\Phi}^\mathcal{R}}\|f\|_{L^{p,1}(\mathbb{R}_+,w_1dx)}.
\end{align*}
\end{proof}

Next, we turn to deal with $\mathcal{A}_{\mathcal{S},b}^m$. We need the following lemma.
\begin{lemma}\label{Orlicz norm involing A}
Let $b\in BMO_{\Delta_\lambda}(\mathbb{R}_+)$ and $w\in\widetilde{A}_{\infty,\lambda-\frac{1}{2}}$ with $\lambda>-1/2,\neq0$. Then
\begin{align*}
\|b-b_{I,\mu}\|_{\exp L,I,w}\lesssim[w]_{\widetilde{A}_{\infty,\lambda-\frac{1}{2}}}
\|b\|_{BMO_{\Delta_\lambda}(\mathbb{R}_+)}.
\end{align*}
Moreover,
\begin{align*}
\||b-b_{I,\mu}|^j\|_{\exp L^{\frac{1}{j}},I,w}\lesssim
[w]_{\widetilde{A}_{\infty,\lambda-\frac{1}{2}}}^j\|b\|_{BMO_{\Delta_\lambda}(\mathbb{R}_+)}^j,~j>0.
\end{align*}
\end{lemma}
\begin{proof}
From the definition of Luxemburg norm, it suffices to demonstrate that
\begin{align*}
\frac{1}{w(I)}\int_I\left(e^{\frac{|b(x)-b_{I,\mu}|}{c
[w]_{\widetilde{A}_{\infty,\lambda-\frac{1}{2}}}\|b\|_{BMO_{\Delta_\lambda}(\mathbb{R}_+)}}}-1\right)w(x)dx
\leq1.
\end{align*}
Applying layer cake formula, for $\gamma>0$, we have the following identity,
\begin{align}\label{layer cake formula}
&\frac{1}{w(I)}\int_I\Big(e^{\frac{|b(x)-b_{I,\mu}|}{\gamma}}-1\Big)w(x)dx\\
&\quad=\frac{1}{w(I)}\int_0^\infty e^t
w(\{x\in I:|b(x)-b_{I,\mu}|>\gamma t\})dt\nonumber.
\end{align}
Then by Corollary \ref{Ratio property} and \eqref{John-Nirengerg inequality}, we have
\begin{align*}
\frac{w(\{x\in I:|b(x)-b_{I,\mu}|>\gamma t\})}{w(I)}&\leq 2\left(\frac{\mu(\{x\in I:|b(x)-b_{I,\mu}|>\gamma t\})}{\mu(I)}\right)^{\frac{1}{c_\lambda[w]_{\widetilde{A}_{\infty,\lambda-\frac{1}{2}}}}}\\
&\leq c_3e^{-\frac{\gamma t}{c_2c_\lambda[w]_{\widetilde{A}_{\infty,\lambda-\frac{1}{2}}}
\|b\|_{BMO_{\Delta_\lambda}(\mathbb{R}_+)}}}.
\end{align*}
Plug this estimate into \eqref{layer cake formula}, and take $\gamma=(c_3+2)c_2c_\lambda[w]_{\widetilde{A}_{\infty,\lambda-\frac{1}{2}}}
\|b\|_{BMO_{\Delta_\lambda}(\mathbb{R}_+)}$, we further have
\begin{align*}
\frac{1}{w(I)}\int_I\Big(e^{\frac{|b(x)-b_{I,\mu}|}{\gamma}}-1\Big)w(x)dx\leq c_3\int_0^\infty
e^{-t(1+c_3)}dt\leq1.
\end{align*}

Finally, the second result follows by the first result and $\||b-b_{I,\mu}|^j\|_{\exp L^{\frac{1}{j}},I,w}=\|b-b_{I,\mu}\|_{\exp L,I,w}^j$.
\end{proof}

\begin{theorem}\label{quantitative two weight restricted weak type estimate for sparse operator of commutator 11}
Let $1<p<\infty$, $m\in\mathbb{Z}^+$ and $\mathcal{S}$ be an $\eta$-sparse family. If $w_2\in\widetilde{A}_{\infty,\lambda-\frac{1}{2}}$, $(w_1,w_2)\in\widetilde{A}_{p,\lambda-\frac{1}{2}}^\mathcal{R}$ and $b\in BMO_{\Delta_\lambda}(\mathbb{R}_+)$. Then there exists $r\geq1$ such that
\begin{align*}
&\|\mathcal{A}_{\mathcal{S},b}^mf\|_{L^{p,\infty}(\mathbb{R}_+,w_2dx)}\\
&\quad\lesssim\|b\|_{BMO_{\Delta_\lambda}(\mathbb{R}_+)}^m
[w_2]_{\widetilde{A}_{\infty,\lambda-\frac{1}{2}}}^m
[w_2]_{\widetilde{A}_{r,\lambda-\frac{1}{2}}^\mathcal{R}}^{r}
[w_1,w_2]_{\widetilde{A}_{p,\lambda-\frac{1}{2}}^\mathcal{R}}\|f\|_{L^{p,1}(\mathbb{R}_+,w_1dx)}.
\end{align*}
\end{theorem}
\begin{proof}
By duality,
\begin{align*}
&\|\mathcal{A}_{\mathcal{S},b}^mf\|_{L^{p,\infty}(\mathbb{R}_+,w_2dx)}\\
&\quad\sim\underset{\|g\|_{L^{p',1}(\mathbb{R}_+,w_2dx)=1}}{\sup_{0\leq g\in L^{p',1}(\mathbb{R}_+,w_2dx)}}\int_{\mathbb{R}_+}
\mathcal{A}_{\mathcal{S},b}^mf(x)g(x)w_2(x)dx\nonumber\\
&\quad=\underset{\|g\|_{L^{p',1}(\mathbb{R}_+,w_2dx)=1}}{\sup_{0\leq g\in L^{p',1}(\mathbb{R}_+,w_2dx)}}\sum_{I\in\mathcal{S}}
\Big(\frac{1}{w_2(I)}\int_I|b(x)-b_{I,\mu}|^m|g(x)|w_2(x)dx\Big)\nonumber\\
&\qquad\qquad\qquad\qquad\qquad\times\Big(\frac{1}{\mu(I)}\int_{I}|f(x)|d\mu(x)\Big)w_2(I).\nonumber
\end{align*}

Since $w_2\in\widetilde{A}_{\infty,\lambda-\frac{1}{2}}$, there exists $r\geq1$ such that $w_2\in\widetilde{A}_{r,\lambda-\frac{1}{2}}^\mathcal{R}$. Then by generalized H\"{o}lder's inequality, Theorem \ref{Orlicz norm involing A} and \eqref{sparse property}, we deduce that
\begin{align*}
&\sum_{I\in\mathcal{S}}
\Big(\frac{1}{w_2(I)}\int_I|b(x)-b_{I,\mu}|^m|g(x)|w_2(x)dx\Big)
\Big(\frac{1}{\mu(I)}\int_{I}|f(x)|d\mu(x)\Big)w_2(I)\\
&\lesssim\sum_{I\in\mathcal{S}}\||b-b_{I,\mu}|^m\|_{\exp L^{1/m},I,w_2}\|g\|_{L(\log L)^m,I,w_2}
\frac{1}{\mu(I)}\int_I|f(x)|d\mu(x)w_2(I)\\
&\lesssim\|b\|_{BMO_{\Delta_\lambda}(\mathbb{R}_+)}^m
[w_2]_{\widetilde{A}_{r,\lambda-\frac{1}{2}}^\mathcal{R}}^r[w_2]_{A_{\infty,\lambda-\frac{1}{2}}}^m
\sum_{I\in\mathcal{S}}\|g\|_{L(\log L)^m,I,w_2}
\frac{1}{\mu(I)}\int_I|f(x)|d\mu(x)w_2(E_I)\\
&\leq\|b\|_{BMO_{\Delta_\lambda}(\mathbb{R}_+)}^m
[w_2]_{\widetilde{A}_{r,\lambda-\frac{1}{2}}^\mathcal{R}}^r[w_2]_{A_{\infty,\lambda-\frac{1}{2}}}^m
\sum_{I\in\mathcal{S}}\int_{E_I}\mathcal{M}_{L(\log L)^m,w_2}^{\mathcal{S}}g(x)M_{\mu}f(x)w_2(x)dx\\
&\lesssim\|b\|_{BMO_{\Delta_\lambda}(\mathbb{R}_+)}^m
[w_2]_{\widetilde{A}_{r,\lambda-\frac{1}{2}}^\mathcal{R}}^r[w_2]_{A_{\infty,\lambda-\frac{1}{2}}}^m
\|\mathcal{M}_{L(\log L)^m,w_2}g\|_{L^{p',1}(\mathbb{R}_+,w_2dx)}\|M_{\mu}f\|_{L^{p,\infty}(\mathbb{R}_+,w_2dx)},
\end{align*}
where
\begin{align*}
\mathcal{M}_{L(\log L)^m,w}^{\mathcal{S}}(f):=\sup_{Q\in\mathcal{S}}\|f\|_{L(\log L)^m,Q,w}\chi_Q.
\end{align*}

Now, we illustrate that for any weight $w$,
\begin{align}\label{Lorentz estimate for weighted Orlicz maximal}
\|\mathcal{M}_{L(\log L)^m,w}^\mathcal{S}g\|_{L^{p,1}(\mathbb{R}_+,wdx)}\lesssim\|g\|_{L^{p,1}(\mathbb{R}_+,wdx)},
\end{align}
with constant independent of $w$. One can check that
\begin{align*}
\|g\|_{L(\log L)^m,I,w}\lesssim\Big(\frac{1}{w(I)}\int_I|g(x)|^pw(x)dx\Big)^{1/p}\leq
\|g\|_{L^{\infty}(\mathbb{R}_+,wdx)},~1<p<\infty,
\end{align*}
so $\|\mathcal{M}_{L(\log L)^m,w}^\mathcal{S}\|_{L^{\infty}(\mathbb{R}_+,wdx)}\lesssim\|g\|_{L^{\infty}
(\mathbb{R}_+,wdx)}$, where the implicit constant is independent of $w$. Applying the same technique as in \cite[Lemma 2.6]{Cr}, we can prove
\begin{align*}
\|\mathcal{M}_{L(\log L)^m,w}^\mathcal{S}\|_{L^{r}(\mathbb{R}_+,wdx)}\lesssim\|g\|_{L^{r}(\mathbb{R}_+,wdx)},~
1<r<p,
\end{align*}
where the implicit constant is also independent of $w$. By Marcinkiewicz's interpolation theorem,
$$\|\mathcal{M}_{L(\log L)^m,w}g\|_{L^{p,1}(\mathbb{R}_+,wdx)}\lesssim\|g\|_{L^{p,1}(\mathbb{R}_+,wdx)},$$
where the implicit constant is independent of $w$.

Finally, using \eqref{Lorentz estimate for weighted Orlicz maximal} and Theorem \ref{characterization restricted weak type of HL}, we conclude that
\begin{align*}
&\|\mathcal{A}_{\mathcal{S},b}^mf\|_{L^{p,\infty}(\mathbb{R}_+,w_2dx)}\\
&\quad\lesssim\|b\|_{BMO_{\Delta_\lambda}(\mathbb{R}_+)}^m
[w_2]_{\widetilde{A}_{\infty,\lambda-\frac{1}{2}}}^m
[w_2]_{\widetilde{A}_{r,\lambda-\frac{1}{2}}^\mathcal{R}}^{r}
[w_1,w_2]_{\widetilde{A}_{p,\lambda-\frac{1}{2}}^\mathcal{R}}\|f\|_{L^{p,1}(\mathbb{R}_+,w_1dx)}.
\end{align*}
\end{proof}

\begin{proof}[Proof of Theorem \ref{quantitative two weight restricted weak type estimate for sparse operator of commutator 2}]
The result can be easily deduced from Theorems \ref{quantitative two weight restricted weak type estimate for sparse operator of commutator 1}, \ref{quantitative two weight restricted weak type estimate for sparse operator of commutator 11} and $[w_1,w_2]_{\widetilde{A}_{p,\lambda-\frac{1}{2}}^\mathcal{R}}\leq
\|w_1,w_2\|_{\widetilde{A}_{p,\lambda-\frac{1}{2},\Phi}^\mathcal{R}}$.
\end{proof}

\section{Applications}
In this section, we present some applications of our main results. To begin with, we recall the definitions of the Hankel convolution and the Hankel translation as presented in \cite{Ha,Hi}. For appropriate functions $f$ and $g$, the Hankel convolution is defined by
\begin{align*}
(f\sharp_\lambda g)(x):=\int_0^\infty f(y)\tau_x^\lambda(g)(y)y^{2\lambda}dy,~x\in(0,\infty),
\end{align*}
where $\tau_x^\lambda g$, denotes the Hankel translation of $g$, which is defined as follows
\begin{align*}
\tau_x^\lambda(g)(y):=c_\lambda\int_0^\pi g(\sqrt{x^2+y^2-2xy\cos\theta})\sin\theta^{2\lambda-1}d\theta,~x,y\in(0,\infty).
\end{align*}

Following \cite{YY}, the $Z^\lambda$ denote the space of all functions $\varphi\in C^2(0,\infty)$ that satisfy the following three conditions:
\begin{itemize}
\item [(1)] $|\varphi(x)|\lesssim(1+x^2)^{-\lambda-1},~x\in(0,\infty)$;
\item [(2)] $|\varphi'(x)|\lesssim x(1+x^2)^{-\lambda-2},~x\in(0,\infty)$;
\item [(3)] $|\varphi''(x)|\lesssim (1+x^2)^{-\lambda-2},~x\in(0,\infty)$.
\end{itemize}
We introduce the operators that will be studied. Let $\varphi\in Z^\lambda$. The maximal operator $\varphi_\ast^\lambda$ is defined by
\begin{align*}
\varphi_\ast^\lambda(f)(x):=\sup_{t>0}|f\sharp_\lambda\varphi_t(x)|,
\end{align*}
where $\varphi_t(x)=t^{-2\lambda-1}\varphi(x/t)$, $x>0$. The square function $g_\varphi^\lambda$ is given as follows:
\begin{align*}
g_\varphi^\lambda(f)(x):=\Big(\int_0^\infty|f\sharp_\lambda\varphi_t(x)|^2\frac{dt}{t}\Big)^{1/2},
~x\in(0,\infty).
\end{align*}
Finally, for $\rho>2$, the $\rho$-variation operator $V_\rho^\lambda(\{\varphi_t\}_{t>0})$ associated to the family $\{\varphi_t\}_{t>0}$ is given by
\begin{align*}
V_\rho^\lambda(\{\varphi_t\}_{t>0})(f)(x):=\underset{n\in\mathbb{N}}{\sup_{0<t_n<t_{n-1}<\cdots<t_1}}
\Big(\sum_{j=1}^{n-1}
|f\sharp_\lambda\varphi_{t_j}(x)-f\sharp_\lambda\varphi_{t_{j+1}}(x)|^\rho\Big)^{1/\rho}.
\end{align*}
Let $b\in L_{loc}^1(\mathbb{R}_+,d\mu)$ and $m\in\mathbb{Z}^+$. The higher order commutators of $R_{\Delta_\lambda}$, $\varphi_\ast^\lambda$, $g_\varphi^\lambda$ and $V_\rho^\lambda(\{\varphi_t\}_{t>0})$ are defined by
\begin{align*}
R_{\Delta_\lambda,b}^{m}(f)(x):=R_{\Delta_\lambda}((b(\cdot)-b(x))^mf)(x),
~x\in(0,\infty),
\end{align*}
\begin{align*}
\varphi_{\ast,b}^{\lambda,m}(f)(x):=\sup_{t>0}|(\varphi_t\sharp_\lambda[(b(\cdot)-b(x))^mf])(x)|,
~x\in(0,\infty),
\end{align*}
\begin{align*}
g_{\varphi,b}^{\lambda,m}(f)(x):=\Big(\int_0^\infty|(\varphi_t\sharp_\lambda[(b(\cdot)-b(x))^mf])(x)|^2
\frac{dt}{t}\Big)^{1/2},
~x\in(0,\infty),
\end{align*}
\begin{align*}
V_{\rho,b}^{\lambda,m}(\{\varphi_t\}_{t>0})(f)(x):&=
\underset{n\in\mathbb{N}}{\sup_{0<t_n<t_{n-1}<\cdots<t_1}}\Big(\sum_{j=1}^{n-1}
|(\varphi_{t_j}\sharp_\lambda[(b(\cdot)-b(x))^mf])(x)\\
&\qquad\quad\qquad\qquad\quad-
(\varphi_{t_{j+1}}\sharp_\lambda[(b(\cdot)-b(x))^mf])(x)|^\rho\Big)^{1/\rho}.
\end{align*}
respectively.

Maximal operators, square functions and variation operators have received considerable attention in probability theory, ergodic theory and harmonic analysis. It is well known that the maximal operators $\varphi_\ast^\lambda$ and $V_\rho^\lambda(\{\varphi_t\}_{t>0})$ are both connected with the convergence of $\{f\sharp_\lambda\varphi_t\}_{t>0}$, while the advantage of the latter is that it can directly derive pointwise convergence of $\{f\sharp_\lambda\varphi_t\}_{t>0}$, without requiring prior knowledge that pointwise convergence holds for a dense subclass of functions, which may not be satisfied in certain ergodic models, we refer readers to see \cite{Bou}. Additionally, the boundedness properties of the square functions $g_\varphi^\lambda$ leads directly to the boundedness of Bessel spectral multipliers.

Let $T$ represents to $\mathcal{R}_{\Delta_\lambda}$, $\varphi_\ast^\lambda$, $g_\varphi^\lambda$ or $V_\rho^\lambda(\{\varphi_t\}_{t>0})$. Let $T_b^m$ represents to $\mathcal{R}_{\Delta_\lambda,b}^m$, $\varphi_{\ast,b}^{\lambda,m}$, $g_{\varphi,b}^{\lambda,m}$ or $V_{\rho,b}^{\lambda,m}(\{\varphi_t\}_{t>0})$.  In \cite{AlBeFaRo,DuoGonKufLiWiYa}, the authors proved that there exist $\mathcal{T}$ dyadic systems $\mathcal{D}^j$, $j=1,2,\cdots,\mathcal{T}$, and $\eta$-sparse families $\mathcal{S}_j\subset\mathcal{D}^j$ such that,
\begin{align}\label{sparse domination for various operator}
|Tf|\lesssim\sum_{j=1}^{\mathcal{T}}\mathcal{A}_{\mathcal{S}_j}f,
\end{align}
\begin{align}\label{sparse domination for commutator of various operator}
&|T_b^mf(x)|\lesssim\sum_{j=1}^{\mathcal{T}}\sum_{k=0}^{m}C_m^k\sum_{I\in\mathcal{S}_j}
|b(x)-b_{I,\mu}|^{m-k}\\
&\qquad\qquad\qquad\qquad\times\Big(\frac{1}{\mu(I)}\int_I|b(z)-b_{I,\mu}|^k|f(z)|d\mu(z)\Big)\chi_I(x).
\nonumber
\end{align}

In the following, we present some applications of extrapolation theorems. As in the classical Laplace setting, we can also employ the sharp extrapolation theorem to reprove the following quantitative weighted estimate of Bessel Riesz transform $R_{\Delta_\lambda}$:
\begin{align}\label{quantitative weighted estimates of Bessel Riesz transform}
\|R_{\Delta_\lambda}f\|_{L^p(\mathbb{R}_+,wdx)}\lesssim
[w]_{\widetilde{A}_{p,\lambda-\frac{1}{2}}}^{\max\{1,\frac{1}{p-1}\}}\|f\|_{L^p(\mathbb{R}_+,wdx)},~w\in
\widetilde{A}_{p,\lambda-\frac{1}{2}},~1<p<\infty.
\end{align}
\begin{proof}
In \cite{LiLiangShenW}, the authors prove the above result by considering two cases: $p\geq2$ and $1<p<2$. Nevertheless, we only need to consider the case $p=2$ and each weight $w\in \widetilde{A}_{2,\lambda-\frac{1}{2}}$,
\begin{align}\label{p=2}
\|R_{\Delta_\lambda}f\|_{L^2(\mathbb{R}_+,wdx)}\lesssim
[w]_{\widetilde{A}_{2,\lambda-\frac{1}{2}}}\|f\|_{L^2(\mathbb{R}_+,wdx)},
\end{align}
which was already proved in \cite{LiLiangShenW}. The rest results can be derived by the sharp extrapolation theorem.

Indeed, by the sharp extrapolation Theorem \ref{Rubio de Francia extrapolation theorem for new weights} and Theorem \ref{mixed quantitative weighted strong estimates}, for $1<p<2$,
\begin{align}\label{p<2}
\|R_{\Delta_\lambda}f\|_{L^p(\mathbb{R}_+,wdx)}&\lesssim
\|\mathcal{M}_{\mu}\|_{\mathcal{B}(L^p(\mathbb{R}_+,wdx))}^{2-p}
[w]_{\widetilde{A}_{p,\lambda-\frac{1}{2}}}\|f\|_{L^p(\mathbb{R}_+,wdx)}\\
&\lesssim[w]_{\widetilde{A}_{p,\lambda-\frac{1}{2}}}^{\frac{2-p}{p-1}}
[w]_{\widetilde{A}_{p,\lambda-\frac{1}{2}}}\|f\|_{L^p(\mathbb{R}_+,wdx)}\nonumber\\
&=[w]_{\widetilde{A}_{p,\lambda-\frac{1}{2}}}^{\frac{1}{p-1}}\|f\|_{L^p(\mathbb{R}_+,wdx)}.\nonumber
\end{align}
Again by Theorem \ref{Rubio de Francia extrapolation theorem for new weights}, Theorem \ref{mixed quantitative weighted strong estimates} and
\begin{align*}
[v_\lambda^{p'}w^{1-p'}]_{\widetilde{A}_{p',\lambda-\frac{1}{2}}}=
[w]_{\widetilde{A}_{p,\lambda-\frac{1}{2}}}^{\frac{1}{p-1}},
\end{align*}
for $2<p<\infty$,
\begin{align}\label{p>2}
\|R_{\Delta_\lambda}f\|_{L^p(\mathbb{R}_+,wdx)}&\lesssim
\|\mathcal{M}_{\mu}\|_{\mathcal{B}(L^{p'}(\mathbb{R}_+,v_\lambda^{p'}w^{1-p'}dx))}^{\frac{p-2}{p-1}}
[w]_{\widetilde{A}_{p,\lambda-\frac{1}{2}}}^{\frac{1}{p-1}}\|f\|_{L^p(\mathbb{R}_+,wdx)}\\
&\lesssim[v_\lambda^{p'}w^{1-p'}]_{\widetilde{A}_{p',\lambda-\frac{1}{2}}}^{\frac{p-2}{(p'-1)(p-1)}}
[w]_{\widetilde{A}_{p,\lambda-\frac{1}{2}}}^{\frac{1}{p-1}}\|f\|_{L^p(\mathbb{R}_+,wdx)}\nonumber\\
&=[w]_{\widetilde{A}_{p,\lambda-\frac{1}{2}}}^{\frac{p-2}{p-1}+\frac{1}{p-1}}
\|f\|_{L^p(\mathbb{R}_+,wdx)}\nonumber\\
&=[w]_{\widetilde{A}_{p,\lambda-\frac{1}{2}}}\|f\|_{L^p(\mathbb{R}_+,wdx)}.\nonumber
\end{align}
Combined with \eqref{p=2}-\eqref{p>2}, we obtain \eqref{quantitative weighted estimates of Bessel Riesz transform}.
\end{proof}

Now we concentrate on some other important weighted inequalities in harmonic analysis. By Corollary \ref{vector-value extrapolation}, we can extend \eqref{scalar characterization new weight} to the vector-valued case.
\begin{theorem}\label{vector-value characterization of HL}
Let $w$ be a weight on $\mathbb{R}_+$ and $\lambda>-1/2,\neq0$. Then for all $1<p,q<\infty$,
\begin{align*}
w\in\widetilde{A}_{p,\lambda}\Leftrightarrow\Big\|\Big(\sum_{i}\mathcal{M}_\lambda f_i^q\Big)^{1/q}\Big\|_{L^p(\mathbb{R}_+,wdx)}\lesssim
\Big\|\Big(\sum_{i}|f_i|^q\Big)^{1/q}\Big\|_{L^p(\mathbb{R}_+,wdx)}.
\end{align*}
\end{theorem}
\begin{proof}
The necessity is a direct application of Corollary \ref{vector-value extrapolation} and \eqref{scalar characterization new weight}. Let $f_1=f$ and $f_i=0$, $i\neq1$. Then
\begin{align*}
\Big\|\Big(\sum_{i}\mathcal{M}_\lambda f_i^q\Big)^{1/q}\Big\|_{L^p(\mathbb{R}_+,wdx)}\lesssim
\Big\|\Big(\sum_{i}|f_i|^q\Big)^{1/q}\Big\|_{L^p(\mathbb{R}_+,wdx)}
\end{align*}
reduces to
\begin{align*}
\|\mathcal{M}_\lambda f\|_{L^p(\mathbb{R}_+,wdx)}\lesssim\|f\|_{L^p(\mathbb{R}_+,wdx)}.
\end{align*}
Hence, the sufficiency follows from \eqref{scalar characterization new weight}.
\end{proof}

Relied upon Theorem \ref{characterization A} and Corollary \ref{A extraapolaation}, we have the following Coifman-Fefferman inequalities for harmonic analysis operators associated to $\Delta_\lambda$.
\begin{theorem}\label{Coifman-Fefferman inequalities}
Let $T$ represents to $\mathcal{R}_{\Delta_\lambda}$, $\varphi_\ast^\lambda$, $g_\varphi^\lambda$ or $V_\rho^\lambda(\{\varphi_t\}_{t>0})$. Then for any $0<p<\infty$ and $w\in\widetilde{A}_{\infty,\lambda-\frac{1}{2}}$,
\begin{align*}
\|Tf\|_{L^p(\mathbb{R}_+,wdx)}\lesssim\|\mathcal{M}_\mu f\|_{L^p(\mathbb{R}_+,wdx)}.
\end{align*}
\end{theorem}
\begin{proof}
It suffices to show the conclusion holds for $\mathcal{A}_{\mathcal{S}}$ by \eqref{sparse domination for various operator}, where $\mathcal{S}$ is an $\eta$-sparse family. By the definition of sparse operator, there exists $E_I\subset I$ such that $\mu(E_I)\geq\eta\mu(I)$. Theorem \ref{characterization A} implies that $w(E_I)\gtrsim w(I)$. From this, we deduce that
\begin{align*}
&\int_{\mathbb{R}_+}|\mathcal{A}_{\mathcal{S}}f(x)|w(x)dx\\
&\quad=\sum_{I\in\mathcal{S}}\int_{\mathbb{R}_+}\frac{1}{\mu(I)}\int_I|f(x)|d\mu(x)\chi_I(x)w(x)dx\\
&\quad\leq\sum_{I\in\mathcal{S}}\inf_{x\in I}\mathcal{M}_{\mu}(f)(x)w(I)\\
&\quad\lesssim\sum_{I\in\mathcal{S}}\inf_{x\in I}\mathcal{M}_{\mu}(f)(x)w(E_I)\\
&\quad\lesssim\int_{\mathbb{R}_+}\mathcal{M}_{\mu}(f)(x)w(x)dx,
\end{align*}
this verifies the assumption of Corollary \ref{A extraapolaation} when $p_0=1$. Then the desired result follows by Corollary \ref{A extraapolaation}.
\end{proof}

From Theorems \ref{vector-value characterization of HL}, \ref{Coifman-Fefferman inequalities} and Corollary \ref{A extrapolaation vector-value}, we can also extend above theorem to the vector-valued case.
\begin{theorem}
Let $T$ represents to $\mathcal{R}_{\Delta_\lambda}$, $\varphi_\ast^\lambda$, $g_\varphi^\lambda$ or $V_\rho^\lambda(\{\varphi_t\}_{t>0})$. Then for any $0<p,q<\infty$ and $w\in\widetilde{A}_{\infty,\lambda-\frac{1}{2}}$,
\begin{align*}
\Big\|\big(\sum_{i}|Tf_i|^q\big)^{1/q}\Big\|_{L^p(\mathbb{R}_+,wdx)}\lesssim
\Big\|\big(\sum_i(\mathcal{M}_\mu f_i)^q\big)^{1/q}\Big\|_{L^p(\mathbb{R}_+,wdx)}.
\end{align*}
In particular, for all $1<p,q<\infty$ and $w\in \widetilde{A}_{p,\lambda-\frac{1}{2}}$,
\begin{align*}
\Big\|\big(\sum_{i}|Tf_i|^q\big)^{1/q}\Big\|_{L^p(\mathbb{R}_+,wdx)}\lesssim
\Big\|\big(\sum_i|f_i|^q\big)^{1/q}\Big\|_{L^p(\mathbb{R}_+,wdx)}.
\end{align*}
\end{theorem}

Finally, we present some applications of Theorems \ref{quantitative two weight restricted weak type of sparse operator}, \ref{quantitative two weight restricted weak type estimate for sparse operator of commutator 2}. The following theorem is a consequence of Theorem \ref{quantitative two weight restricted weak type of sparse operator} and \eqref{sparse domination for various operator}.
\begin{theorem}\label{restricted weak type estimates for various operators}
Let $1<p<\infty$, $(w_1,w_2)\in\widetilde{A}_{p,\lambda-\frac{1}{2}}^\mathcal{R}$, $w_2\in\widetilde{A}_{\infty,\lambda-\frac{1}{2}}$. Let $T$ represents to $\mathcal{R}_{\Delta_\lambda}$, $\varphi_\ast^\lambda$, $g_\varphi^\lambda$ or $V_\rho^\lambda(\{\varphi_t\}_{t>0})$. Then there exists $r\geq1$ such that
\begin{align*}
\|Tf\|_{L^{p,\infty}(\mathbb{R}_+,w_2dx)}\lesssim
[w_2]_{\widetilde{A}_{r,\lambda-\frac{1}{2}}^{\mathcal{R}}}^{r}
[w_1,w_2]_{\widetilde{A}_{p,\lambda-\frac{1}{2}}^\mathcal{R}}\|f\|_{L^{p,1}(\mathbb{R}_+,w_1dx)}.
\end{align*}
\end{theorem}

By Theorem \ref{quantitative two weight restricted weak type estimate for sparse operator of commutator 2}, we have the following quantitative restricted weighted weak type estimates for commutators of harmonic analysis operators in the Bessel setting.
\begin{theorem}\label{restricted weak type estimates for commutators of various operators}
Let $1<p<\infty$, $\Phi(t)=t\log(e+t)^m$ with $m\in\mathbb{Z}^+$ and $b\in BMO_{\Delta_\lambda}(\mathbb{R}_+)$. Let $T_b^m$ represents to $\mathcal{R}_{\Delta_\lambda,b}^m$, $\varphi_{\ast,b}^{\lambda,m}$, $g_{\varphi,b}^{\lambda,m}$ or $V_{\rho,b}^{\lambda,m}(\{\varphi_t\}_{t>0})$. If $w_2\in\widetilde{A}_{\infty,\lambda-\frac{1}{2}}$ and $(w_1,w_2)\in\widetilde{A}_{p,\lambda-\frac{1}{2},\Phi}^\mathcal{R}$. Then there exists $r\geq1$ such that
\begin{align*}
&\|T_{b}^mf\|_{L^{p,\infty}(\mathbb{R}_+,w_2dx)}\\
&\quad\lesssim\|b\|_{BMO_{\Delta_\lambda}(\mathbb{R}_+)}^m
[w_2]_{\widetilde{A}_{\infty,\lambda-\frac{1}{2}}}^m
[w_2]_{\widetilde{A}_{r,\lambda-\frac{1}{2}}^\mathcal{R}}^{r+\frac{r}{p}}
\|w_1,w_2\|_{\widetilde{A}_{p,\lambda-\frac{1}{2},\Phi}^\mathcal{R}}\|f\|_{L^{p,1}(\mathbb{R}_+,w_1dx)}.
\end{align*}
\end{theorem}
\begin{proof}
We borrow some ideas from \cite{CrMoTr}. Denote
\begin{align*}
\mathcal{A}_{\mathcal{S},b}^{m,k}f(x):=\sum_{k=0}^{m}\sum_{I\in\mathcal{S}}
|b(x)-b_{I,\mu}|^{m-k}\Big(\frac{1}{\mu(I)}\int_I|b(z)-b_{I,\mu}|^k|f(z)|d\mu(z)\Big)\chi_I(x).
\end{align*}
It suffices to prove the result for $\mathcal{A}_{\mathcal{S},b}^{m,k}f$ by \eqref{sparse domination for commutator of various operator}. By making a reduction, we only need to consider $\mathcal{A}_{\mathcal{S},b}^m$ and $(\mathcal{A}_{\mathcal{S},b}^m)^{\ast}$, which are defined in Theorem \ref{quantitative two weight restricted weak type estimate for sparse operator of commutator 2}.

In fact, fix a sparse family $\mathcal{S}$, let $x\in\mathbb{R}_+$ and $I\in\mathcal{S}$ such that $x\in I$. Then one has
\begin{align*}
&\sum_{k=0}^m|b(x)-b_{I,\mu}|^{m-k}\frac{1}{\mu(I)}\int_I|b(z)-b_{I,\mu}|^k|f(z)|d\mu(z)\\
&\quad\leq\frac{1}{\mu(I)}\int_I\Big(\sum_{k=0}^m\max\{|b(x)-b_{I,\mu}|,|b(z)-b_{I,\mu}|\}^m\Big)|f(z)|
d\mu(z)\\
&\quad\leq(m+1)\Big[\frac{|b(x)-b_{I,\mu}|^m}{\mu(I)}\int_I|f(z)|d\mu(z)+\frac{1}{\mu(I)}\int_I
|b(z)-b_{I,\mu}|^m|f(z)|d\mu(z)\Big].
\end{align*}
Thus,
\begin{align*}
\mathcal{A}_{\mathcal{S},b}^{m,k}f\leq(m+1)(\mathcal{A}_{\mathcal{S},b}^mf+
(\mathcal{A}_{\mathcal{S},b}^m)^{\ast}f).
\end{align*}
The desired results follows by Theorem \ref{quantitative two weight restricted weak type estimate for sparse operator of commutator 2}.
\end{proof}
\begin{remark}
We present some special operators, the weighted inequalities for these operators can be deduced from Theorems \ref{Coifman-Fefferman inequalities}-\ref{restricted weak type estimates for commutators of various operators}, one can see \cite{AlBeFaRo} for more details.\\
$(1)$ Let $P_t^\lambda f=P_t^\lambda\sharp_\lambda f$, where $\{P_t^\lambda\}_{t>0}$ is the Poisson semigroup associated to $\Delta_\lambda$ and $P^\lambda(x)=\frac{2\lambda\Gamma(\lambda)}{\sqrt{\pi}\Gamma(\lambda+1/2)}(1+x)^{-\lambda-1}$. Let $l\in \mathbb{N}$. Define the maximal operator $P_{\ast,l}^{\lambda}$, $g_{\{t^l\partial_t^lP_t^\lambda\}_{t>0}}^{\lambda}$, $V_\rho(\{t^l\partial_t^lP_t^\lambda\}_{t>0})$ by
\begin{align*}
P_{\ast,l}^{\lambda}(f)(x):=\sup_{t>0}|t^l\partial_t^lP_t^\lambda(f)(x)|;
\end{align*}
\begin{align*}
g_{\{t^l\partial_t^lP_t^\lambda\}_{t>0}}^{\lambda}(f)(x):=\Big(\int_0^\infty|t^l\partial_t^lP_t^\lambda (f)(x)|^2\frac{dt}{t}\Big)^{1/2};
\end{align*}
\begin{align*}
&V_\rho^\lambda(\{t^l\partial_t^lP_t^\lambda\}_{t>0})(f)(x)\\
&\quad:=
\underset{n\in\mathbb{N}}{\sup_{0<t_n<t_{n-1}<\cdots<t_1}}
\Big(\sum_{j=1}^{n-1}
|f\sharp_\lambda t_{j}^l\partial_t^lP_{t_{j}}^\lambda(x)-f\sharp_\lambda t_{j+1}^l\partial_t^lP_{t_{j+1}}^\lambda(x)|^\rho\Big)^{1/\rho},
\end{align*}
respectively. Then $P_{\ast,l}^{\lambda}$, $g_{\{t^l\partial_t^lP_t^\lambda\}_{t>0}}^{\lambda}$, $V_\rho(\{t^l\partial_t^lP_t^\lambda\}_{t>0})$ also enjoy Theorems \ref{Coifman-Fefferman inequalities}-\ref{restricted weak type estimates for various operators}, the commutators of $P_{\ast,l}^{\lambda}$, $g_{\{t^l\partial_t^lP_t^\lambda\}_{t>0}}^{\lambda}$, $V_\rho(\{t^l\partial_t^lP_t^\lambda\}_{t>0})$ enjoy Theorem \ref{restricted weak type estimates for commutators of various operators}.\\
$(2)$ Let $\varphi\in\mathcal{S}(\mathbb{R}_+)$, where $\mathcal{S}(\mathbb{R}_+)$ is the class of Schwartz function on $\mathbb{R}_+$. Then $\varphi_\ast^\lambda$, $g_\varphi^\lambda$ with Hankel transform
\begin{align}\label{Hankel transform}
h_\lambda(\varphi)(x)=\int_0^\infty (xy)^{-\lambda+\frac{1}{2}}J_{\lambda-\frac{1}{2}}(xy)\varphi(y)y^{2\lambda}dy=0,
\end{align}
where $J_s$ represents the Bessel function of first kind and order $s>-1$, $V_\rho^\lambda(\{\varphi_t\}_{t>0})$ enjoy Theorems \ref{Coifman-Fefferman inequalities}-\ref{restricted weak type estimates for various operators}, the commutators of $\varphi_\ast^\lambda$, $g_\varphi^\lambda$ and $V_\rho^\lambda(\{\varphi_t\}_{t>0})$ enjoy Theorem \ref{restricted weak type estimates for commutators of various operators}. The $\{t^l\partial_t^lW_t^\lambda\}_{t>0}$, $l\in\mathbb{N}$, where $\{W_t^\lambda\}_{t>0}$ is the heat semigroup generated by $-\Delta_\lambda$, is a special case of $\{\varphi_t\}_{t>0}$.\\
$(3)$ Let $\alpha,t>0$,
\begin{align*}
B_t^{\lambda,\alpha}(f)=h_\lambda\Big(\Big(1-\frac{y^2}{t^2}\Big)_+^\alpha h_\lambda(f)\Big)
\end{align*}
be the Bochner-Riesz multiplier associated to $\Delta_\lambda$, where $h_\lambda$ is defined as \eqref{Hankel transform}, $z_+:=\max(0,z)$, $z\in\mathbb{R}$. By
\begin{align*}
h_\lambda\big(2^\alpha\Gamma(\alpha+1)y^{-\alpha-\lambda-1/2}J_{\alpha+\lambda+1/2}(y)\big)(x)=
(1-x^2)_+^\alpha,~x\in(0,\infty),
\end{align*}
see \cite[Page 48]{ErMaObTr}, we can rewrite
\begin{align*}
B_t^{\lambda,\alpha}(f)=\varphi_{1/t}^{\lambda,\alpha}\sharp_\lambda f,
\end{align*}
where $\varphi^{\lambda,\alpha}(x)=2^\alpha\Gamma(\alpha+1)x^{-\alpha-\lambda-1/2}J_{\alpha+\lambda+1/2}(x)$.
Then the maximal operator
$$B_{\ast}^{\lambda,\alpha}f:=\sup_{t>0}|B_{t}^{\lambda,\alpha}(f)|,~\alpha\geq \lambda+1,$$  $V_\rho(\{B_{t}^{\lambda,\alpha}\}_{t>0})$ with $\alpha\geq \lambda+3$, and the Stein-square function
\begin{align*}
G^{\lambda,\alpha}(f)(x):=\Big(\int_0^\infty|t\partial_tB_{t}^{\lambda,\alpha}(f)(x)|^2\frac{dt}{t}\Big)
^{\frac{1}{2}}=\Big(\int_0^\infty|\Psi_{1/t}\sharp_\lambda f(x)|^2\frac{dt}{t}\Big)^{\frac{1}{2}},~\alpha\geq \lambda+3,
\end{align*}
where $\Psi(z):=\varphi^{\lambda,\alpha}(z)+z\frac{d}{dz}\varphi^{\lambda,\alpha}(z)$, $z\in\mathbb{R}_+$, enjoy Theorems \ref{Coifman-Fefferman inequalities}-\ref{restricted weak type estimates for various operators}. The commutators of $B_{\ast}^{\lambda,\alpha}$ with $\alpha\geq \lambda+1$, $V_\rho(\{B_{t}^{\lambda,\alpha}\}_{t>0})$ with $\alpha\geq \lambda+3$, $G^{\lambda,\alpha}$ with $\alpha\geq \lambda+3$ enjoy Theorem \ref{restricted weak type estimates for commutators of various operators}.
\end{remark}

\subsection*{Declaration of competing interest}
The authors declare that there is no competing interest.

%%%%%%%%%%%% References %%%%%%%%%%%%%
%%
%<Author name> is written as Initial of Given Name, and Family Name.
%<Title> is written in roman letters.
%<Journal name> should be abbreviated according to
% the MR Serials Abbreviations List of Mathematical Reviews:
% (Abbreviations of Names of Serials; http://www.ams.org/mr-database)
%For <Pages>, use en-dash "--" between page numbers.
%%

\end{document}